\documentclass{article}
\usepackage{alphabeta,amsfonts,amsmath,amssymb,amsthm,authblk}
\usepackage[english]{babel}
\usepackage{booktabs}
\usepackage[makeroom]{cancel}
\usepackage{comment}
\usepackage{enumerate}
\usepackage[shortlabels]{enumitem}
\usepackage[margin = 1 in]{geometry}
\usepackage{graphicx}
\usepackage[usenames,dvipsnames]{xcolor}
\usepackage[linktocpage,colorlinks=true,linkcolor=Fuchsia!70!Magenta,urlcolor=blue!60!cyan,citecolor=Green!70!PineGreen,
	hypertexnames=false]{hyperref}
\usepackage[utf8]{inputenc}
\usepackage{leftidx}
\usepackage{mathrsfs,mathtools}
\usepackage{relsize}
\usepackage{scalerel,stackengine}
\usepackage{tikz,tikz-cd}
\usepackage[titles]{tocloft}
\usepackage{xspace}

\newcommand{\N}{\mathbb{N}}

\newcommand{\R}{\mathbb{R}}
\newcommand{\C}{\mathbb{C}}

\newcommand{\cA}{\mathcal{A}}

\newcommand{\cB}{\mathcal{B}}

\newcommand{\cC}{\mathcal{C}}
\newcommand{\cF}{\mathcal{F}}
\newcommand{\sF}{\mathscr{F}}

\newcommand{\sH}{\mathscr{H}}
\newcommand{\cM}{\mathcal{M}}

\newcommand{\e}{\varepsilon}
\newcommand{\Om}{\Omega} 
\newcommand{\om}{\omega}
\newcommand{\blambda}{\boldsymbol{\lambda}}

\DeclareMathOperator{\ev}{ev}
\DeclareMathOperator{\Hom}{Hom}
\DeclareMathOperator{\End}{End}
\newcommand{\MnC}{\mathrm{M}_N(\C)}

\newcommand{\Tr}{\mathrm{Tr}}
\newcommand{\tr}{\mathrm{tr}}

\DeclareMathOperator{\id}{id}

\DeclareMathOperator{\cRe}{Re}
\DeclareMathOperator{\cIm}{Im}

\newcommand{\flip}{\mathsmaller{\mathrm{flip}}}
\newcommand{\la}{\langle}
\newcommand{\ra}{\rangle}
\newcommand{\sh}{\text{\smaller $\#$}}

\newcommand{\op}{\mathrm{op}}

\newcommand{\sa}{\mathrm{sa}}
\newcommand{\vertiii}[1]{{\left\vert\kern-0.25ex\left\vert\kern-0.25ex\left\vert #1 
    \right\vert\kern-0.25ex\right\vert\kern-0.25ex\right\vert}}
\DeclareMathOperator{\supp}{supp}

\newcommand{\loc}{\mathrm{loc}}
\newcommand{\hotimes}{\otimes_2}

\newcommand{\potimes}{\hat{\otimes}_{\pi}}
\newcommand{\wotimes}{\bar{\otimes}}
\newcommand{\iotimes}{\hat{\otimes}_i}

\newcommand{\wh}{\widehat}

\newcommand{\SB}{\mathbb{S}}
\newcommand{\SBa}{\mathbb{S}_{\mathrm{a}}}

\newcommand\numberthis{\addtocounter{equation}{1}\tag{\theequation}}
\makeatletter
\newcommand{\oset}[3][0ex]{%
  \mathrel{\mathop{#3}\limits^{
    \vbox to#1{\kern-2\ex@
    \hbox{$\scriptstyle#2$}\vss}}}}
\makeatother

\theoremstyle{plain}
\newtheorem{prop}{Proposition}[subsection]
\newtheorem{propa}{Proposition}[section]

\newtheorem{lem}[prop]{Lemma}
\newtheorem{lema}[propa]{Lemma}

\newtheorem{thm}[prop]{Theorem}
\newtheorem{thma}[propa]{Theorem}

\newtheorem{pthm}[prop]{Pseudotheorem}

\newtheorem{cor}[prop]{Corollary}
\newtheorem{cora}[propa]{Corollary}

\theoremstyle{definition}
\newtheorem{defi}[prop]{Definition}
\newtheorem{defia}[propa]{Definition}

\newtheorem{nota}[prop]{Notation}
\newtheorem{notaa}[propa]{Notation}

\newtheorem{ex}[prop]{Example}

\newtheorem{rem}[prop]{Remark}
\newtheorem{rema}[propa]{Remark}

\newtheorem*{ack}{Acknowledgements}

\makeatletter
\renewenvironment{proof}[1][\proofname]{%
  \par\pushQED{\qed}\normalfont%
  \topsep6\p@\@plus6\p@\relax
  \trivlist\item[\hskip\labelsep\bfseries#1\@addpunct{.}]%
  \ignorespaces
}{%
  \popQED\endtrivlist\@endpefalse
}
\makeatother
\begin{document}

\title{It\^{o}'s formula for noncommutative \\ \texorpdfstring{$C^2$}{} functions of free It\^{o} processes}
\author{Evangelos A. Nikitopoulos\thanks{Supported by NSF grant DGE 2038238 and partially supported by NSF grants DMS 1253402 and DMS 1800733}}
\affil{Department of Mathematics, University of California San Diego\protect\\
\noindent 9500 Gilman Drive, La Jolla, CA 92093-0112 (USA)\protect\\
Email: {\tt \href{mailto:enikitop@ucsd.edu}{enikitop@ucsd.edu}}}
\date{\vspace{-6ex}}

\maketitle

\begin{abstract}
In a recent paper, the author introduced a rich class $NC^k(\mathbb{R})$ of ``noncommutative $C^k$" functions $\mathbb{R} \to \mathbb{C}$ whose operator functional calculus is $k$-times differentiable and has derivatives expressible in terms of multiple operator integrals (MOIs).
In the present paper, we explore a connection between free stochastic calculus and the theory of MOIs by proving an It\^{o} formula for noncommutative $C^2$ functions of self-adjoint free It\^{o} processes.
To do this, we first extend P. Biane and R. Speicher's theory of free stochastic calculus --- including their free It\^{o} formula for polynomials --- to allow free It\^{o} processes driven by multidimensional semicircular Brownian motions.
Then, in the self-adjoint case, we reinterpret the objects appearing in the free It\^{o} formula for polynomials in terms of MOIs.
This allows us to enlarge the class of functions for which one can formulate and prove a free It\^{o} formula from the space originally considered by Biane and Speicher (Fourier transforms of complex measures with two finite moments) to the strictly larger space $NC^2(\mathbb{R})$.
Along the way, we also obtain a useful ``traced" It\^{o} formula for arbitrary $C^2$ scalar functions of self-adjoint free It\^{o} processes.
Finally, as motivation, we study an It\^{o} formula for $C^2$ scalar functions of $N \times N$ Hermitian matrix It\^{o} processes.

$\,$\vspace{-1.35mm}


$\,$\vspace{-1.35mm}

\end{abstract}\vspace{-1.5mm}
\tableofcontents
\clearpage

\section{Introduction}\label{sec.intro}

\subsection{Motivation}

P. Biane and R. Speicher developed in \cite{bianespeicher} a theory of free stochastic calculus with respect to semicircular Brownian motion that has yielded many fruitful applications --- e.g., to the theories of free SDEs \cite{capitaine,demni,gao,kargin}, entropy \cite{bianespeicherF}, and transport \cite{dabrowskietal};
analysis on Wigner space \cite{bianespeicher,kemp4th};
and the calculation of Brown measures \cite{driverhallkemp,hozhong,demnihamdi,hallho}.
In this paper, we present an extension and reinterpretation of this free stochastic calculus that naturally connects the It\^{o}-type formulas thereof to the theory of \textit{multiple operator integrals} (MOIs) via the class $NC^k(\R)$ of \textit{noncommutative $C^k$ functions} (Definition \ref{def.NCk}) introduced by the author in \cite{nikitopoulosNCk}.

Our main results (Theorems \ref{thm.trFFIF} and \ref{thm.FFIF}) are ``free It\^{o} formulas" for scalar functions of self-adjoint ``free It\^{o} processes" with respect to an $n$-dimensional semicircular Brownian motion $(x_1,\ldots,x_n)$.
As a consequence of the work of D.-V. Voiculescu \cite{voiculescu}, $(x_1,\ldots,x_n)$ is in a precise sense the large-$N$ limit of an $n$-tuple $\big(X_1^{(N)},\ldots,X_n^{(N)}\big)$ of independent Brownian motions on the space of $N \times N$ Hermitian matrices.
Therefore, interesting formulas involving $(x_1,\ldots,x_n)$ are often best motivated by studying formulas involving $\big(X_1^{(N)},\ldots,X_n^{(N)}\big)$ and then (formally or rigorously) taking $N \to \infty$.
This is certainly true for our formulas.
In Appendix \ref{sec.fdmotiv}, we study some independently interesting matrix stochastic calculus formulas that motivate the present paper's main results.
In order to explain the appearance of MOIs, we discuss a special case of one of these formulas.

In this preliminary discussion and in Appendix \ref{sec.fdmotiv}, we assume familiarity with the theory of continuous-time stochastic processes and stochastic integration, though these subjects are not used elsewhere in the paper.
Please see \cite{chungwilliams,karatzasshreve} for some relevant background. Fix a filtered probability space $(\Om,\sF,(\sF_t)_{t \geq 0},P)$, with filtration satisfying the usual conditions, to which all processes we discuss will be adapted.

We begin by recalling the statement of It\^{o}'s formula from classical stochastic analysis.
Let $V$ and $W$ be finite-dimensional inner product spaces, and let $M = (M(t))_{t \geq 0}$ be a continuous $V$-valued semimartingale.
It\^{o}'s formula says that if $F \in C^2(V;W)$, then
\[
d\,F(M(t)) = DF(M(t))[dM(t)] + \frac{1}{2}D^2F(M(t))[dM(t), dM(t)],\numberthis\label{eq.Itoform}
\]
where $D^kF$ is the $k^{\text{th}}$ Fr\'{e}chet derivative of $F$.
The $DF(M)[dM]$ term in Equation \eqref{eq.Itoform} is the differential notation for the stochastic integral against $M$ of the $\Hom(V;W) = \{$linear maps $V \to W\}$-valued process $DF(M)$.
The notation for the second term (the ``It\^{o} correction term") in Equation \eqref{eq.Itoform} is to be understood as follows.
Let $e_1,\ldots,e_n \in V$ be a basis for $V$, and write $M = \sum_{i=1}^nM_i e_i$.
Then
\[
\int_0^t D^2F(M(s))[dM(s),dM(s)] = \sum_{i,j=1}^n\int_0^t \underbrace{D^2F(M(s))[e_i,e_j]}_{\partial_{e_i}\partial_{e_j}F(M(s))} \, dM_i(s)\,dM_j(s),
\]
where $dM_i(s)\,dM_j(s) = d\la\hspace{-0.3mm}\la M_i,M_j\ra\hspace{-0.3mm}\ra(s)$ denotes Riemann--Stieltjes integration against the quadratic covariation $\la\hspace{-0.3mm}\la M_i,M_j\ra\hspace{-0.3mm}\ra$ of $M_i$ and $M_j$.
Our present motivation is an application of Equation \eqref{eq.Itoform} to matrix-valued processes $M$ and maps $F$ arising from scalar functional calculus.

\begin{nota}\label{nota.mat}
Fix $N \in \N$.
\begin{enumerate}[label=(\alph*),leftmargin=2\parindent]
    \item Write $\MnC$ for the set of $N \times N$ complex matrices and $\MnC_{\sa}$ for the set of $M \in \MnC$ such that $M^*=M$.
    \item Write $\la A,B \ra_N \coloneqq N \, \Tr(B^*A) = N^2\,\tr(B^*A)$ for all $A,B \in \MnC$, where $\tr = \frac{1}{N}\Tr$ is the normalized trace.
    Note that $\la \cdot,\cdot \ra_N$ restricts to a real inner product on the real vector space $\MnC_{\sa}$.
    \item For $M \in \MnC_{\sa}$ and $\lambda \in \sigma(M) = \{$eigenvalues of $M\} \subseteq \R$, write $P_{\lambda}^M \in \MnC$ for (the standard representation of) the orthogonal projection onto the $\lambda$-eigenspace of $M$.
    For a function $f \colon \sigma(M) \to \C$, write $f(M) \coloneqq \sum_{\lambda \in \sigma(M)} f(\lambda)\,P_{\lambda}^M \in \MnC$.
    Recall that the Spectral Theorem for Hermitian matrices says precisely that if $M \in \MnC_{\sa}$, then $M = \sum_{\lambda \in \sigma(M)}\lambda\,P_{\lambda}^M = \id_{\sigma(M)}(M)$.
\end{enumerate}
\end{nota}

Now, let $\big(X_1^{(N)},\ldots,X_n^{(N)}\big) = (X_1,\ldots,X_n)$ be an $n$-tuple of independent standard $(\MnC_{\sa},\la \cdot, \cdot \ra_N)$-valued Brownian motions, and let $M$ be a $\MnC$-valued stochastic process satisfying
\[
dM(t) = \sum_{i=1}^n\sum_{j=1}^{\ell} A_{ij}(t)\,dX_i(t)\,B_{ij}(t) + K(t)\,dt \numberthis\label{eq.specialmatrItoprocess}
\]
for continuous adapted $\MnC$-valued processes $A_{ij}$, $B_{ij}$, $K$.
The term $A_{ij}(t)\,dX_i(t)\,B_{ij}(t)$ above is the differential notation for the stochastic integral against $X_i$ of the $\End(\MnC)= \Hom(\MnC;\MnC)$-valued process $[0,\infty)  \times \Om \ni (t,\om) \mapsto (\MnC \ni E \mapsto A_{ij}(t,\om)\,E\,B_{ij}(t,\om) \in \MnC)$.
Such processes $M$ are special kinds of $N \times N$ \textit{matrix It\^{o} processes} (Definition \ref{def.matItoprocess}).
Finally, for $f \in C^2(\R)$, define
\[
f^{[1]}(\lambda,\mu) \coloneqq \frac{f(\lambda)-f(\mu)}{\lambda-\mu} \; \text{ and } \; f^{[2]}(\lambda,\mu,\nu) \coloneqq \frac{f^{[1]}(\lambda,\mu) - f^{[1]}(\lambda,\nu)}{\mu - \nu}
\]
to be the \textit{first} and \textit{second divided differences} (Definition \ref{def.divdiff} and Proposition \ref{prop.divdiff}) of $f$, respectively.

\begin{thm}\label{thm.specialFIF}
If $M$ is as in Equation \eqref{eq.specialmatrItoprocess}, $M^*=M$, and $f \in C^2(\R)$, then
\[
d\,f(M(t)) = \sum_{\lambda,\mu \in \sigma(M(t))}f^{[1]}(\lambda,\mu) \,P_{\lambda}^{M(t)}\,dM(t)\,P_{\mu}^{M(t)} + \sum_{i=1}^n C_i(t)\,dt,
\]
where the process $C_i$ above is given by
\[
C_i = \sum_{j,k=1}^{\ell}\sum_{\lambda,\mu,\nu \in \sigma(M)}f^{[2]}(\lambda,\mu,\nu)\big(P_{\lambda}^MA_{ij}\,\tr(B_{ij}P_{\mu}^M A_{ik})\,B_{ik}P_{\nu}^M + P_{\lambda}^MA_{ik}\,\tr(B_{ik}P_{\mu}^M A_{ij})\,B_{ij}P_{\nu}^M\big).
\]
\end{thm}
\begin{rem}
This is the special case of Theorem \ref{thm.FIF} with $U_i = \sum_{j=1}^{\ell} A_{ij}\otimes B_{ij}$.
\end{rem}

This result is proven from It\^{o}'s formula using the quadratic covariation rules (Theorem \ref{thm.IMPR})
\begin{align}
    & A(t)\,dX_i(t)\,B(t)\,dX_j(t)\,C(t) = \delta_{ij} \,A(t)\,\tr(B(t))\,C(t)\,dt \; \text{ and} \label{eq.dXidXj1}\\
    & A(t)\,dX_i(t)\,B(t)\,dt\,C(t) = A(t)\,dt\,B(t)\,dX_i(t)\,C(t) = A(t)\,dt\,B(t)\,dt\,C(t) = 0\label{eq.dXidXj2}
\end{align}
and a result (Theorem \ref{thm.matfunccalcder}) of Yu. L. Daletskii and S. G. Krein \cite{daletskiikrein} saying that if $f \in C^2(\R)$ and $f_{\MnC}$ denotes the map $\MnC_{\sa} \ni M \mapsto f(M) \in \MnC$, then $f_{\MnC} \in C^2(\MnC_{\sa};\MnC)$ with
\begin{align}
    D f_{\MnC}(M)[B] & = \sum_{\lambda,\mu \in \sigma(M)}f^{[1]}(\lambda,\mu) \, P_{\lambda}^MBP_{\mu}^M \; \text{ and} \label{eq.matder1}\\
    D^2f_{\MnC}(M)[B_1,B_2] & = \sum_{\lambda,\mu,\nu \in \sigma(M)} f^{[2]}(\lambda,\mu,\nu) \, \big(P_{\lambda}^MB_1P_{\mu}^MB_2P_{\nu}^M + P_{\lambda}^MB_2P_{\mu}^MB_1P_{\nu}^M\big),\label{eq.matder2}
\end{align}
for all $M,B,B_1,B_2 \in \MnC_{\sa}$.
One of the main results of this paper is the formal large-$N$ limit of (a generalization \hspace{-0.25mm}of)\hspace{-0.25mm} Theorem \hspace{-0.25mm}\ref{thm.specialFIF}\hspace{-0.25mm} that \hspace{-0.25mm}arises\hspace{-0.25mm} --- \hspace{-0.25mm}at\hspace{-0.25mm} least \hspace{-0.25mm}heuristically\hspace{-0.25mm} --- \hspace{-0.25mm}by\hspace{-0.25mm} taking $N \hspace{-0.25mm}\to\hspace{-0.25mm} \infty$ in \hspace{-0.25mm}Equations\hspace{-0.25mm} \eqref{eq.dXidXj1}--\eqref{eq.matder2}.

Loosely speaking, Voiculescu's results from \cite{voiculescu} imply that there is an operator algebra $\cA$ with a ``trace" $\tau \colon \cA \to \C$ and ``freely independent processes" $x_1,\ldots,x_n \colon [0,\infty) = \R_+ \to \cA_{\sa} = \{a \in \cA : a^*=a\}$ called \textit{semicircular Brownian motions} such that
\[
\tr\Big(P\Big(X_{i_1}^{(N)}(t_1),\ldots,X_{i_r}^{(N)}(t_r)\Big)\Big) \to \tau(P(x_{i_1}(t_1),\ldots,x_{i_r}(t_r)))
\]
almost surely (and in expectation) as $N \to \infty$, for all $i_1,\ldots,i_r \in \{1,\ldots,n\}$, times $t_1,\ldots,t_r \geq 0$, and polynomials $P$ in $r$ noncommuting indeterminates.
Now, using Biane and Speicher's work from \cite{bianespeicher}, one can make sense of stochastic differentials $a(t)\,dx_i(t)\,b(t)$ when $a,b \colon \R_+ \to \cA$ are ``continuous adapted processes."
Imagining then a situation in which
\[
(A,B,C) = \big(A^{(N)},B^{(N)},C^{(N)}\big) \,\text{``}\to\text{"} \,(a,b,c)
\]
as $N \to \infty$, we might expect to be able to take $N \to \infty$ in Equations \eqref{eq.dXidXj1}--\eqref{eq.dXidXj2} and thus to get quadratic covariation rules
\begin{align}
    & a(t)\,dx_i(t)\,b(t)\,dx_j(t)\,c(t) = \delta_{ij} \,a(t)\,\tau(b(t))\,c(t)\,dt \; \text{ and} \label{eq.dxidxj1}\\
    & a(t)\,dx_i(t)\,b(t)\,dt\,c(t) = a(t)\,dt\,b(t)\,dx_i(t)\,c(t) = a(t)\,dt\,b(t)\,dt\,c(t) = 0.\label{eq.dxidxj2}\displaybreak
\end{align}
Interpreted appropriately, these rules do in fact hold (Theorem \ref{thm.FIPR}).
How about Equations \eqref{eq.matder1}--\eqref{eq.matder2}?
In this operator algebraic setting, we would be working with the map $f_{\mathsmaller{\cA}} \colon \cA_{\sa} \to \cA$ defined via functional calculus by $\cA_{\sa} \ni m \mapsto f(m) = \int_{\sigma(m)} f\,dP^m \in \cA$, where $P^m$ is the projection-valued spectral measure of $m$ (Section \ref{sec.freeprob}).
Therefore, it would be appropriate to guess that one could replace the sums $\sum_{\lambda \in \sigma(M)} \boldsymbol{\cdot}\,P_{\lambda}^M$ in Equations \eqref{eq.matder1}--\eqref{eq.matder2} with integrals $\int_{\sigma(m)} \boldsymbol{\cdot} \, dP^m$.
Explicitly, we might expect that if $f \in C^2(\R)$, then $f_{\mathsmaller{\cA}} \in C^2(\cA_{\sa};\cA)$ and
\begin{align*}
    D f_{\mathsmaller{\cA}}(m)[b] & = \int_{\sigma(m)}\hspace{-0.3mm}\int_{\sigma(m)} f^{[1]}(\lambda,\mu) \, P^m(d\lambda) \, b\, P^m(d\mu) \; \text{ and} \numberthis\label{eq.opder1}\\
    D^2 f_{\mathsmaller{\cA}}(m)[b_1,b_2] & = \int_{\sigma(m)}\hspace{-0.3mm}\int_{\sigma(m)}\hspace{-0.3mm}\int_{\sigma(m)} f^{[2]}(\lambda,\mu,\nu) \, \big(P^m(d\lambda)\,b_1\,P^m(d\mu)\,b_2\,P^m(d\nu) \\
    & \hspace{57.5mm} + P^m(d\lambda)\,b_2\,P^m(d\mu)\,b_1\,P^m(d\nu)\big),\numberthis\label{eq.opder2}
\end{align*}
for all $m,b,b_1,b_2 \in \cA_{\sa}$.
(These integrals do not actually make sense with standard projection-valued measure theory.
We shall ignore this subtlety for now.)
Finally, consider a process $m \colon \R_+ \to \cA$ satisfying
\[
dm(t) = \sum_{i=1}^n\sum_{j=1}^{\ell}a_{ij}(t)\,dx_i(t)\,b_{ij}(t) + k(t)\,dt \numberthis\label{eq.specialfrItoprocess}
\]
for some continuous adapted processes $a_{ij},b_{ij},k \colon \R_+ \to \cA$.
Such processes $m$ are special kinds of \textit{free It\^{o}\hspace{-0.2mm} processes} \hspace{-0.2mm}(Definition\hspace{-0.2mm} \ref{def.freeItoprocess}).\hspace{-0.2mm}
Formally \hspace{-0.2mm}combining\hspace{-0.2mm} Equations \hspace{-0.2mm}\eqref{eq.dxidxj1}--\eqref{eq.opder2}\hspace{-0.2mm} and \hspace{-0.2mm}applying\hspace{-0.2mm} the \hspace{-0.2mm}hypothetical\hspace{-0.2mm} It\^{o} \hspace{-0.2mm}formula
\[
d\,f_{\mathsmaller{\cA}}(m(t)) =Df_{\mathsmaller{\cA}}(m(t))[dm(t)]+\frac{1}{2}D^2f_{\mathsmaller{\cA}}(m(t))[dm(t),dm(t)]
\]
then gives the following guess.

\begin{pthm}\label{pthm.specialFFIF}
If $m$ is as in Equation \eqref{eq.specialfrItoprocess}, $m^*=m$, and $f \in C^2(\R)$, then
\[
d\,f(m(t)) = \int_{\sigma(m(t))}\hspace{-0.5mm}\int_{\sigma(m(t))}f^{[1]}(\lambda,\mu) \,P^{m(t)}(d\lambda)\,dm(t)\,P^{m(t)}(d\mu) + \sum_{i=1}^n c_i(t)\,dt,
\]
where
\begin{align*}
    c_i & = \sum_{j,k=1}^{\ell}\int_{\sigma(m)}\hspace{-0.5mm}\int_{\sigma(m)}\hspace{-0.5mm}\int_{\sigma(m)}f^{[2]}(\lambda,\mu,\nu)\,\big(P^m(d\lambda)\,a_{ij}\tau(b_{ij}\,P^m(d\mu)\, a_{ik})\,b_{ik}\,P^m(d\nu) \\
    & \hspace{65mm} + P^m(d\lambda)\,a_{ik}\tau(b_{ik}\,P^m(d\mu) \,a_{ij})\,b_{ij}\,P^m(d\nu)\big).
\end{align*}
\end{pthm}

As we hinted above, the integrals in Equations \eqref{eq.opder1}--\eqref{eq.opder2} and the pseudotheorem above are purely formal:
\textit{a priori}, it doesn't make sense to integrate operator-valued functions against projection-valued measures.
In fact, this is precisely the (nontrivial) problem multiple operator integrals (MOIs) were invented to solve.
However, even with the realization that a MOI is the right object to consider when interpreting Pseudotheorem \ref{pthm.specialFFIF}, the relevant MOIs do not necessarily make sense for arbitrary $f \in C^2(\R)$.
This is where noncommutative $C^2$ functions come in.
The space $NC^2(\R) \subseteq C^2(\R)$ is essentially tailor-made to ensure that MOI expressions such as the ones above make sense and are well-behaved.
(For example, the derivative formulas \eqref{eq.opder1}--\eqref{eq.opder2} are proven rigorously in \cite{nikitopoulosNCk} for $f \in NC^2(\R)$.)
The result is that we are able to turn Pseudotheorem \ref{pthm.specialFFIF} into a (special case of a) rigorous statement --- Theorem \ref{thm.FFIF} --- if we take $f \in NC^2(\R)$.
Moreover, we demonstrate in Example \ref{ex.BS} that Theorem \ref{thm.FFIF} generalizes and conceptually clarifies Proposition 4.3.4 in \cite{bianespeicher}, Biane and Speicher's free It\^{o} formula for a certain class (strictly smaller than $NC^2(\R)$) of scalar functions.

\subsection{Summary and Guide to Reading}\label{sec.guide}

In this section, we describe the structure of the paper and summarize our results.
All the results in the paper are proven both for $n$-dimensional semicircular Brownian motions and $n$-dimensional circular Brownian motions.
To ease the present exposition, we summarize only the statements in the semicircular case.

Section \ref{sec.freeprob} contains a review of some terminology and relevant results from free probability theory, for example the concepts of filtered $W^*$-probability spaces and (semi)circular Brownian motions.
Section \ref{sec.tensprod} contains a review of the various topological tensor products of which we shall make use, for example the von Neumann algebra tensor product $\wotimes$.

In Section \ref{sec.integhasdx}, we review Biane and Speicher's construction from \cite{bianespeicher} of the free stochastic integral of certain ``biprocesses" against semicircular Brownian motion.
More specifically, if $(\cA,(\cA_t)_{t \geq 0},\tau)$ is a filtered $W^*$-probability space and $x \colon \R_+ \to \cA_{\sa}$ is a semicircular Brownian motion, then $\int_0^t u(s)\sh dx(s) \in \cA_t$ is defined for certain maps $u \colon \R_+ \to \cA \wotimes \cA^{\op}$, where $\cA^{\op}$ is the opposite of $\cA$.
The $\#$ stands for the operation determined by $(a \otimes b)\sh c = acb$, and the free stochastic integral $\int_0^t u(s)\sh dx(s)$ is determined in an appropriate sense by $\int_0^t (1_{[r_1,r_2)} a \otimes b)(s)\sh dx(s) = (a \otimes b)\sh[x(r_1 \wedge t) - x(r_2 \wedge t)] = a(x(r_1 \wedge t) - x(r_2 \wedge t))b$ whenever $r_1 \leq r_2$ and $a,b \in \cA_{r_1}$.
Now, fix an $n$-dimensional semicircular Brownian motion $(x_1,\ldots,x_n) \colon \R_+ \to \cA_{\sa}^n$.
In Section \ref{sec.FIPR}, we define a free It\^{o} process (Definition \ref{def.freeItoprocess}) as a process $m \colon \R_+ \to \cA$ that satisfies (the integral form of) an equation
\[
dm(t) = \sum_{i=1}^n u_i(t) \sh dx_i(t) + k(t)\,dt \numberthis\label{eq.frItoprocessintro}
\]
for biprocesses $u_1,\ldots,u_n \colon \R_+ \to \cA \wotimes \cA^{\op}$ and a process $k \colon \R_+ \to \cA$.
Then we prove a product rule for free It\^{o} processes (Theorem \ref{thm.FIPR}) that makes the quadratic covariation rules \eqref{eq.dxidxj1}--\eqref{eq.dxidxj2} rigorous.
This product rule is a ``well-known" generalization of Biane and Speicher's product formula (the $n=1$ case, Theorem 4.1.2 in \cite{bianespeicher}).
It is ``well-known" in the sense that it is used regularly in the literature, and it was proven in the ``concrete" setting (the Cuntz algebra) as Theorem 5 in \cite{kummererspeicher}.
However, it seems that --- until now --- the literature lacks a full proof of this formula in the present ``abstract Wigner space" setting.

In Section \ref{sec.NCder1}, we define divided differences $f^{[k]}$ and noncommutative derivatives $\partial^kf$ of various scalar functions $f$.
When $f$ is a polynomial, $\partial^1f$ corresponds to Voiculescu's free difference quotient from \cite{voiculescudiffquot}.
In Section \ref{sec.FFIFpoly}, we use the free It\^{o} product rule to prove a ``functional" It\^{o} formula for polynomials of free It\^{o} processes (Theorem \ref{thm.FFIFpoly}), which says that if $m$ is a free It\^{o} process satisfying Equation \eqref{eq.frItoprocessintro}, then
\[
d\,p(m(t)) = \partial p(m(t)) \sh dm(t) + \frac{1}{2}\sum_{i=1}^n\Delta_{u_i(t)}p(m(t))\,dt,
\]
where $\Delta_up(m)$ is defined (Notation \ref{nota.Deltap} and Definition \ref{def.Deltap}) in terms of the second noncommutative derivative $\partial^2 p$ of $p$.
This formula generalizes Proposition 4.3.2 in \cite{bianespeicher} (the $n=1$ case).

Our first main result comes in Section \ref{sec.trFFIF}, where we use the free It\^{o} formula for polynomials, some beautiful symmetry properties of the objects in the formula, and an approximation argument to prove a ``traced" It\^{o} formula (Theorem \ref{thm.trFFIF}) for \textit{all} $C^2$ functions of self-adjoint free It\^{o} processes.
(The aforementioned symmetry properties allow one to avoid the multiple-operator-integral-related complications mentioned in the previous section.)
This formula says that if 1) $m$ is a free It\^{o} process satisfying Equation \eqref{eq.frItoprocessintro}, 2) $m^* = m$, and 3) $f \colon \R \to \C$ is a function that is $C^2$ on a neighborhood of the closure of $\bigcup_{t \geq 0}\sigma(m(t))$, then
\[
\frac{d}{dt} \tau(f(m(t))) = \tau\big(f'(m(t))\,k(t)\big) + \frac{1}{2}\sum_{i=1}^n\int_{\R^2}\frac{f'(\lambda)-f'(\mu)}{\lambda-\mu}\,\rho_{m(t),u_i(t)}(d\lambda,d\mu),
\]
where $\rho_{m,u_i}$ is the finite Borel measure on $\R^2$ determined by
\[
\int_{\R^2}\lambda^{j_1}\mu^{j_2}\,\rho_{m,u_i}(d\lambda,d\mu) = \la (m^{j_1} \otimes m^{j_2})u_i,u_i \ra_{L^2(\tau \wotimes \tau^{\op})} = (\tau \wotimes \tau^{\op})(u_i^*(m^{j_1} \otimes m^{j_2})u_i)
\]
for $j_1,j_2 \in \N_0$.
The result is not stated in exactly this way, but this interpretation is derived in Remark \ref{rem.law}.
As an application, we demonstrate in Example \ref{ex.trlog} how to use Theorem \ref{thm.trFFIF} to give simple, computationally transparent (re-)proofs of some key identities from \cite{driverhallkemp,hozhong,demnihamdi,hallho} that are used in the computation of Brown measures of solutions to various free SDEs.
The original proofs of these identities proceeded via rather unintuitive power series arguments, and understanding what was really happening in these arguments was the original motivation for the present study of functional free It\^{o} formulas.
We note that Theorem \ref{thm.trFFIF} is also motivated in the appendix;
the corresponding matrix stochastic calculus formula is given in Corollary \ref{cor.trFIF}.

In Section \ref{sec.NCk}, we define --- following \cite{nikitopoulosNCk} --- the space $NC^k(\R)$ of noncommutative $C^k$ functions $\R \to \C$ (Definition \ref{def.NCk}) and describe large classes of examples (Theorem \ref{thm.NCk}).
For the purposes of this discussion, there is just one important thing to know.
Write $W_k(\R)$ for the space of functions $\R \to \C$ that are Fourier transforms of Borel complex measures on $\R$ with $k$ finite moments (Definition \ref{def.Wk}).
Now, write $W_k(\R)_{\loc}$ for the space of functions $f \colon \R \to \C$ such that for all $r > 0$, there exists $g \in W_k(\R)$ satisfying $f|_{[-r,r]} = g|_{[-r,r]}$.
Then we have\vspace{-0.2mm}
\[
W_k(\R)_{\loc} \subsetneq NC^k(\R), \, \text{ for all } k \in \N. \numberthis\label{eq.Wkloc}
\]
Please see Theorem \ref{thm.NCk} and the sketch of its proof for references to the relevant parts of \cite{nikitopoulosNCk}.

In Section \ref{sec.MOI}, we review the portion of the theory of multiple operator integrals (MOIs) that is relevant to this paper.
Most importantly, we lay out what is needed to make sense of the MOIs in Pseudotheorem \ref{pthm.specialFFIF} when $f \in NC^2(\R)$.
This brings us to Section \ref{sec.FFIF}, which contains our second main result:
the functional free It\^{o} formula for noncommutative $C^2$ functions (Theorem \ref{thm.FFIF}), a generalization of the rigorous version of Pseudotheorem \ref{pthm.specialFFIF} and an extension --- in the self-adjoint case --- of the free It\^{o} formula for polynomials to functions in $NC^2(\R)$.
It says that if 1) $m$ is a free It\^{o} process satisfying Equation \eqref{eq.frItoprocessintro}, 2) $m^*=m$, and 3) $f \in NC^2(\R)$, then\vspace{-0.2mm}
\begin{align*}
    d\,f(m(t)) & = \partial f(m(t))\sh dm(t) + \frac{1}{2}\sum_{i=1}^n\Delta_{u_i(t)}f(m(t))\,dt \\
    & = \int_{\sigma(m(t))}\hspace{-0.5mm}\int_{\sigma(m(t))} f^{[1]}(\lambda,\mu)\,P^{m(t)}(d\lambda)\,dm(t)\,P^{m(t)}(d\mu) + \frac{1}{2}\sum_{i=1}^n\Delta_{u_i(t)}f(m(t))\,dt, \numberthis\label{eq.FFIFintro}\vspace{-0.2mm}
\end{align*}
where $\Delta_u f(m)$ (defined officially in Definition \ref{def.Deltaf}) is determined, in a certain sense (Corollary \ref{cor.expressDeltaf} and Remark \ref{rem.tech}), as a quadratic form by\vspace{-0.2mm}
\[
\frac{1}{2}\Delta_{a \otimes b}f(m) = \int_{\sigma(m)}\hspace{-0.5mm}\int_{\sigma(m)}\hspace{-0.5mm}\int_{\sigma(m)}f^{[2]}(\lambda_1,\lambda_2,\lambda_3)\, P^m(d\lambda_1)\,a\,\tau(b\,P^m(d\lambda_2) \,a)\,b\,P^m(d\lambda_3)\vspace{-0.2mm}
\]
for $a,b \in \cA$.
Recall from the previous section that $f^{[1]}(\lambda,\mu) = \frac{f(\lambda)-f(\mu)}{\lambda-\mu}$ and $f^{[2]}(\lambda,\mu,\nu) = \frac{f^{[1]}(\lambda,\mu) - f^{[1]}(\lambda,\nu)}{\mu - \nu}$ are, respectively, the first and second divided differences of $f$.
Now, Biane and Speicher also established a formula (Proposition 4.3.4 in \cite{bianespeicher}) for $f(m)$ when $f \in W_2(\R)$ and $m$ is a self-adjoint free It\^{o} process driven by a single semicircular Brownian motion.
In Example \ref{ex.BS}, we show that when $n=1$ and $f \in W_2(\R)$, Equation \eqref{eq.FFIFintro} recovers Biane and Speicher's formula.
Owing to the strict containment in Equation \eqref{eq.Wkloc}, this means that not only have we extended Biane and Speicher's formula to the case $n > 1$, but we have also --- through the use of MOIs --- meaningfully enlarged the class of functions for which it can be formulated.
\vspace{-0.75mm}

\section{Background}\label{sec.bg}

\subsection{Free Probability}\label{sec.freeprob}

In this section, we discuss some basic definitions and facts about free probability, noncommutative $L^p$-spaces, noncommutative martingales, and free Brownian motions.
We assume the reader is familiar with these, and we recall only what is necessary for the present application.
For a proper treatment of the basics of free probability, please see \cite{mingospeicher} or \cite{nicaspeicher}.

A pair $(\cA,\varphi)$ is called a \textbf{$\boldsymbol{\ast}$-probability space} if $\cA$ is a unital $\ast$-algebra and $\varphi \colon \cA \to \C$ is a \textbf{state} --- i.e., $\varphi$ is $\C$-linear, unital ($\varphi(1)=1$), and positive ($\varphi(a^*a) \geq 0$ for $a \in \cA$).
We say that a collection $(\cA_i)_{i \in I}$ of (not necessarily $\ast$-)subalgebras of $\cA$ is \textbf{freely independent} if $\varphi(a_1\cdots a_n) = 0$ whenever $\varphi(a_1) = \cdots = \varphi(a_n) = 0$ and $a_1 \in \cA_{i_1},\ldots,a_n \in \cA_{i_n}$ with $i_1 \neq i_2,i_2 \neq i_3$, $\ldots, i_{n-2} \neq i_{n-1}, i_{n-1}\neq i_n$.
We say that a collection $(a_i)_{i \in I}$ of elements of $\cA$ is \textbf{($\boldsymbol{\ast}$-)freely independent} if the collection of ($\ast$-)subalgebras generated by these elements is freely independent.

Let $H$ be a complex Hilbert space and $B(H) \coloneqq \{$bdd. linear maps $H \to H\}$.
A \textbf{von Neumann algebra} is a unital $\ast$-subalgebra of $B(H)$ that is closed in the weak operator topology (WOT).
A pair $(\cA,\tau)$ is a \textbf{$\boldsymbol{W^*}$-probability space} if $\cA$ is a von Neumann algebra and $\tau \colon \cA \to \C$ a \textbf{trace} --- i.e., $\tau$ is a state that is tracial ($\tau(ab) = \tau(ba)$ for $a,b \in \cA$), faithful ($\tau(a^*a)=0$ implies $a=0$), and normal ($\sigma$-WOT continuous).
All $\ast$-probability spaces considered in this paper will be $W^*$-probability spaces.
For more information about von Neumann algebras, please see \cite{dixmier}.
\pagebreak

Fix now a $W^*$-probability space $(\cA,\tau)$.
If $a \in \cA$ is normal, i.e., $a^*a = aa^*$, then the \textbf{$\boldsymbol{\ast}$-distribution} of $a$ is the Borel probability measure
\[
\mu_a(d\lambda) \coloneqq \tau(P^a(d\lambda))
\]
on the spectrum $\sigma(a) \subseteq \C$ of $a$, where $P^a \colon \cB_{\sigma(a)} \to \cA$ is the \textbf{projection-valued spectral measure} of $a$, i.e., the projection-valued measure characterized by the identity $a = \int_{\sigma(a)} \lambda \, P^a(d\lambda)$.
Please see Chapter IX of \cite{conwayfunc} for the basics of projection-valued measures and the Spectral Theorem.
Recall in particular that $f(a) = \int_{\sigma(a)} f(\lambda) \, P^a(d\lambda) = \int_{\sigma(a)} f \, dP^a \in \cA$ for all bounded Borel measurable functions $f \colon \sigma(a) \to \C$.

Let $\mu^{\mathrm{sc}}_0 \coloneqq \delta_0$ and, for $t > 0$,
\[
\mu^{\mathrm{sc}}_t(ds) \coloneqq \frac{1}{2\pi t}\sqrt{(4t-s^2)_+} \, ds
\]
be the semicircle distribution of variance $t$.
Notice that if $t \geq 0$, then $\supp \mu_t^{\mathrm{sc}}$ is equal to $[-2\sqrt{t},2\sqrt{t}] \subseteq \R$, so that if $a \in \cA$ is normal and has $\ast$-distribution $\mu_t^{\mathrm{sc}}$, then $a \in \cA_{\sa}$.
Such an element $a$ is called a \textbf{semicircular element of variance $\boldsymbol{t}$}.
We call $b \in \cA$ a \textbf{circular element of variance $\boldsymbol{t}$} if
\[
b = \frac{1}{\sqrt{2}}(a_1+ia_2)
\]
for two freely independent semicircular elements $a_1,a_2 \in \cA_{\sa}$ of variance $t$.
Since $-a_2$ is still semicircular, we have that if $b \in \cA$ is a circular element of variance $t$, then $b^*$ is as well.

It is worth mentioning that there is a more general algebraic/combinatorial definition of $\ast$-distribution, and one may define (semi)circular elements in a $\ast$-probability space in a more ``intrinsic" way using the notion of \textit{free cumulants}.
Please see \cite{nicaspeicher} for this approach.
Since we shall not need this combinatorial machinery, we content ourselves with the analytic definition above.

Next, we turn to noncommutative $L^p$-spaces.
Please see \cite{dasilva} for a detailed development of the basic properties of noncommutative $L^p$-spaces.

\begin{nota}[Noncommutative $L^p$-spaces]
Let $(\cA,\tau)$ be a $W^*$-probability space.
If $p \in [1,\infty)$, then we define
\[
\|a\|_{L^p(\tau)} \coloneqq \tau(|a|^p)^{\frac{1}{p}} = \tau\big((a^*a)^{\frac{p}{2}}\big)^{\frac{1}{p}},
\]
for all $a \in \cA$, and $L^p(\cA,\tau)$ to be the completion of $\cA$ with respect to the norm $\|\cdot\|_{L^p(\tau)}$.
We also define $L^{\infty}(\cA,\tau) \coloneqq \cA$ and $\|\cdot\|_{L^{\infty}(\tau)} \coloneqq \|\cdot\|_{\cA} = \|\cdot\|$. 
\end{nota}

Similar to the classical case, we have \textit{noncommutative H\"{o}lder's inequality}:
if $a_1,\ldots,a_n \in \cA$, then $\|a_1 \cdots a_n\|_{L^p(\tau)} \leq \|a_1\|_{L^{p_1}(\tau)}\cdots\|a_n\|_{L^{p_n}(\tau)}$ whenever $p_1,\ldots,p_n,p \in [1,\infty]$ and $p_1^{-1}+\cdots+p_n^{-1}=p^{-1}$.
This allows us to extend multiplication to a bounded $n$-linear map $L^{p_1}(\cA,\tau) \times \cdots \times L^{p_n}(\cA,\tau) \to L^p(\cA,\tau)$.
In addition, there is a dual characterization of the noncommutative $L^p$-norm:
if $a \in \cA$, then we have $\|a\|_{L^p(\tau)} = \sup\{\tau(ab) : b \in \cA, \|b\|_{L^q(\tau)} \leq 1\}$ whenever $p^{-1}+q^{-1}=1$.
This leads to the duality relationship $L^q(\cA,\tau) \cong L^p(\cA,\tau)^*$, via the map $a \mapsto (b \mapsto \tau (ab))$, when $p^{-1}+q^{-1}=1$ and $p \neq \infty$, as in the classical case.
Moreover, the $\sigma$-WOT on $\cA$ coincides with the weak$^*$ topology on $L^1(\cA,\tau)^* \cong L^{\infty}(\cA,\tau) = \cA$.

Finally, we briefly discuss noncommutative martingales and free Brownian motions.
For this, we recall that if $\cB \subseteq \cA$ is a \textbf{$\boldsymbol{W^*}$-subalgebra} --- i.e., a WOT-closed $\ast$-subalgebra --- then there is a unique positive linear map $\tau[ \,\cdot  \mid \cB] \colon \cA \to \cB$ such that $\tau[b_1ab_2 \mid \cB] = b_1 \tau[a \mid \cB] b_2$, for all $a \in \cA$ and $b_1,b_2 \in \cB$.
We call $\tau[\,\cdot \mid \cB]$ the \textbf{conditional expectation onto $\boldsymbol{\cB}$}. It was introduced in \cite{takesaki}.
It extends to a (weak) contraction $L^p(\cA,\tau) \to L^p(\cB,\tau)$, for all $p \in [1,\infty]$.
When $p=2$, we get the orthogonal projection of $L^2(\cA,\tau)$ onto $L^2(\cB,\tau) \subseteq L^2(\cA,\tau)$.
In particular, as it is often useful to remember, if $a \in \cA$ and $b \in \cB$, then
\[
b=\tau[a \mid \cB] \iff \tau(b_0a) = \tau(b_0b), \text{ for all } b_0 \in \cB.
\]
This implies, for instance, that if $a$ is freely independent from $\cB$, then we have $\tau[a \mid \cB] = \tau(a)1 = \tau(a)$.

Now, an increasing collection $(\cA_t)_{t \geq 0}$ of $W^*$-subalgebras of $\cA$ is called a \textbf{filtration} of $\cA$, and the triple $(\cA,(\cA_t)_{t \geq 0},\tau)$ is called a \textbf{filtered $\boldsymbol{W^*}$-probability space}.
Fix a filtration $(\cA_t)_{t \geq 0}$ of $\cA$ and $p \in [1,\infty]$.
A \textbf{$\boldsymbol{L^p}$-process} $a = (a(t))_{t \geq 0} \colon \R_+ \to L^p(\cA,\tau)$ is \textbf{adapted} (\textbf{to $\boldsymbol{(\cA_t)_{t \geq 0}}$}) if $a(t) \in L^p(\cA_t,\tau) \subseteq L^p(\cA,\tau)$, for every $t \geq 0$.
An adapted $L^p$-process $m \colon \R_+ \to L^p(\cA,\tau)$ is called a \textbf{noncommutative $\boldsymbol{L^p}$-martingale} (\textbf{with respect to $\boldsymbol{((\cA_t)_{t \geq 0},\tau)}$}) if $\tau[m(t) \mid \cA_s] = m(s)$ whenever $0 \leq s \leq t < \infty$.
If $p=\infty$, then we shall omit the ``$L^p$" from these terms.

An\hspace{-0.2mm} $n$-tuple $m \hspace{-0.2mm}=\hspace{-0.2mm} (m_1,\ldots,m_n) \hspace{-0.2mm}\colon\hspace{-0.2mm} \R_+ \hspace{-0.2mm}\to\hspace{-0.2mm} \cA^n$ \hspace{-0.2mm}of\hspace{-0.2mm} adapted \hspace{-0.2mm}processes\hspace{-0.2mm} is \hspace{-0.2mm}called\hspace{-0.2mm} an \textbf{$\boldsymbol{n}$-dimensional\hspace{-0.2mm} (semi)circular Brownian motion} (\textbf{in $\boldsymbol{(\cA,(\cA_t)_{t \geq 0},\tau)}$}) if $m(0) = 0$ and $\{m_i(t)-m_i(s) : 1 \leq i \leq n\}$ is a $\ast$-freely independent collection of (semi)circular elements of variance $t-s$ that is $\ast$-freely independent from $\cA_s$ when $0 \leq s < t < \infty$.
More concisely, $m(0) = 0$ and $m$ has ``jointly $\ast$-free (semi)circular increments."
It follows from the comments about conditional expectation and the free increments property that (semi)circular Brownian motion is a noncommutative martingale.
Also, if $m$ is an $n$-dimensional circular Brownian motion, then the process $\sqrt{2} (\cRe m,\cIm m) = \frac{1}{\sqrt{2}}(m+m^*,\frac{1}{i}(m-m^*))$ is a $2n$-dimensional semicircular Brownian motion.

\subsection{Tensor Products}\label{sec.tensprod}

In this section, we set notation for and review some information about several topological tensor products of which we shall make use.

\begin{nota}
Write $\otimes$ for the algebraic tensor product, $\otimes_2$ for the Hilbert space tensor product, $\otimes_{\min}$ for the minimal (or spatial) $C^*$-tensor product, $\wotimes$ for the (spatial) von Neumann algebra tensor product, and $\potimes$ for the Banach space projective tensor product.
\end{nota}

Though we assume the reader has some familiarity with these tensor products, we recall their definitions/constructions for convenience.
Let $(H,\la \cdot, \cdot \ra_H)$ and $(K, \la \cdot, \cdot \ra_K)$ be Hilbert spaces.
There exists a unique inner product $\la \cdot, \cdot \ra_{H \hotimes K}$ on $H \otimes K$ determined by $\la h_1 \otimes k_1, h_2 \otimes k_2 \ra_{H \hotimes K} = \la h_1,h_2 \ra_H\la k_1,k_2 \ra_K$, for all $h_1,h_2 \in H$ and $k_1,k_2 \in K$.
The Hilbert space tensor product $H \hotimes K$ is defined to be the completion of $H \otimes K$ with respect to $\la \cdot, \cdot \ra_{H \hotimes K}$.
If $a \in B(H)$ and $b \in B(K)$, then there exists unique $a \hotimes b \in B(H \hotimes K)$ such that $(a \hotimes b)(h \otimes k) = ah \otimes bk$, for all $h \in H$ and $k \in K$.
Moreover, $\|a \hotimes b\|_{B(H \hotimes K)} = \|a\|_{B(H)}\|b\|_{B(K)}$.
It is not difficult to show that the linear map $B(H) \otimes B(K) \to B(H \hotimes K)$ determined by $a \otimes b \mapsto a \hotimes b$ is an injective $\ast$-homomorphism when $B(H) \otimes B(K)$ is given the tensor product $\ast$-algebra structure.
This allows us to view $B(H) \otimes B(K)$ as a $\ast$-subalgebra of $B(H \hotimes K)$ and justifies writing, as we shall, $a \otimes b$ instead of $a \hotimes b$.
In particular, if $\cA \subseteq B(H)$ and $\cB \subseteq B(K)$ are $C^*$-algebras, then we may naturally view $\cA \otimes \cB$ as a $\ast$-subalgebra of $B(H \hotimes K)$.
The minimal $C^*$-tensor product $\cA \otimes_{\min} \cB$ of $\cA$ and $\cB$ is the operator norm closure of $\cA \otimes \cB$ in $B(H \hotimes K)$.
If in addition $\cA$ and $\cB$ are von Neumann algebras, then $\cA \wotimes \cB$ is the WOT closure --- equivalently, by the Kaplansky Density Theorem, the $\sigma$-WOT closure --- of $\cA \otimes \cB$ in $B(H \hotimes K)$.
If $\tau_1$ and $\tau_2$ are traces on $\cA$ and $\cB$, respectively, then we write $\tau_1 \wotimes \tau_2 \colon \cA \wotimes \cB \to \C$ for the unique trace on $\cA \wotimes \cB$ determined by $(\tau_1 \wotimes \tau_2)(a \otimes b) = \tau_1(a) \, \tau_2(b)$, for all $a \in \cA$ and $b \in \cB$.
This is the \textbf{tensor product trace} on $\cA \wotimes \cB$.
For more information on $\otimes_{\min}$ and $\wotimes$, please see Chapter 3 of \cite{brownozawa} or Chapter 11 of \cite{kadisonringrose2}.

Now, let $(V,\|\cdot\|_V)$ and $(W,\|\cdot\|_W)$ be Banach spaces. For $u \in V \otimes W$, define
\[
\|u\|_{V \potimes W} \coloneqq \inf\Bigg\{\sum_{j=1}^n\|v_j\|_V\|w_j\|_W : u = \sum_{j=1}^n v_j\otimes w_j\Bigg\}.
\]
Then $\|\cdot\|_{V \potimes W}$ is a norm on $V \otimes W$, and the projective tensor product $V \potimes W$ is defined as the completion of $V \otimes W$ with respect to $\|\cdot\|_{V \potimes W}$.
This tensor product satisfies the type of universal property that the algebraic tensor product satisfies:
it bounded-linearizes bounded bilinear maps.
If $V$ and $W$ are in addition Banach algebras, then $V \potimes W$ is also a Banach algebra with multiplication extending the tensor product multiplication on $V \otimes W$.

There is also a concrete description of the elements of $V \potimes W$.
Indeed, it can be shown that every element $u \in V \potimes W$ admits a decomposition
\[
u = \sum_{n=1}^{\infty} v_n \otimes w_n \; \text{ with } \; \sum_{n=1}^{\infty}\|v_n\|_V\|w_n\|_W < \infty \numberthis\label{eq.projdecomp}
\]
and that
\[
\|u\|_{V \potimes W} = \inf\Bigg\{\sum_{n=1}^{\infty}\|v_n\|_V\|w_n\|_W : u = \sum_{n=1}^{\infty} v_n \otimes w_n \text{ as in Equation \eqref{eq.projdecomp}}\Bigg\}.
\]
Please see Chapter 2 of \cite{ryan} for a proper development.
This has a number of consequences.
The most relevant one for us will be a description of $\cA \potimes \cB$ for $C^*$-algebras $\cA$ and $\cB$.

\begin{prop}\label{prop.projC}
Let $\cA$ and $\cB$ be $C^*$-algebras and $\iota_{\min} \colon \cA \potimes \cB \to \cA \otimes_{\min} \cB$ be the natural map obtained by applying the universal property of $\potimes$ to the inclusion $\cA \otimes \cB \hookrightarrow \cA \otimes_{\min} \cB$.
Then $\iota_{\min}$ is injective.
\end{prop}

This follows from Proposition 2.2 and the remark following it in \cite{haagerup}.
From Proposition \ref{prop.projC} and the concrete description of $V \potimes W$ above, we see that if $\cA \subseteq B(H)$ and $\cB \subseteq B(K)$ are $C^*$-algebras, then $\cA \potimes \cB$ can be represented as the subalgebra of $B(H \hotimes K)$ of elements $u \in B(H \hotimes K)$ admitting a decomposition $u = \sum_{n = 1}^{\infty}a_n \otimes b_n \in B(H \hotimes K)$ such that $(a_n)_{n \in \N} \in \cA^{\N}, (b_n)_{n \in \N} \in \cB^{\N}$, and $\sum_{n=1}^{\infty}\|a_n\|_{B(H)}\|b_n\|_{B(K)} < \infty$.
In particular, we have the chain of inclusions $\cA \otimes \cB \subseteq \cA \potimes \cB \subseteq \cA \otimes_{\min} \cB \subseteq B(H \hotimes K)$.

\section{Free Stochastic Calculus I: Polynomials and the Traced Formula}\label{sec.fstochcalc1}

\subsection{The Free Stochastic Integral against (Semi)circular Brownian Motion}\label{sec.integhasdx}

In this section, we review Biane and Speicher's construction from \cite{bianespeicher} of the free stochastic integral against semicircular Brownian motion and use it to define free stochastic integrals against circular Brownian motion.
We begin by setting notation for a few useful algebraic operations.

\begin{nota}\label{nota.alg}
Let $H$ be a complex Hilbert space and $\cA \subseteq B(H)$ be a von Neumann algebra.
\begin{enumerate}[label=(\alph*),leftmargin=2\parindent]
    \item $\cA^{\op}$ is the opposite von Neumann algebra of $\cA$, i.e., the von Neumann algebra with the same addition, $\ast$-operation, and topological structure as $\cA$ but the opposite multiplication operation $a \cdot b \coloneqq ba$.
    If $\tau \colon \cA \to \C$ is a trace, then we write $\tau^{\op} \colon \cA^{\op} \to \C$ for the induced trace on $\cA^{\op}$.\label{item.op}
    \item We write $(\cdot)^{\flip} \colon \cA \wotimes \cA^{\op} \to \cA \wotimes \cA^{\op}$ for the unique $\sigma$-WOT continuous (and isometric) linear map determined by $(a \otimes b)^{\flip} = b \otimes a$.
    Also, write $u^{\mathsmaller{\bigstar}} \coloneqq (u^*)^{\flip}$, for all $u \in \cA \wotimes \cA^{\op}$, where $(\cdot)^*$ denotes the standard tensor product $\ast$-operation on $\cA \wotimes \cA^{\op}$ (e.g., $(a \otimes b)^* = a^* \otimes b^*$).\label{item.bigstar}
    \item We write $\# \colon \cA \potimes \cA^{\op} \to B(\cA) = \{$bounded linear maps $\cA \to \cA\}$ for the bounded linear map --- and, actually, algebra homomorphism --- determined by $\#(a \otimes b)c = acb$.
    Also, we write $u \sh c \coloneqq \#(u)c$, for all $u \in \cA \potimes \cA^{\op}$ and $c \in \cA$.
    Note that if $u \in \cA \potimes \cA^{\op} \subseteq \cA \wotimes \cA^{\op}$, then $u^*,u^{\flip},u^{\mathsmaller{\bigstar}} \in \cA \potimes \cA^{\op}$ and $(u \sh c)^* = (u^{\mathsmaller{\bigstar}})\sh c^*$, for all $c\in \cA$.\label{item.hash}
    \item (Not used until Section \ref{sec.FFIFpoly}) Write $B_2((\cA \wotimes \cA^{\op})^2;\cA \wotimes \cA^{\op})$ for the space of bounded bilinear maps $(\cA \wotimes \cA^{\op})^2 \to \cA \wotimes \cA^{\op}$, and write $\#_2^{\otimes} \colon (\cA \wotimes \cA^{\op})^{\potimes 3} \to B_2((\cA \wotimes \cA^{\op})^2;\cA \wotimes \cA^{\op})$ for the bounded linear map determined by $\#_2^{\otimes}(u_1 \otimes u_2 \otimes u_3)[v_1,v_2] = u_1v_1u_2v_2u_3$ for $u_1,u_2, u_3,v_1,v_2 \in \cA \wotimes \cA^{\op}$.
    If $A \in (\cA \wotimes \cA^{\op})^{\potimes 3}$ and $u,v \in \cA \wotimes \cA^{\op}$, then we write $A \sh_2^{\mathsmaller{\otimes}}[u,v] \coloneqq \#_2^{\otimes}(A)[u,v]$.\label{item.tenshash}
\end{enumerate}
\end{nota}
\begin{rem}
If $H$ is finite-dimensional and $\cA = B(H)$, then one can use elementary linear algebra to show that $\# \colon \cA \potimes \cA^{\op} = \cA \otimes \cA^{\op} \to B(\cA)$ is a linear isomorphism.
Moreover, $\#$ is a $\ast$-homomorphism when $\cA \otimes \cA^{\op}$ is given the tensor product $\ast$-operation and $B(\cA)$ is given the adjoint operation associated to the Hilbert--Schmidt inner product on $\cA = B(H)$.
This is why we have chosen to write $(\cdot)^*$ for the tensor product $\ast$-operation on $\cA \wotimes \cA^{\op}$;
in \cite{bianespeicher}, the symbol $(\cdot)^*$ is used for the operation $(\cdot)^{\mathsmaller{\bigstar}}$ from \ref{item.bigstar}.
\end{rem}

Some justification is in order for what is written in \ref{item.op} and \ref{item.bigstar} above.
First, we observe that $\cA^{\op}$ is, indeed, a von Neumann algebra.
Abstractly, $\cA^{\op}$ is clearly a $C^*$-algebra with a predual (the same predual as $\cA$).
Concretely, $\cA^{\op}$ can be represented on the dual $H^*$ of $H$ via the transpose map
\[
B(H) \ni a \mapsto (H^* \ni \ell \mapsto \ell \circ a \in H^*) \in B(H^*).
\]
This map is a $\ast$-anti-homomorphism that is a homeomorphism with respect to the WOT and the $\sigma$-WOT, so the image of $\cA$ under the transpose map is a von Neumann algebra isomorphic to $\cA^{\op}$.
Next, using this representation of $\cA^{\op}$, we confirm that $(\cdot)^{\flip}$ is well-defined.
Certainly, the condition in the definition determines a linear map $(\cdot)^{\flip} \colon \cA \otimes \cA^{\op} \to \cA \otimes \cA^{\op}$.
What remains to be confirmed is that the latter linear map is $\sigma$-WOT continuous and isometric.
To see this, write $(\cdot)^{\mathfrak{f}} \colon H \hotimes H^* \to H \hotimes H^*$ for the conjugate-linear surjective isometry determined by $h \otimes \la \cdot,k \ra \mapsto k \otimes \la \cdot,h \ra$.
Then it is easy to show that
\[
\la u^{\flip}\xi, \eta \ra_{H \hotimes H^*} = \big\la u\eta^{\mathfrak{f}},\xi^{\mathfrak{f}} \big\ra_{H \hotimes H^*},
\]
for all $u \in \cA \otimes \cA^{\op} \subseteq B(H \hotimes H^*)$ and $\xi,\eta \in H \hotimes H^*$.
This implies both desired conclusions. 
\pagebreak

Next, we define simple biprocesses and their integrals against arbitrary functions.
For the remainder  of this paper, fix a filtered $W^*$-probability space $(\cA,(\cA_t)_{t \geq 0},\tau)$.

\begin{defi}[Biprocesses]
A map $u \colon \R_+ \to \cA \otimes \cA^{\op}$ is called a \textbf{biprocess}.
If $u(t) \in \cA_t \otimes \cA_t^{\op}$, for every $t \geq 0$, then $u$ is called \textbf{adapted}.
If there is a finite partition $0 = t_0 < t_1 < \cdots < t_n < \infty$ of $\R_+$ such that, for all $i \in \{1,\ldots,n\}$, $u$ is constant on $[t_{i-1},t_i)$, and $u(t) = 0$ for $t \geq t_n$, then $u$ is called \textbf{simple}.
We write $\SB$ for the space of simple biprocesses and $\SBa \subseteq \SB$ for the subspace of simple adapted biprocesses.
\end{defi}

\begin{nota}[Integrals of Simple Biprocesses]
If $u \in \SB$, then
\[
u = \sum_{i=1}^n1_{[t_{i-1},t_i)} u(t_{i-1})
\]
for some partition $0 = t_0 < \cdots < t_n < \infty$.
If $m \colon \R_+ \to \cA$ is any function, then we define
\[
\int_0^{\infty} u(t) \sh dm(t) = \int_0^{\infty} u \sh dm \coloneqq \sum_{i=1}^n u(t_{i-1})\sh[m(t_i)-m(t_{i-1})] \in \cA.
\]
By standard arguments (from scratch or using the basic theory of finitely additive vector measures), the element $\int_0^{\infty} u \sh dm$ does not depend on the chosen decomposition of $u$, and $\SB \ni u \mapsto \int_0^{\infty} u \sh dm \in \cA$ is linear.
\end{nota}

Note that if $u \in \SB$ and $r,s \geq 0$ are such that $r \leq s$, then $1_{[r,s)}u \in \SB$ and $u^{\mathsmaller{\bigstar}} \in \SB$.
Thus the statement of the lemma below, the proof of which we leave to the reader, makes sense.

\begin{lem}[Properties of Integrals of Simple Biprocesses]\label{lem.algstochint}
Let $m \colon \R_+ \to \cA$ be any function and $u \in \SB$.
Fix $r,s \geq 0$ with $r \leq s$, and define
\[
\int_r^s u(t) \sh dm(t) = \int_r^s u \sh dm \coloneqq \int_0^{\infty}(1_{[r,s)}u)\sh dm \in \cA.
\]
Then
\begin{enumerate}[label=(\roman*),font=\normalfont,leftmargin=2\parindent]
    \item $\SB \ni u \mapsto \int_r^s u(t) \sh dm(t) \in \cA$ is linear;
    \item $\big(\int_r^s u \sh dm \big)^* = \int_r^s u^{\mathsmaller{\bigstar}} \sh dm^*$;
    \item if $u \in \SBa$ and $m$ is adapted, then $\int_0^{\boldsymbol{\cdot}} u \sh dm \coloneqq \big(\int_0^t u \sh dm\big)_{t \geq 0}$ is adapted;
    \item if $u(t) = 0$ for all $t \geq s > 0$, then $\int_r^{s_1} u(t) \sh dm(t) = \int_r^{s_2} u(t) \sh dm(t)$ whenever $s_1,s_2 \geq s$;
    and
    \item $\int_r^s u(t) \sh dm(t) = \int_0^s u(t) \sh dm(t) - \int_0^r u(t) \sh dm(t)$.
\end{enumerate}
\end{lem}

Next, we introduce a larger space of integrands for the case when $m$ is a (semi)circular Brownian motion.
Notice that a simple biprocess $u \colon \R_+ \to \cA \otimes \cA^{\op} \subseteq L^p(\cA \wotimes \cA^{\op}, \tau \wotimes \tau^{\op})$ is a compactly supported simple --- in particular, Bochner integrable\footnote{For information on the Bochner integral, please see Appendix E of \cite{cohn}.} --- map $\R_+ \to L^p(\cA \wotimes \cA^{\op},\tau \wotimes \tau^{\op})$, for all $p \in [1,\infty]$.

\begin{nota}
Fix $p,q \in [1,\infty]$, and let $(\cB,\eta)$ be a $W^*$-probability space. 
\begin{enumerate}[label=(\alph*),leftmargin=2\parindent]
    \item For $u \in L_{\loc}^q(\R_+;L^p(\cB,\eta)) = L_{\loc}^q(\R_+,\mathrm{Lebesgue};L^p(\cB,\eta))$, define
    \[
    \|u\|_{L_t^qL^p(\eta)} \coloneqq \Bigg(\int_0^t \|u(s)\|_{L^p(\eta)}^q \, ds \Bigg)^{\frac{1}{q}} \; \text{ and } \; \|u\|_{L^qL^p(\eta)} \coloneqq \Bigg(\int_0^{\infty} \|u(s)\|_{L^p(\eta)}^q \, ds \Bigg)^{\frac{1}{q}}
    \]
    for $t \geq 0$ (with the obvious modification for $q=\infty$).
    Of course, $\|\cdot\|_{L_t^2L^2(\eta)}$ comes from the ``inner product" $\la u,v \ra_{L_t^2L^2(\eta)} = \int_0^t \la u(s),v(s) \ra_{L^2(\eta)} \, ds$.
    \item Define
    \begin{align*}
    \mathcal{L}^{2,p} & \coloneqq \overline{\SBa} \subseteq L^2(\R_+;L^p(\cA \wotimes \cA^{\op},\tau \wotimes \tau^{\op})) \; \text{ and} \\
    \Lambda^{2,p} & \coloneqq \overline{\SBa} \subseteq L_{\loc}^2(\R_+;L^p(\cA \wotimes \cA^{\op},\tau \wotimes \tau^{\op})),
    \end{align*}
    where the first closure above takes place in the Banach space $L^2(\R_+;L^p(\cA \wotimes \cA^{\op},\tau \wotimes \tau^{\op}))$ and the second takes place in the Fr\'{e}chet space $L_{\loc}^2(\R_+;L^p(\cA \wotimes \cA^{\op},\tau \wotimes \tau^{\op}))$.
    We write
    \begin{align*}
        \mathcal{L}^2 & \coloneqq \mathcal{L}^{2,\infty} \subseteq L^2(\R_+;\cA \wotimes \cA^{\op}) \; \text{ and} \\
        \Lambda^2 & \coloneqq \Lambda^{2,\infty} \subseteq L_{\loc}^2(\R_+;\cA \wotimes \cA^{\op})
    \end{align*}
    for the case $p=\infty$.
\end{enumerate}
To be clear, the $L^q$- and $L_{\loc}^q$-spaces above are the Bochner $L^q$- and $L_{\loc}^q$-spaces.
\end{nota}
\begin{rem}\label{rem.min}
The use of $\mathcal{L}$ and $\Lambda$ above is inspired by the notation used in \cite{chungwilliams} for the classical case.
Biane and Speicher use the notation $\mathscr{B}_p^a$ in \cite{bianespeicher} for the space $\mathcal{L}^{2,p}$, though their definition is stated as an abstract completion of $\SBa$.
Also, we note that simple biprocesses take values in $\cA \otimes \cA^{\op} \subseteq \cA \otimes_{\min} \cA^{\op}$, and $\cA \otimes_{\min} \cA^{\op} \subseteq \cA \wotimes \cA^{\op}$ is a norm-closed subspace.
In particular, all elements of $\Lambda^2$ actually take values (almost everywhere) in $\cA \otimes_{\min} \cA^{\op}$.
In other words, $\Lambda^2 \subseteq L_{\loc}^2(\R_+;\cA \otimes_{\min} \cA^{\op})$.
\end{rem}

Only the case $p=\infty$ will matter to us in later sections.
However, we note that in the case $p=2$, there is an It\^{o} Isometry, just as in the classical case.
It says that if $x \colon \R_+ \to \cA$ is a semicircular Brownian motion (or, in fact, a circular Brownian motion), then
\[
\Bigg\la \int_0^t u \sh dx, \int_0^t v \sh dx \Bigg\ra_{L^2(\tau)} = \la u,v \ra_{L_t^2L^2(\tau \wotimes \tau^{\op})},
\]
for all $u,v \in \SBa$ and $t \geq 0$.
Please see Proposition 3.1.1 in \cite{bianespeicher}.
We now focus on the case $p=\infty$.

\begin{thm}[Biane--Speicher \cite{bianespeicher}]\label{thm.BSstochint}
Let $x \colon \R_+ \to \cA_{\sa}$ be a semicircular Brownian motion and $z \colon \R_+ \to \cA$ be a circular Brownian motion.
Fix $u \in \SBa$ and $m \in \{x,z,z^*\}$.
\begin{enumerate}[label=(\roman*),font=\normalfont,leftmargin=2\parindent]
    \item $\int_0^{\boldsymbol{\cdot}}u \sh dm$ is a noncommutative martingale.\label{item.mart}
    \item ($L^{\infty}$-Burkholder--Davis--Gundy (BDG) Inequality) We have
    \[
    \Bigg\|\int_0^{\infty} u(t) \sh dx(t) \Bigg\| \leq 2\sqrt{2}\|u\|_{L^2L^{\infty}(\tau \wotimes \tau^{\op})}.
    \]
    It follows that the map $\{(r_1,r_2) : 0 \leq r_1 \leq r_2\} \ni (s,t) \mapsto \int_s^t u \sh dm \in \cA$ is continuous.\label{item.BDG}
\end{enumerate}
\end{thm}
\begin{proof}
If $m=x$, then item \ref{item.mart} is Proposition 2.2.2 in \cite{bianespeicher}.
The inequality in item \ref{item.BDG} is Theorem 3.2.1 in \cite{bianespeicher}.
The remainder of the claims in the theorem (i.e., those for $m \in \{z,z^*\}$) follow from the corresponding claims for $m=x$ because $z = \frac{1}{\sqrt{2}}(x_1+ix_2)$ and $z^* = \frac{1}{\sqrt{2}}(x_1-ix_2)$, where $x_1 = \sqrt{2} \cRe z$ and $x_2 = \sqrt{2} \cIm z$ are (freely independent) semicircular Brownian motions.
\end{proof}

\begin{cor}\label{cor.extensionofstochinteg}
Retain the setup of Theorem \ref{thm.BSstochint}, and fix $s \geq 0$.
The linear map $\int_s^{\boldsymbol{\cdot}}\boldsymbol{\cdot} \, \sh dm \colon \SBa \to C([s,\infty);\cA)$ extends uniquely to a continuous linear map $\Lambda^2 \to C([s,\infty);\cA)$, which we notate the same way.
If $u \in \Lambda^2$, then $\int_0^{\boldsymbol{\cdot}}u \sh dm$ is a continuous noncommutative martingale that satisfies the identities
\[
\int_s^t u \sh dm = \int_0^t u \sh dm - \int_0^s u \sh dm \; \text{ and } \; \Bigg(\int_s^t u \sh dm\Bigg)^* = \int_s^t u^{\mathsmaller{\bigstar}} \sh dm^*,
\]
and the bounds
\[
\Bigg\|\int_s^t u \sh dx \Bigg\| \leq 2\sqrt{2}\Bigg(\int_s^t\|u(r)\|_{L^{\infty}(\tau \wotimes \tau^{\op})}^2\,dr\Bigg)^{\frac{1}{2}} \; \text{ and } \; \Bigg\|\int_s^t u \sh dz^{\e} \Bigg\| \leq 4\Bigg(\int_s^t\|u(r)\|_{L^{\infty}(\tau \wotimes \tau^{\op})}^2\,dr\Bigg)^{\frac{1}{2}}
\]
for $t \geq s$ and $\e \in \{1,\ast\}$.
Similar comments apply to $\int_0^{\infty} u \sh dm$ for $u \in \mathcal{L}^2$.
\end{cor}

\begin{defi}[Free Stochastic Integral]
For every $u \in \Lambda^2$ and $m \in \{x,z,z^*\}$ as above, the process $\int_0^{\boldsymbol{\cdot}}u \sh dm$ from Corollary \ref{cor.extensionofstochinteg} is called the \textbf{free stochastic integral of} $\boldsymbol{u}$ \textbf{against} $\boldsymbol{m}$.
\end{defi}

We end this section by giving a large class of examples of members of $\Lambda^{2,p}$.
Note that $u \in \Lambda^{2,p}$ if and only if $1_{[0,t)}u \in \mathcal{L}^{2,p}$, for all $t > 0$.
We shall use this freely below. 

\begin{prop}\label{prop.RCLB}
Suppose that $u \colon \R_+ \to \cA \otimes_{\min} \cA^{\op}$ is (norm) right-continuous, locally bounded, and adapted, i.e., $u(t) \in \cA_t \otimes_{\min} \cA_t^{\op}$ for all $t \geq 0$.
If $p \in [1,\infty]$ and $v \in \Lambda^{2,p}$, then $u\,v \in \Lambda^{2,p}$.
The latter juxtaposition is the (pointwise) usual action of $\cA \wotimes \cA^{\op}$ on $L^p(\cA \wotimes \cA^{\op},\tau \wotimes \tau^{\op})$.
\end{prop}
\begin{proof}
First, note that if $0 \leq s \leq t < \infty$ and $w \in \cA_s \otimes \cA^{\op}_s$, then we have $1_{[s,t)}w\,v \in \mathcal{L}^{2,p}$.
(Approximate $v$ by simple adapted biprocesses to see this.)
We claim this holds for $w \in \cA_s \otimes_{\min} \cA_s^{\op}$ as well.
Indeed, let $(w_n)_{n \in \N}$ be a sequence in $\cA_s \otimes \cA_s^{\op}$ converging in the norm topology to $w$.
By noncommutative H\"{o}lder's Inequality, we have
\[
\|1_{[s,t)}w_n\,v-1_{[s,t)}w\,v\|_{L^2L^p(\tau\wotimes \tau^{\op})} \leq \|w_n-w\|_{L^{\infty}(\tau\wotimes\tau^{\op})}\|v\|_{L_t^2L^p(\tau \wotimes \tau^{\op})} \to 0,
\]
i.e., $1_{[s,t)}w_n\,v \to 1_{[s,t)} w\,v$ in $L^2(\R_+;L^p(\cA \wotimes \cA^{\op},\tau \wotimes \tau^{\op}))$ as $n \to \infty$.
Thus $1_{[s,t)}w\,v \in \mathcal{L}^{2,p}$, as claimed.

Now, let $t > 0$, and define
\[
u^n \coloneqq \sum_{i=1}^n 1_{[\text{\scalebox{0.7}{$\frac{i-1}{n}t,\frac{i}{n}t$}})}u\big(\mathsmaller{\frac{i-1}{n}}t\big)
\]
for $n \in \N$. Then $u^n\,v \in \mathcal{L}^2$ by the previous paragraph.
Since $u$ is right-continuous, $u^n \to 1_{[0,t)}u$ pointwise in $\cA \otimes_{\min} \cA^{\op} \subseteq \cA \wotimes \cA^{\op}$ as $n \to \infty$.
In particular, $u^n\,v \to 1_{[0,t)}u\,v$ pointwise in $L^p(\cA \wotimes \cA^{\op},\tau \wotimes \tau^{\op})$ as $n \to \infty$.
Also, $\sup_{n \in \N} \|u^n\,v\|_{L^p(\tau \wotimes \tau^{\op})} \leq  1_{[0,t)}\|v\|_{L^p(\tau \wotimes \tau^{\op})} \sup_{0 \leq r < t}\|u(r)\|_{L^{\infty}(\tau \wotimes \tau^{\op})}$, which is in $L^2(\R_+)$ because $u$ is locally bounded.
Therefore, by the Dominated Convergence Theorem,
\[
\|u^n\,v-1_{[0,t)}u\,v\|_{L^2L^p(\tau \wotimes \tau^{\op})} = \|u^n\,v-u\,v\|_{L_t^2L^p(\tau \wotimes \tau^{\op})} \to 0
\]
as $n \to \infty$.
Thus $u^n\,v \to 1_{[0,t)}u\,v$ in $L^2(\R_+;L^p(\cA \wotimes \cA^{\op},\tau \wotimes \tau^{\op}))$ as $n \to \infty$.
We conclude $1_{[0,t)}u\,v \in \mathcal{L}^{2,p}$, and therefore, since $t > 0$ was arbitrary, $u \,v \in \Lambda^{2,p}$, as desired.
\end{proof}

\begin{cor}\label{cor.RCLL}
Suppose that $u \colon \R_+ \to \cA \otimes_{\min} \cA^{\op}$ is RCLL, i.e., $u$ is (norm) right-continuous and the left limit $u(t-) \coloneqq \lim_{s \nearrow t} u(s) \in \cA \otimes_{\min} \cA^{\op}$ exists for each $t \geq 0$.
If $u$ is adapted, $p \in [1,\infty]$, and $v \in \Lambda^{2,p}$, then $u\,v \in \Lambda^{2,p}$.
\end{cor}
\begin{proof}
If $u$ is RCLL, then $u$ is right-continuous and locally bounded, so Proposition \ref{prop.RCLB} applies.
\end{proof}

\begin{ex}\label{ex.elemtens}
Suppose $u \colon \R_+ \to \cA \potimes \cA^{\op}$ is continuous with respect to $\|\cdot\|_{\cA \potimes \cA^{\op}}$ and $u(t) \in \cA_t \potimes \cA_t^{\op}$, for all $t \geq 0$.
Since $\cA \potimes \cA^{\op} \subseteq \cA \otimes_{\min} \cA^{\op}$ and $\|\cdot\|_{\cA \potimes \cA^{\op}} \leq \|\cdot\|_{\cA \otimes_{\min} \cA^{\op}}$, $u$ satisfies the hypotheses of Proposition \ref{prop.RCLB} (even Corollary \ref{cor.RCLL}).
A common example of this form is
\[
u = \sum_{i=1}^n a_i \otimes b_i = \Bigg(\sum_{i=1}^n a_i(t) \otimes b_i(t)\Bigg)_{t \geq 0}
\]
for continuous adapted processes $a_1,b_1\ldots,a_n,b_n \colon \R_+ \to \cA$. 
\end{ex}

\subsection{Free It\^{o} Product Rule}\label{sec.FIPR}

In this section, we set up and prove an It\^{o} product rule for free It\^{o} processes (Theorem \ref{thm.FIPR}).
We begin by officially introducing free It\^{o} processes.
Recall that $(\cA,(\cA_t)_{t \geq 0},\tau)$ is a fixed $W^*$-probability space.

\begin{defi}[Free It\^{o} Process]\label{def.freeItoprocess}
Fix $n \in \N$ and an $n$-dimensional semicircular Brownian motion $(x_1,\ldots,x_n) \colon \R_+ \to \cA_{\sa}^n$.
A \textbf{free It\^{o} process} is a process $m \colon \R_+ \to \cA$ satisfying
\begin{align*}
    dm(t) & = \sum_{i=1}^n u_i(t) \sh dx_i(t) + k(t) \, dt, \,\text{ i.e.,} \numberthis\label{eq.frItoprocess}\\
    m & = m(0)+\sum_{i=1}^n\int_0^{\boldsymbol{\cdot}} u_i(t) \sh dx_i(t) + \int_0^{\boldsymbol{\cdot}} k(t) \, dt,\displaybreak
\end{align*}
where $m(0) \in \cA_0$, $u_i \in \Lambda^2$ for all $i \in \{1,\ldots,n\}$, and $k \colon \R_+ \to \cA$ is locally Bochner integrable and adapted.
If $w \colon \R_+ \to \cA \potimes \cA^{\op}$ is continuous and adapted (as in Example \ref{ex.elemtens}), then we write
\begin{align*}
    w(t)\sh dm(t) & \coloneqq \sum_{i=1}^n (w(t) \, u_i(t))\sh dx_i(t) + w(t) \sh k(t) \,dt , \, \text{ i.e.,} \\
    \int_0^{\boldsymbol{\cdot}} w(t)\sh dm(t) & \coloneqq w(0) \sh m(0) + \sum_{i=1}^n\int_0^{\boldsymbol{\cdot}}(w(t) \, u_i(t))\sh dx_i(t) + \int_0^{\boldsymbol{\cdot}} w(t) \sh k(t) \,dt,
\end{align*}
where the multiplication $w\,u_i$ occurs in $\cA \wotimes \cA^{\op}$.
If $w = a \otimes b$ for continuous adapted $a,b \colon \R_+ \to \cA$, then we write $a(t)\,dm(t) \,b(t) \coloneqq w(t) \sh dm(t)  = (a(t) \otimes b(t))\sh dm(t)$.
\end{defi}

Note that if $k$ is as above, then $\int_0^{\cdot} k(t) \, dt \colon \R_+ \to \cA$ is adapted because $\cA_t \subseteq \cA$ is norm-closed, for all $t \geq 0$.
In particular, free It\^{o} processes are continuous and adapted.
Also, if $m$ and $w$ are as above, then $w\,u_i \in \Lambda^2$ by Corollary \ref{cor.RCLL}, and $w \sh k \colon \R_+ \to \cA$ is locally Bochner integrable because $k$ is locally Bochner integrable and $\R_+ \ni t \mapsto \#(w(t)) \in B(\cA)$ is continuous.
In particular, both the free stochastic integrals and the Bochner integrals in the second part of the definition above make sense.

Now, suppose that $(z_1,\ldots,z_n) \colon \R_+ \to \cA^n$ is an $n$-dimensional circular Brownian motion.
If $k \colon \R_+ \to \cA$ is locally Bochner integrable and adapted, $u_1,v_1,\ldots,u_n,v_n \in \Lambda^2$, and $m \colon \R_+ \to \cA$ satisfies
\[
dm(t) = \sum_{i=1}^n\big(u_i(t)\sh dz_i(t) + v_i(t)\sh dz_i^*(t)\big) + k(t)\,dt, \numberthis\label{eq.zfrItoprocess}
\]
then $m$ is a free It\^{o} process driven by a $2n$-dimensional semicircular Brownian motion.
Indeed, if $x_i \coloneqq \sqrt{2} \cRe z_i$ and $y_i \coloneqq \sqrt{2} \cIm z_i$, then $(x_1,y_1\ldots,x_n,y_n) \colon \R_+ \to \cA_{\sa}^{2n}$ is a $2n$-dimensional semicircular Brownian motions, and $m$ satisfies
\[
dm(t) = \frac{1}{\sqrt{2}}\sum_{j=1}^n\big((u_j(t)+v_j(t))\sh dx_j(t) + i(u_j(t)-v_j(t))\sh dy_j(t)\big) + k(t)\,dt.
\]
Next, we introduce the operations that show up in the free It\^{o} product rule.

\begin{nota}\label{nota.MtauandQtau}
Let $\mathfrak{m}_{\cA} \colon \cA \otimes \cA \to \cA$ be the linear map induced by multiplication and
\[
\mathcal{M}_{\tau} \coloneqq \mathfrak{m}_{\cA} \circ (\id_{\cA} \otimes \tau \otimes \id_{\cA}) \colon \cA \otimes \cA \otimes \cA \to \cA,
\]
i.e., $\mathcal{M}_{\tau}$ is the linear map determined by $\mathcal{M}_{\tau}(a \otimes b \otimes c) = a\, \tau(b)\,c = \tau(b)\,ac$.
Now, for $u,v \in \cA \otimes \cA^{\op}$, let
\[
Q_{\tau}(u,v) \coloneqq \mathcal{M}_{\tau}((1 \otimes v)\boldsymbol{\cdot}(u \otimes 1)), \numberthis\label{eq.Qtau}
\]
where $\boldsymbol{\cdot}$ is multiplication in $\cA \otimes \cA^{\op} \otimes \cA$, i.e., $Q_{\tau} \colon (\cA \otimes \cA^{\op}) \times (\cA \otimes \cA^{\op}) \to \cA$ is the bilinear map determined by $Q_{\tau}(a \otimes b, c \otimes d) = a\,\tau(bc)\,d$.
\end{nota}

In \cite{bianespeicher}, $\mathcal{M}_{\tau}$ is written as $\eta$, and $Q_{\tau}$ is written as $\la \la \cdot,\cdot \ra\ra$.
Note that, using the universal property of the projective tensor product, we can extend $\mathcal{M}_{\tau}$ to a bounded linear map $\cA^{\potimes 3} \to \cA$ and $Q_{\tau}$ to a bounded bilinear map $(\cA \potimes \cA^{\op})^2 \to \cA$.
Unfortunately, however, the multiplication map $\mathfrak{m}_{\cA} \colon \cA \otimes \cA \to \cA$ is \textit{not} bounded with respect to $\|\cdot\|_{L^{\infty}(\tau \wotimes \tau)}$ (Proposition 3.6 in \cite{deyaschott}), so there is no hope of extending $\mathcal{M}_{\tau}$ to a bounded linear map $\cA \otimes_{\min} \cA \otimes_{\min} \cA \to \cA$, let alone $\cA \wotimes \cA \wotimes \cA \to \cA$. 
Nevertheless, using the following elementary but crucial algebraic observation, we learn that the ``tracing out the middle" in the definition implies that $Q_{\tau}$ can be extended to a bounded bilinear map $(\cA \wotimes \cA^{\op})^2 \to \cA$.

\begin{lem}\label{lem.keyalg}
If $u,v \in \cA \otimes \cA^{\op}$ and $a,b,c,d \in \cA$, then
\[
\tau(a\,\mathcal{M}_{\tau} ((1 \otimes v) \boldsymbol{\cdot} (b \otimes c \otimes d) \boldsymbol{\cdot} (u \otimes 1)) ) = (\tau \otimes \tau^{\op})((a \otimes 1)(b \otimes 1)uv^{\flip}(1 \otimes c)(d \otimes 1)),
\]
where the juxtapositions on the right hand side are multiplications in $\cA \otimes \cA^{\op}$.
\end{lem}
\begin{proof}
It suffices to assume $u = a_1 \otimes b_1$ and $v = c_1 \otimes d_1$ are pure tensors.
In this case, we have
\begin{align*}
    \tau(a\,\mathcal{M}_{\tau} ((1 \otimes v) \boldsymbol{\cdot} (b \otimes c \otimes d) \boldsymbol{\cdot} (u \otimes 1)) ) & = \tau(aba_1\tau(b_1cc_1)d_1d) = (\tau \otimes \tau^{\op})((aba_1d_1d) \otimes (c_1\cdot c \cdot b_1)) \\
    & = (\tau \otimes \tau^{\op})((a \otimes 1)(b\otimes 1)(a_1 \otimes 1)(d_1 \otimes c_1)(d\otimes c)(1 \otimes b_1)) \\
    & = (\tau \otimes \tau^{\op})((1 \otimes b_1)(a \otimes 1)(b \otimes 1)(a_1 \otimes 1)(d_1 \otimes c_1)(1 \otimes c)(d \otimes 1)) \\
    & = (\tau \otimes \tau^{\op})((a \otimes 1)(b \otimes 1)(a_1 \otimes b_1)(d_1 \otimes c_1)(1 \otimes c)(d \otimes 1)),
\end{align*}
where in the second-to-last equality we used traciality of $\tau \otimes \tau^{\op}$.
\end{proof}

In particular, if $u,v \in \cA \otimes \cA^{\op}$ and $a \in \cA$, then
\[
\tau(a \, Q_{\tau}(u,v)) = (\tau \wotimes \tau^{\op})((a \otimes 1)uv^{\flip}). \numberthis\label{eq.Qpredef}
\]
Now, note that the right hand side of Equation \eqref{eq.Qpredef} makes sense for arbitrary $u,v \in \cA \wotimes \cA^{\op}$ and $a \in L^1(\cA,\tau)$.
We may therefore use the duality relationship $L^1(\cA,\tau)^* \cong L^{\infty}(\cA,\tau) = \cA$ to extend the definition of $Q_{\tau}$.
Specifically, if $u,v \in \cA \wotimes \cA^{\op}$, $a \in L^1(\cA,\tau)$, and $\ell_{u,v}(a) \coloneqq (\tau \wotimes \tau^{\op})((a \otimes 1)uv^{\flip})$, then
\[
|\ell_{u,v}(a)| \leq \|(a \otimes 1)uv^{\flip}\|_{L^1(\tau \wotimes \tau^{\op})} \leq \|a \otimes 1\|_{L^1(\tau \wotimes \tau^{\op})}\|uv^{\flip}\|_{L^{\infty}(\tau \wotimes \tau^{\op})} = \|a\|_{L^1(\tau)}\|uv^{\flip}\|_{L^{\infty}(\tau \wotimes \tau^{\op})}.
\]
Thus $\|\ell_{u,v}\|_{L^1(\cA,\tau)^*} \leq \|uv^{\flip}\|_{L^{\infty}(\tau \wotimes \tau^{\op})} < \infty$.
In particular, since $\cA$ is dense in $L^1(\cA,\tau)$, the following definition makes sense and extends the algebraic definition of $Q_{\tau}$.

\begin{defi}[Extended Definition of $Q_{\tau}$]\label{def.Q}
If $u,v \in \cA \wotimes \cA^{\op}$, then we define $Q_{\tau}(u,v)$ to be the unique element of $\cA$ such that
\[
\tau(a \, Q_{\tau}(u,v)) = (\tau \wotimes \tau^{\op})((a \otimes 1)uv^{\flip}),
\]
for all $a \in \cA$ (and thus $a \in L^1(\cA,\tau)$).
\end{defi}

It is clear from the definition that the map $Q_{\tau}(u,v)$ is bilinear in $(u,v)$.
Also, by the paragraph before Definition \ref{def.Q}, if $u,v \in \cA \wotimes \cA^{\op}$, then
\[
\|Q_{\tau}(u,v)\| = \|\ell_{u,v}\|_{L^1(\cA,\tau)^*} \leq \|uv^{\flip}\|_{L^{\infty}(\tau \wotimes \tau^{\op})} \leq \|u\|_{L^{\infty}(\tau \wotimes \tau^{\op})}\|v\|_{L^{\infty}(\tau \wotimes \tau^{\op})}.
\]
Therefore, by the Cauchy--Schwarz Inequality, if $u,v \in L_{\loc}^2(\R_+;\cA \wotimes \cA^{\op})$, then $Q_{\tau}(u,v) \in L_{\loc}^1(\R_+;\cA)$ and
\[
\|Q_{\tau}(u,v)\|_{L_t^1L^{\infty}(\tau)} \leq \|u\|_{L_t^2L^{\infty}(\tau \wotimes \tau^{\op})}\|v\|_{L_t^2L^{\infty}(\tau \wotimes \tau^{\op})}, \numberthis\label{eq.Qbd}
\]
for all $t \geq 0$.
It is then easy to see --- by starting with simple adapted biprocesses and then taking limits --- that if $u,v \in \Lambda^2$, then $Q_{\tau}(u,v) \in L_{\loc}^1(\R_+;\cA)$ is adapted (i.e., has an adapted representative).
This is all the information we need about $Q_{\tau}$, so we are now in a position to state the free It\^{o} product rule.
(However, please see Remark \ref{rem.tech} for additional comments about $Q_{\tau}$.)

\begin{thm}[Free It\^{o} Product Rule]\label{thm.FIPR}
The following formulas hold.
\begin{enumerate}[label=(\roman*),font=\normalfont,leftmargin=2\parindent]
    \item Suppose that $(x_1,\ldots,x_n) \colon \R_+ \to \cA_{\sa}^n$ is an $n$-dimensional semicircular Brownian motion.
    If, for each $\ell \in \{1,2\}$, $m_{\ell} \colon \R_+ \to \cA$ is a free It\^{o} process satisfying $dm_{\ell}(t) = \sum_{i=1}^n u_{\ell i}(t) \sh dx_i(t) + k_{\ell}(t)\,dt$, then
    \[
    d(m_1m_2)(t) = dm_1(t) \,m_2(t) + m_1(t) \, dm_2(t) + \sum_{i=1}^n Q_{\tau}(u_{1i}(t),u_{2i}(t)) \, dt,
    \]
    i.e., $dm_1(t)\,dm_2(t) = \sum_{i=1}^n Q_{\tau}(u_{1i}(t),u_{2i}(t)) \, dt$ in the classical notation.\label{item.FIPRx}
    \item Suppose that $(z_1,\ldots,z_n) \colon \R_+ \to \cA^n$ is an $n$-dimensional circular Brownian motion.
    If, for each $\ell \in \{1,2\}$, $m_{\ell} \colon \R_+ \to \cA$ is a free It\^{o} process (driven by $(z_1,\ldots,z_n)$) satisfying
    \[
    dm_{\ell}(t) = \sum_{i=1}^n \big(u_{\ell i}(t) \sh dz_i(t) + v_{\ell i}(t) \sh dz_i^*(t)\big) + k_{\ell}(t) \,dt,
    \]
    then
    \[
    d(m_1m_2)(t) = dm_1(t) \,m_2(t) + m_1(t) \, dm_2(t) + \sum_{i=1}^n \big(Q_{\tau}(u_{1i}(t),v_{2i}(t)) + Q_{\tau}(v_{1i}(t),u_{2i}(t))\big) \, dt,
    \]
    i.e., $dm_1(t)\,dm_2(t) = \sum_{i=1}^n (Q_{\tau}(u_{1i}(t),v_{2i}(t)) + Q_{\tau}(v_{1i}(t),u_{2i}(t))) \, dt$ in the classical notation.\label{item.FIPRz}
\end{enumerate}
\end{thm}

By the comments following Definition \ref{def.freeItoprocess}, item \ref{item.FIPRz} follows from item \ref{item.FIPRx} with twice the dimension.
Before launching into the proof of item \ref{item.FIPRx}, we perform a useful example calculation.

\begin{ex}\label{ex.FIprod}
Let $z \colon \R_+ \to \cA$ be a circular Brownian motion.
Written in the classical notation for quadratic covariation, Theorem \ref{thm.FIPR}\ref{item.FIPRz} says that
\begin{align}
    a(t)\,dz(t)\,b(t)\,dz^*(t)\,c(t) & = a(t)\,dz^*(t)\,b(t)\,dz(t)\,c(t) = a(t)\,\tau(b(t))\,c(t)\,dt \; \text{ and} \label{eq.fquadcov1} \\
    a(t)\,dz^{\e}(t)\,b(t)\,dz^{\e}(t)\,c(t) & = a(t)\,dz^{\e}(t)\,b(t)\,dt\,c(t) = a(t)\,dt\,b(t)\,dz^{\e}(t) \,c(t) = a(t) \, dt \, b(t) \,dt \, c(t) = 0 \label{eq.fquadcov2}
\end{align}
for $\e \in \{1,\ast\}$ and continuous adapted processes $a,b,c \colon \R_+ \to \cA$.
Now, let $n_1,n_2 \in \N$ be natural numbers, and fix continuous adapted processes $a_1,b_1\ldots,a_{n_1},b_{n_1},c_1,d_1\ldots,c_{n_2},d_{n_2}, k \colon \R_+ \to \cA$. Suppose that $m \colon \R_+ \to \cA$ is a free It\^{o} process satisfying
\[
dm(t) = \sum_{i=1}^{n_1}a_i(t)\,dz(t)\,b_i(t)+\sum_{j=1}^{n_2}c_j(t)\,dz^*(t)\,d_j(t) + k(t)\,dt. \numberthis\label{eq.exfrItoprocess}
\]
Such \hspace{-0.3mm}$m$\hspace{-0.3mm} show \hspace{-0.3mm}up\hspace{-0.3mm} frequently ``in \hspace{-0.3mm}the\hspace{-0.3mm} wild."
It \hspace{-0.3mm}is\hspace{-0.3mm} often necessary \hspace{-0.2mm}---\hspace{-0.2mm} especially when $m$ is not self-adjoint \hspace{-0.2mm}---\hspace{-0.2mm} to work with $|m|^2 \hspace{-0.3mm}=\hspace{-0.3mm} m^*m$.
The free It\^{o} product rule says that $d|m|^2(t) = dm^*(t)\,m(t)+m^*(t)\,dm(t)+dm^*(t)\,dm(t)$.
Let us derive an expression for $dm^*(t)\,dm(t)$.
First, we have
\[
dm^*(t) = \sum_{j=1}^{n_2}d_j^*(t)\,dz(t)\,c_j^*(t)+\sum_{i=1}^{n_1}b_i^*(t)\,dz^*(t)\,a_i^*(t) + k^*(t)\,dt.
\]
Therefore, by the free It\^{o} product rule (in the form of Equations \eqref{eq.fquadcov1}--\eqref{eq.fquadcov2}),
\begin{align*}
    dm^*(t)\,dm(t) & = \Bigg(\sum_{j=1}^{n_2}d_j^*(t)\,dz(t)\,c_j^*(t)+\sum_{i=1}^{n_1}b_i^*(t)\,dz^*(t)\,a_i^*(t)\Bigg)\Bigg(\sum_{i=1}^{n_1}a_i(t)\,dz(t)\,b_i(t)+\sum_{j=1}^{n_2}c_j(t)\,dz^*(t)\,d_j(t)\Bigg) \\
    & = \sum_{j_1,j_2=1}^{n_2}d_{j_1}^*(t)\,\tau\big(c_{j_1}^*(t) \, c_{j_2}(t)\big)\,d_{j_2}(t)\,dt+\sum_{i_1,i_2=1}^{n_1}b_{i_1}^*(t)\,\tau\big(a_{i_1}^*(t)\,a_{i_2}(t)\big)\,b_{i_2}(t)\,dt.
\end{align*}
Now, let $h \in \cA_0$ be arbitrary, and suppose $g \colon \R_+ \to \cA$ satisfies $dg(t) = g(t)\,dz(t)$ and $g(0)=h$, i.e., $g$ is \textbf{free multiplicative Brownian motion} starting at $h$.
Then, letting $\lambda \in \C$ and $g_{\lambda}(t) \coloneqq g(t) - \lambda1 = g(t)-\lambda$, we have $dg_{\lambda}(t) = g(t) \, dz(t)$ and $dg_{\lambda}^*(t) = dz^*(t) \, g^*(t)$.
Therefore, by the formula above, we have
\begin{align*}
    d|g_{\lambda}|^2(t) & = dg_{\lambda}^*(t)\,g_{\lambda}(t) + g_{\lambda}^*(t)\,dg_{\lambda}(t) + dg_{\lambda}^*(t)\,dg_{\lambda}(t) \\
    & = dg_{\lambda}^*(t)\,g_{\lambda}(t) + g_{\lambda}^*(t)\,dg_{\lambda}(t) + 1\tau(g^*(t)g(t))1\, dt \\
    & = dz^*(t)\,g^*(t) \,g_{\lambda}(t) + g_{\lambda}^*(t)\,g(t) \,dz(t) + \tau\big(|g(t)|^2\big)\, dt.
\end{align*}
We shall use this equation in Example \ref{ex.trlog}.
\end{ex}

We now turn to the proof of Theorem \ref{thm.FIPR}\ref{item.FIPRx}.
Our approach is similar to that of Biane and Speicher, though we use less free probabilistic machinery by mimicking a classical approach to calculating the quadratic covariation of It\^{o} processes:
computing a $L^2$-limit of second-order Riemann--Stieltjes-type sums.

\begin{nota}[Partitions]
Let $T > 0$ and $\Pi = \{0 = t_0  < \cdots < t_n = T\}$ be a partition of $[0,T]$.
If $t \in \Pi$, then we write $t_- \in \Pi$ for the member of $\Pi$ to the left of $t$;
more precisely, $(t_0)_- \coloneqq t_0$ and $(t_i)_- \coloneqq t_{i-1}$ whenever $1 \leq i \leq n$.
In addition, if $t \in \Pi$, then $\Delta t \coloneqq t-t_-$, and $|\Pi| \coloneqq \max_{s \in \Pi}\Delta s$.
Finally, if $V$ is a vector space, $F \colon [0,T] \to V$ is a function, and $t \in \Pi$, then we write $\Delta_tF \coloneqq F(t)-F(t_-)$ and $F^{\Pi} \coloneqq \sum_{s \in \Pi} 1_{[s_-,s)}F(s_-) \colon [0,T] \to V$.
\end{nota}

\begin{lem}\label{lem.RSsums}
If $m_1$ and $m_2$ are as in Theorem \ref{thm.FIPR}\ref{item.FIPRx} and $T > 0$, then
\[
L^{\infty}\text{-}\lim_{|\Pi| \to 0}\sum_{t \in \Pi} \big(\Delta_t m_1\big)\big(\Delta_t m_2\big) = m_1(T)m_2(T) - m_1(0) m_2(0) - \int_0^T dm_1(t)\,m_2(t) -\int_0^T m_1(t)\,dm_2(t),
\]
where the limit is over partitions $\Pi$ of $[0,T]$.
\end{lem}
\begin{proof}
If $\Pi$ is a partition of $[0,T]$, then
\begin{align*}
    \delta_T & \coloneqq m_1(T)m_2(T) - m_1(0) m_2(0) = \sum_{t \in \Pi}\big(m_1(t)m_2(t) - m_1(t_-)m_2(t_-)\big) \\
    & = \sum_{t \in \Pi} \big((m_1(t_-)+\Delta_tm_1)(m_2(t_-)+\Delta_tm_2) - m_1(t_-)m_2(t_-)\big) \\
    & = \sum_{t \in \Pi}\big( \big(\Delta_tm_1\big)m_2(t_-)+m_1(t_-)\,\Delta_tm_2 + \big(\Delta_t m_1\big)\big(\Delta_t m_2\big)\big) \\
    & = \int_0^T dm_1(t)\,m_2^{\Pi}(t) + \int_0^Tm_1^{\Pi}(t) \,dm_2(t) + \sum_{t \in \Pi}\big(\Delta_t m_1\big)\big(\Delta_t m_2\big).
\end{align*}
Now, since $m_{\ell}$ is continuous --- and therefore uniformly continuous --- on $[0,T]$, $m_{\ell}^{\Pi} \to m_{\ell}$ uniformly on $[0,T]$ as $|\Pi| \to 0$.
Therefore, by the $L^{\infty}$-BDG Inequality (and the vector-valued Dominated Convergence Theorem), $\int_0^T dm_1(t)\,m_2^{\Pi}(t) \to \int_0^T dm_1(t)\,m_2(t)$ and $\int_0^Tm_1^{\Pi}(t) \,dm_2(t) \to \int_0^Tm_1(t) \,dm_2(t)$ in $\cA$ as $|\Pi| \to 0$.
It then follows from the calculation above that
\[
\sum_{t \in \Pi} \big(\Delta_t m_1\big)\big(\Delta_t m_2\big) \to m_1(T)m_2(T) - m_1(0) m_2(0) - \int_0^T dm_1(t)\,m_2(t) -\int_0^T m_1(t)\,dm_2(t)
\]
in $\cA$ as $|\Pi| \to 0$, as desired.
\end{proof}

\begin{lem}\label{lem.L2lim}
Let $(x_1,\ldots,x_n) \colon \R_+ \to \cA_{\sa}^n$ be an $n$-dimensional semicircular Brownian motion, and fix $s,t \geq 0$ such that $s < t$.
For all $N \in \N$ and $k \in \{0,\ldots,N\}$, define $t_{k,N} \coloneqq \frac{N-k}{N}s + \frac{k}{N}t$.
If $a \in \cA_s$, then
\[
L^2\text{-}\lim_{N \to \infty}\sum_{k=1}^N\big(x_i(t_{k,N})-x_i(t_{k-1,N})\big)a\big(x_j(t_{k,N})-x_j(t_{k-1,N})\big) = (t-s)\,\tau(a)\,\delta_{ij},
\]
for all $i,j \in \{1,\ldots,n\}$.
\end{lem}
\begin{proof}
By writing $a = (a - \tau(a)1)+\tau(a)1$, it suffices to prove the formula when $a$ is centered and when $a=1$.
To this end, write $\Delta_{k,N}x_i \coloneqq x_i(t_{k,N})-x_i(t_{k-1,N})$, and fix $i,j \in \{1,\ldots,n\}$.
First, note that if $k \neq \ell$, then $\cA_s$, $\Delta_{k,N}x_i$, $\Delta_{\ell,N}x_i$ are freely independent;
and if in addition $i \neq j$, then $\cA_s$, $\Delta_{k,N}x_i$, $\Delta_{\ell,N}x_i$, $\Delta_{k,N}x_j$, $\Delta_{\ell,N}x_j$ are freely independent.
(This is because $s = t_{0,N} < t_{k,N}$ when $k \geq 1$.)
Second, recall that $\|x_i(r_1)-x_i(r_2)\| =  2\sqrt{|r_1-r_2|}$ whenever $r_1,r_2 \geq 0$.
Therefore, by definition of free independence, if either 1) $i=j$ and $a \in \{b \in \cA_s : \tau(b) = 0\}$ or 2) $i \neq j$ and $a \in \{b \in \cA_s : \tau(b)=0\} \cup \{1\}$, then
\begin{align*}
    \Bigg\|\sum_{k = 1}^N\Delta_{k,N}x_i\, a\,\Delta_{k,N}x_j\Bigg\|_{L^2(\tau)}^2 & = \sum_{k,\ell = 1}^N\tau\big(\Delta_{k,N}x_j a^* \Delta_{k,N}x_i \Delta_{\ell,N}x_i \, a\,\Delta_{\ell,N}x_j\big) \\
    & = \sum_{k = 1}^N \tau\big(\Delta_{k,N}x_j a^* \Delta_{k,N}x_i \Delta_{k,N}x_i \, a\,\Delta_{k,N}x_j\big) \\
    & \hspace{15mm} + \sum_{k \neq \ell} \tau\big(\Delta_{k,N}x_j a^* \Delta_{k,N}x_i \Delta_{\ell,N}x_i \, a\,\Delta_{\ell,N}x_j\big)\\
    & = \sum_{k = 1}^N \tau\big(\Delta_{k,N}x_j a^* \Delta_{k,N}x_i \Delta_{k,N}x_i \, a\,\Delta_{k,N}x_j\big) \\
    & \leq \|a\|^2\frac{16(t-s)^2}{N} \to 0,
\end{align*}
as $N \to \infty$.
The only case that remains is $i=j$ and $a=1$.
To take care of this case, note that if $k \neq \ell$, then the elements $(\Delta_{k,N}x_i)^2-(t_{k,N}-t_{k-1,N})$ and $(\Delta_{\ell,N}x_i)^2-(t_{\ell,N}-t_{\ell-1,N})$ are freely independent and centered.
Thus
\[
\tau\Big(\big((\Delta_{k,N}x_i)^2-(t_{k,N}-t_{k-1,N})\big)\big((\Delta_{\ell,N}x_i)^2-(t_{\ell,N}-t_{\ell-1,N})\big)\Big) = 0,
\]
from which it follows --- as above --- that
\begin{align*}
    \Bigg\|\sum_{k = 1}^N(\Delta_{k,N}x_i)^2 - (t-s)\Bigg\|_{L^2(\tau)}^2 & = \Bigg\|\sum_{k=1}^N\big((\Delta_{k,N}x_i)^2 - (t_{k,N}-t_{k-1,N})\big)\Bigg\|_{L^2(\tau)}^2 \\
    & = \sum_{k=1}^N\tau\Big(\big((\Delta_{k,N}x_i)^2 - (t_{k,N}-t_{k-1,N})\big)^2\Big) \\
    & = \sum_{k=1}^N(t_{k,N}-t_{k-1,N})^2 =  \frac{(t-s)^2}{N} \to 0
\end{align*}
as $N \to \infty$.
The third equality holds because $x \coloneqq \Delta_{k,N}x_i$ is semicircular with variance $r \coloneqq t_{k,N}-t_{k-1,N}$, so $\tau(x^{2p}) = C_pr^p$ whenever $p \in \N_0$, where $C_p = \frac{1}{p+1}\binom{2p}{p}$ is the $p^{\text{th}}$ Catalan number.
\end{proof}

\begin{proof}[Proof of Theorem \ref{thm.FIPR}\ref{item.FIPRx}]
By the $L^{\infty}$-BDG Inequality, Equation \eqref{eq.Qbd}, and the vector-valued Dominated Convergence Theorem, it suffices to prove the formula when $u_{\ell i} \in \SBa$, for all $\ell \in \{1,2\}$ and $i \in \{1,\ldots,n\}$.
By Lemma \ref{lem.RSsums}, it therefore suffices to prove that if $T > 0$ and $u_{\ell i} \in \SBa$, then
\[
L^2\text{-}\lim_{|\Pi| \to 0}\sum_{t \in \Pi} \big(\Delta_t m_1\big)\big(\Delta_t m_2\big) = \sum_{i=1}^n\int_0^T Q_{\tau}(u_{1i}(t),u_{2i}(t))\,dt,
\]
where the limit is over partitions $\Pi$ of $[0,T]$.
To this end, write $a_{\ell} \coloneqq \int_0^{\boldsymbol{\cdot}}k_{\ell}(t)\,dt$ for $\ell \in \{1,2\}$. Then
\begin{align*}
    \sum_{t \in \Pi} \big(\Delta_t m_1\big)\big(\Delta_t m_2\big) & = \sum_{t \in \Pi} \big(\Delta_t a_1 + \Delta_t(m_1-a_1)\big)\big(\Delta_t a_2 + \Delta_t(m_2-a_2)\big) \\
    & = \sum_{t \in \Pi}\Delta_t(m_1-a_1)\,\Delta_t(m_2-a_2) + \sum_{t \in \Pi} \big(\Delta_ta_1 \big) \big(\Delta_t m_2\big) + \sum_{t \in \Pi} \Delta_t(m_1-a_1)\,\Delta_ta_2.
\end{align*}
Since $\Delta_ta_{\ell} = \int_{t_-}^t k_{\ell}(s)\,ds$ for $t \in \Pi$, we have
\begin{align*}
    \Bigg\|\sum_{t \in \Pi} \big(\Delta_ta_1 \big) \big(\Delta_t m_2\big)\Bigg\| & \leq \max_{s \in \Pi}\|\Delta_s m_2\|\sum_{t \in \Pi} \|\Delta_ta_1\| \leq \max_{s \in \Pi}\|\Delta_s m_2\|\int_0^T\|k_1(t)\|\,dt \to 0 \; \text{ and} \\
    \Bigg\|\sum_{t \in \Pi} \Delta_t(m_1-a_1)\,\Delta_ta_2\Bigg\| & \leq \max_{s \in \Pi}\|\Delta_s(m_1-a_1)\|\sum_{t \in \Pi} \|\Delta_ta_2\| \leq \max_{s \in \Pi}\|\Delta_s(m_1-a_1)\|\int_0^T\|k_2(t)\|\,dt \to 0
\end{align*}
as $|\Pi| \to 0$ because $m_2$ and $m_1-a_1$ are continuous --- and therefore uniformly continuous --- on $[0,T]$.
In particular, if we write $I_i[u] \coloneqq \int_0^{\boldsymbol{\cdot}} u \sh dx_i$ for $u \in \Lambda^2$ and $i \in \{1,\ldots,n\}$, then we have
\[
L^2\text{-}\lim_{|\Pi| \to 0}\sum_{t \in \Pi}\big(\Delta_t m_1\big)\big(\Delta_t m_2\big) = L^2\text{-}\lim_{|\Pi| \to 0}\sum_{t \in \Pi} \sum_{i,j=1}^n \Delta_t(I_i[u_{1i}])\,\Delta_t(I_j[u_{2j}]).
\]
Thus the proof is complete if we can show that
\[
L^2\text{-}\lim_{|\Pi| \to 0}\sum_{t \in \Pi} \Delta_t(I_i[u])\,\Delta_t(I_j[v]) = \delta_{ij}\int_0^TQ_{\tau}(u(t),v(t))\,dt, \numberthis\label{eq.IzIz}
\]
for all $u,v \in \SBa$ and $i,j \in \{1,\ldots,n\}$.
Since Equation \eqref{eq.IzIz} is bilinear in $(u,v)$, it suffices to prove it assuming that $u = 1_{[s_1,t_1)} a \otimes b$ and $v = 1_{[s_2,t_2)} c \otimes d$, where $[s_1,t_1), [s_2,t_2) \subseteq [0,T)$, $a,b \in \cA_{s_1}$, $c,d \in \cA_{s_2}$, and either $[s_1,t_1) \cap [s_2,t_2) = \emptyset$ or $[s_1,t_1) = [s_2,t_2)$.
We take both cases in turn, but we first observe that if $w \in \SBa$, $i \in \{1,\ldots,n\}$, and $t \in \Pi$, then
\[
\Delta_t(I_i[w]) = \int_0^{\infty} \big(1_{[t_-,t)}w\big)\sh dx_i = \int_{t_-}^t w \sh dx_i.
\]
In particular, if $w \equiv 0$ on $[t_-,t)$, then $\Delta_t(I_i[w]) = 0$.

\underline{Case 1}:
$[s_1,t_1) \cap [s_2,t_2) = \emptyset$.
In this case, the observation at the end of the previous paragraph gives immediately that $\sum_{t \in \Pi}\Delta_t(I_i[u])\,\Delta_t(I_j[v]) = 0$ when $|\Pi|$ is sufficiently small.
But also $Q_{\tau}(u,v) \equiv 0$, so Equation \eqref{eq.IzIz} holds.
\pagebreak

\underline{Case 2}:
$[s_1,t_1) = [s_2,t_2) =: [s,t)$.
Fix $N \in \N$, let $\{t_{k,N} : 0 \leq k \leq N\}$ be as in Lemma \ref{lem.L2lim}, and suppose that $\Pi_N$ is a partition on $[0,T]$ such that $\{t_{k,N} : 0 \leq k \leq N\} \subseteq \Pi_N$.
If $|\Pi_N| \to 0$ as $N \to \infty$, then
\begin{align*}
    L^2\text{-}\lim_{|\Pi| \to 0} \sum_{t \in \Pi} \Delta_t(I_i[u])\,\Delta_t(I_j[v]) & = L^2\text{-}\lim_{N \to \infty}\sum_{t \in \Pi_N} \Delta_t(I_i[u])\,\Delta_t(I_j[v]) \\
    & = L^2\text{-}\lim_{N \to \infty}\sum_{k=1}^N a\big(x_i(t_{k,N}) - x_i(t_{k-1,N})\big)bc\big(x_j(t_{k,N}) - x_j(t_{k-1,N})\big)d \\
    & = (t-s)\,a\,\tau(bc)\,d\,\delta_{ij} = \delta_{ij}\int_0^T Q_{\tau}(u(t),v(t))\,dt
\end{align*}
by the observation made just before the previous paragraph, the definition of $I_i$, Lemma \ref{lem.L2lim}, and the definition of $Q_{\tau}$.
This completes the proof.
\end{proof}

\begin{cor}
If $m_1$ and $m_2$ are as in Theorem \ref{thm.FIPR}\ref{item.FIPRx} and $T > 0$, then
\[
L^{\infty}\text{-}\lim_{|\Pi| \to 0}\sum_{t \in \Pi} \big(\Delta_tm_1\big)\big(\Delta_tm_2\big) = \sum_{i=1}^n\int_0^T Q_{\tau}(u_{1i}(t),u_{2i}(t))\,dt,
\]
where the limit is over partitions of $[0,T]$.
\end{cor}
\begin{proof}
Combine Lemma \ref{lem.RSsums} and Theorem \ref{thm.FIPR}.
\end{proof}

\subsection{Noncommutative Derivatives}\label{sec.NCder1}

In this section, we define noncommutative derivatives of various scalar functions.
We begin by defining divided differences and collecting their relevant properties.

\begin{defi}[Divided Differences]\label{def.divdiff}
Let $S \subseteq \C$ and $f \colon S \to \C$ be a function.
Define $f^{[0]} \coloneqq f$ and, for $k \in \N$ and distinct $\lambda_1,\ldots,\lambda_{k+1} \in S$, recursively define
\[
f^{[k]}(\lambda_1,\ldots,\lambda_{k+1}) \coloneqq \frac{f^{[k-1]}(\lambda_1,\ldots,\lambda_k) - f^{[k-1]}(\lambda_1,\ldots,\lambda_{k-1},\lambda_{k+1})}{\lambda_k-\lambda_{k+1}}.
\]
We call $f^{[k]}$ the $\boldsymbol{k^{\textbf{th}}}$ \textbf{divided difference} of $f$.
\end{defi}

\begin{prop}[Properties of Divided Differences]\label{prop.divdiff}
Fix $S \subseteq \C$, functions $f,g \colon S \to \C$, and $k \in \N$.
In addition, write $\Sigma_k \coloneqq \big\{(s_1,\ldots,s_k) \in \R_+^k : s_1+\cdots+s_k \leq 1\big\}$.
\begin{enumerate}[label=(\roman*),font=\normalfont,leftmargin=2\parindent]
   \item $f^{[k]}$ is a symmetric function.\label{item.divdiffsymm}
    \item If $S = \R$ and $f \in C^k(\R)$ or if $S = \C$ and $f \colon \C \to \C$ is entire, then
    \[
    f^{[k]}(\lambda_1,\ldots,\lambda_{k+1}) = \int_{\Sigma_k}f^{(k)}\Bigg(\sum_{j=1}^ks_j\lambda_j+\Bigg(1-\sum_{j=1}^k s_j\Bigg)\lambda_{k+1}\Bigg) \,  ds_1 \cdots ds_k,
    \]
    for all distinct $\lambda_1,\ldots,\lambda_{k+1}$ belonging to $\R$ or $\C$, respectively.
    In particular, if $f \in C^k(\R)$, then $f^{[k]}$ extends uniquely to a (symmetric) continuous function $\R^{k+1} \to \C$;
    and if $f \colon \C \to \C$ is entire, then $f^{[k]}$ extends uniquely to a (symmetric) continuous function $\C^{k+1} \to \C$.
    We shall use the same notation for these extensions.\label{item.divdiffcont}
    \item If $\lambda_1,\ldots,\lambda_{k+1} \in S$ are distinct, then 
    \[
    (fg)^{[k]}(\lambda_1,\ldots,\lambda_{k+1}) = \sum_{j=0}^kf^{[j]}(\lambda_1,\ldots,\lambda_{j+1})\, g^{[k-j]}(\lambda_{j+1},\ldots,\lambda_{k+1}).
    \]
    If $S = \R$ and $f,g \in C^k(\R)$, then the product formula above holds for all $\lambda_1,\ldots,\lambda_{k+1} \in \R$.
    If $S = \C$ and $f,g \colon \C \to \C$ are entire, then the formula holds for all $\lambda_1,\ldots,\lambda_{k+1} \in \C$.\label{item.divdiffprod}
\end{enumerate}
\end{prop}
\begin{proof}[Sketch of proof]
Each item is proven by induction on $k$.
For item \ref{item.divdiffsymm}, one shows that if $\lambda_1,\ldots,\lambda_{k+1} \in S$ are distinct, then $f^{[k]}(\lambda_1,\ldots,\lambda_{k+1}) = \sum_{i=1}^{k+1} f(\lambda_i) \prod_{j \neq i}(\lambda_i-\lambda_j)^{-1}$, which is clearly symmetric in its arguments.
For \ref{item.divdiffcont}, one proceeds from the identity $\frac{f(\lambda)-f(\mu)}{\lambda-\mu} = \int_0^1 f'(t \lambda+(1-t)\mu)\,dt$, which follows in either case from the Fundamental Theorem of Calculus.
(Please see the proof of Proposition 2.1.3(ii) in \cite{nikitopoulosNCk} for details.)
The induction argument for item \ref{item.divdiffprod} proceeds straightforwardly from the definitions.
We encourage the reader to work out the details when $k \in \{1,2\}$, since these are the cases of interest in this paper.
\end{proof}

Next, we work out two important examples of divided differences.

\begin{ex}[Divided Differences of Polynomials]\label{ex.divdiffpoly}
Fix a polynomial $p(\lambda) = \sum_{i=0}^n c_i \lambda^i \in \C[\lambda]$, viewed as an entire function $\C \to \C$.
If $\blambda \coloneqq (\lambda_1,\ldots,\lambda_{k+1}) \in \C^{k+1}$ has distinct entries, then
\[
p^{[k]}(\blambda) = \sum_{i=0}^nc_i\sum_{|\delta| = i-k} \blambda^{\delta} = \sum_{i=0}^nc_i\sum_{\delta \in \N_0^{k+1} : |\delta| = i-k}\lambda_1^{\delta_1}\cdots \lambda_{k+1}^{\delta_{k+1}}, \numberthis\label{eq.divdiffpoly}
\]
where $|\delta| \coloneqq \delta_1+\cdots+\delta_{k+1}$ is the order of $\delta = (\delta_1,\ldots,\delta_{k+1}) \in \N_0^{k+1}$.
(Empty sums are, by convention, zero.)
As is the case with many properties of divided differences, Equation \eqref{eq.divdiffpoly} may be proven by induction on $k$.
Please see Example 2.1.5 in \cite{nikitopoulosNCk}.
By continuity --- i.e., Proposition \ref{prop.divdiff}\ref{item.divdiffcont} --- we have that Equation \eqref{eq.divdiffpoly} holds for \textit{all} $\blambda \in \C^{k+1}$.
In particular, $p^{[k]} \in \C[\lambda_1,\ldots,\lambda_{k+1}]$.
\end{ex}

\begin{ex}[Divided Differences of $W_k$ Functions]\label{ex.divdiffWk}
Suppose that $\mu$ is a Borel complex measure on $\R$ and $f(\lambda) \coloneqq \int_{\R} e^{i\lambda \xi} \,\mu(d\xi)$ for $\lambda \in \R$.
If $\int_{\R} |\xi|^k\,|\mu|(d\xi) < \infty$, in which case $f \in W_k(\R)$ (Definition \ref{def.Wk}), then $f \in C^k(\R)$ and $f^{(k)}(\lambda) = \int_{\R} (i\xi)^ke^{i\lambda\xi} \, \mu(d\xi)$, for all $\lambda \in \R$.
In particular, by Proposition \ref{prop.divdiff}\ref{item.divdiffcont}, we have
\[
\hspace{-0.55mm}f^{[k]}(\blambda) = \int_{\Sigma_k}\int_{\R} (i\xi)^k e^{is_1\lambda_1\xi}\cdots e^{is_k\lambda_k\xi}e^{i(1-\sum_{j=1}^ks_j)\lambda_{k+1}\xi}\,\mu(d\xi) \,ds_1\cdots ds_k, \numberthis\label{eq.divdiffWk}
\]
for all $\blambda = (\lambda_1,\ldots,\lambda_{k+1}) \in \R^{k+1}$.
\end{ex}

We now move on to noncommutative derivatives.
Let $A$ be a unital $\C$-algebra and $\tilde{a}_1,\ldots,\tilde{a}_{k+1} \in A$ be commuting elements.
Then there exists a unique unital algebra homomorphism
\[
\ev_{(\tilde{a}_1,\ldots,\tilde{a}_{k+1})} \colon \C[\lambda_1,\ldots,\lambda_{k+1}] \to A
\]
determined by $\lambda_j \mapsto \tilde{a}_j$, $1 \leq j \leq k+1$.

\begin{defi}[Noncommutative Derivatives of Polynomials]\label{def.NCderalg}
Let $A$ be a unital $\C$-algebra, and fix $\boldsymbol{a} = (a_1,\ldots,a_{k+1}) \in A^{k+1}$.
For $j$ between $1$ and $k+1$, write
\[
\tilde{a}_j \coloneqq 1^{\otimes (j-1)} \otimes a_j \otimes 1^{\otimes(k+1-j)}  \in A^{\otimes(k+1)}.
\]
Now, for $p(\lambda) = \sum_{i=0}^n c_i\lambda^i \in \C[\lambda]$, we define 
\[
\partial^kp(\boldsymbol{a}) \coloneqq k!\ev_{(\tilde{a}_1,\ldots,\tilde{a}_{k+1})}\big(p^{[k]}\big) = k!\,p^{[k]}(\tilde{a}_1,\ldots,\tilde{a}_{k+1}) = k!\sum_{i=0}^n c_i \sum_{|\delta| = i-k} a_1^{\delta_1} \otimes \cdots \otimes a_{k+1}^{\delta_{k+1}} \in A^{\otimes(k+1)} \numberthis\label{eq.polyNCder} 
\]
to be the \textbf{$\boldsymbol{k^{\text{th}}}$ noncommutative derivative} of $p$ evaluated at $\boldsymbol{a}$.
We often write $\partial \coloneqq \partial^1$ and consider $\partial p(a_1,a_2)$ as an element of $A \otimes A^{\op}$.
Finally, write
\[
\partial^kp(a) \coloneqq \partial^kp(\underbrace{a,\ldots,a}_{\mathsmaller{k+1 \,\mathrm{times}}})
\]
for a single element $a \in A$. 
\end{defi}

Now, fix a unital $C^*$-algebra $\cB$.
We use analysis to define $k^{\text{th}}$ noncommutative derivatives (in $\cB$) of $C^k$ functions.
First, note that if $a_1,\ldots,a_{k+1} \in \cB_{\sa}$, then
\[
a_1 \otimes 1^{\otimes k}, \ldots, 1^{\otimes(j-1)} \otimes a_j \otimes 1^{\otimes(k+1-j)},\ldots, 1^{\otimes k} \otimes a_{k+1} \in \cB^{\otimes (k+1)} \subseteq \cB^{\otimes_{\min}(k+1)}
\]
is a list of commuting self-adjoint elements in $\cB^{\otimes_{\min}(k+1)}$ with joint spectrum $\sigma(a_1) \times \cdots \times \sigma(a_{k+1}) \subseteq \R^{k+1}$.
(Please see \cite{ceausescu}.)
The following definition therefore makes sense using multivariate functional calculus.

\begin{defi}[Noncommutative Derivatives of $C^k$ Functions]\label{def.NCdertensor}
For $\boldsymbol{a} = (a_1,\ldots,a_{k+1}) \in \cB_{\sa}^{k+1}$ and $f \in C^k(\R)$, we define 
\[
\partial^kf(\boldsymbol{a}) \coloneqq k!\,f^{[k]}\big(a_1 \otimes 1^{\otimes k},\ldots,1^{\otimes k} \otimes a_{k+1}\big) \in \cB^{\otimes_{\min}(k+1)}
\]
to be the \textbf{$\boldsymbol{k^{\text{th}}}$ noncommutative derivative} of $f$ evaluated at $\boldsymbol{a}$.
As before, we often write $\partial \coloneqq \partial^1$ and consider $\partial f(a_1,a_2)$ as an element of $\cB \otimes_{\min} \cB^{\op}$.
Also, write
\[
\partial^kf(a) \coloneqq \partial^kf(\underbrace{a,\ldots,a}_{\mathsmaller{k+1 \,  \mathrm{times}}}) 
\]
for a single element $a \in \cB_{\sa}$.
\end{defi}

Of course, if we view $\cB^{\otimes(k+1)}$ as a subalgebra of $\cB^{\otimes_{\min}(k+1)}$, Definition \ref{def.NCdertensor} agrees with Definition \ref{def.NCderalg} when $f = p \in \C[\lambda]$.
We end this section with an important example.

\begin{ex}[Noncommutative Derivatives of $W_k$ Functions]\label{ex.NCderWk}
Let $\mu$ and $f$ be as in Example \ref{ex.divdiffWk}, and suppose again that $\int_{\R} |\xi|^k \, |\mu|(d\xi) < \infty$.
If $\boldsymbol{a} = (a_1,\ldots,a_{k+1}) \in \cB_{\sa}^{k+1}$, it follows from Equation \eqref{eq.divdiffWk} that
\[
\partial^kf(\boldsymbol{a}) = k!\int_{\Sigma_k}\int_{\R} (i\xi)^k e^{is_1\xi \,a_1} \otimes \cdots \otimes e^{is_k\xi \,a_k} \otimes e^{i(1-\sum_{j=1}^k s_j)\xi\,a_{k+1}}\mu(d\xi) \,ds_1\cdots ds_k
\]
where the above is an iterated Bochner integral in $\cB^{\otimes_{\min}(k+1)}$.
When $k=1$, we note for later use that actually $\partial f(a_1,a_2) = i\int_0^1 \int_{\R}\xi \,e^{ita_1} \otimes e^{i(1-t)a_2} \, \mu(d\xi)\,dt$ is an iterated Bochner integral in $\cB \potimes \cB^{\op} \subseteq \cB \otimes_{\min} \cB^{\op}$ (with respect to $\|\cdot\|_{\cB \potimes \cB^{\op}}$) because the map $[0,1] \times \R \ni (t,\xi) \mapsto \xi \,e^{i ta_1} \otimes e^{i (1-t) a_2} \in \cB \potimes \cB^{\op}$ is continuous.
\end{ex}

\subsection{Functional Free It\^{o} Formula for Polynomials}\label{sec.FFIFpoly}

In this section, we prove the ``functional" It\^{o} formula for polynomials of free It\^{o} processes (Theorem \ref{thm.FFIFpoly}).
We begin by defining the object that appears in the correction term.

\begin{nota}\label{nota.Deltap}
For $p \in \C[\lambda]$, $m \in \cA$, and $u,v \in \cA \otimes \cA^{\op}$, write
\[
\Delta_{u,v}p(m) \coloneqq \frac{1}{2}\mathcal{M}_{\tau}((1 \otimes v) \boldsymbol{\cdot} \partial^2p(m) \boldsymbol{\cdot} (u \otimes 1) + (1 \otimes u) \boldsymbol{\cdot} \partial^2p(m) \boldsymbol{\cdot} (v \otimes 1)), \numberthis\label{eq.Deltap}
\]
where $\boldsymbol{\cdot}$ is multiplication in $\cA \otimes \cA^{\op} \otimes \cA$.
\end{nota}

As was the case when we defined $Q_{\tau}$, we can still make sense of the formula defining $\Delta_{u,v}p(m)$ when $u,v \in \cA \potimes \cA^{\op}$.
And again, though the formula does not make sense as written when $u,v \in \cA \otimes_{\min} \cA^{\op}$ (let alone $u,v \in \cA \wotimes \cA^{\op}$), we can use Lemma \ref{lem.keyalg} to extend $\Delta_{\cdot,\cdot}p(m) \colon (\cA \otimes \cA^{\op})^2 \to \cA$ to a bounded bilinear map $(\cA \wotimes \cA^{\op})^2 \to \cA$.
At this time, we advise the reader to review Notation \ref{nota.alg}\ref{item.tenshash}, as we shall henceforth make heavy use of the $\#_2^{\mathsmaller{\otimes}}$ operation defined therein.

Fix $p \in \C[\lambda]$ and $m \in \cA$.
For $a \in L^1(\cA,\tau)$ and $u,v \in \cA \wotimes \cA^{\op}$, define
\[
\ell_{p,u,v}(a) \coloneqq \frac{1}{2}(\tau \wotimes \tau^{\op})\big((a \otimes 1)\, \partial^2p(m \otimes 1, 1 \otimes m, m \otimes 1)\sh_2^{\mathsmaller{\otimes}}[uv^{\flip}+vu^{\flip},1 \otimes 1]\big).
\]
If $a \in \cA$ and $u,v \in \cA \otimes \cA^{\op}$, then Lemma \ref{lem.keyalg} and Equation \eqref{eq.polyNCder} imply
\[
\tau\big(a\,\Delta_{u,v}p(m)\big) = \ell_{p,u,v}(a). \numberthis\label{eq.Deltappredef}
\]
We use this equation to extend the definition of $\Delta_{\cdot,\cdot}p(m)$.
Indeed, note that if $u,v \in \cA \wotimes \cA^{\op}$, then
\[
\|\ell_{p,u,v}\|_{L^1(\cA,\tau)^*} \leq \frac{1}{2}\big\|\partial^2p(m \otimes 1,1 \otimes m, m \otimes 1)\sh_2^{\mathsmaller{\otimes}}[uv^{\flip}+vu^{\flip},1 \otimes 1]\big\|_{L^{\infty}(\tau \wotimes \tau^{\op})}.
\]
Thus, by the duality relationship $L^1(\cA,\tau)^* \cong \cA$, the following definition makes sense and extends the algebraic definition of $\Delta_{u,v}p(m)$.

\begin{defi}[Extended Definition of $\Delta_{u,v}p(m)$]\label{def.Deltap}
For a polynomial $p \in \C[\lambda]$, an element $m \in \cA$, and tensors $u,v \in \cA \wotimes \cA^{\op}$, we define $\Delta_{u,v}p(m)$ to be the unique element of $\cA$ such that
\[
\tau\big(a\Delta_{u,v}p(m)\big) = \ell_{p,u,v}(a),
\]
for all $a \in \cA$ (and thus $a \in L^1(\cA,\tau)$).
Also, we write $\Delta_up(m) \coloneqq \Delta_{u,u}p(m)$.
\end{defi}
\pagebreak

It is clear from the definition that $\Delta_{u,v}p(m)$ is trilinear in $(u,v,p)$ and symmetric in $(u,v)$.
Also, if $n \in \N_0$ and $p_n(\lambda) = \lambda^n$, then, by Equation \eqref{eq.polyNCder} and the paragraph before Definition \ref{def.Deltap}, we have
\begin{align*}
    \|\Delta_{u,v}p_n(m)\| & = \|\ell_{p_n,u,v}\|_{L^1(\cA,\tau)^*} \leq \frac{1}{2}\Bigg\|2\sum_{|\delta|=n-2}(m \otimes 1)^{\delta_1}(uv^{\flip}+vu^{\flip})(1 \otimes m)^{\delta_2}(m \otimes 1)^{\delta_3}\Bigg\|_{L^{\infty}(\tau \wotimes \tau^{\op})} \\
    & \leq 2\|u\|_{L^{\infty}(\tau \wotimes \tau^{\op})}\|v\|_{L^{\infty}(\tau \wotimes \tau^{\op})}\sum_{|\delta|=n-2}\|m \otimes 1\|_{L^{\infty}(\tau \wotimes \tau^{\op})}^{\delta_1}\|1 \otimes m\|_{L^{\infty}(\tau \wotimes \tau^{\op})}^{\delta_2}\|m \otimes 1\|_{L^{\infty}(\tau \wotimes \tau^{\op})}^{\delta_3} \\
    & = n(n-1)\|m\|^{n-2}\|u\|_{L^{\infty}(\tau \wotimes \tau^{\op})}\|v\|_{L^{\infty}(\tau \wotimes \tau^{\op})}. \numberthis\label{eq.Deltapbd}
\end{align*}
Therefore, if $u,v \in L_{\loc}^2(\R_+;\cA \wotimes \cA^{\op})$ and $m \in C(\R_+;\cA)$, then we have that $\Delta_{u,v}p(m) \in L_{\loc}^1(\R_+;\cA)$ and
\[
\|\Delta_{u,v}p_n(m)\|_{L_t^1L^{\infty}(\tau)} \leq n(n-1) \|m\|_{L_t^{\infty}L^{\infty}(\tau)}^{n-2}\|u\|_{L_t^2L^{\infty}(\tau \wotimes \tau^{\op})}\|v\|_{L_t^2L^{\infty}(\tau \wotimes \tau^{\op})},
\]
for all $t \geq 0$.
It is then easy to see that if $u,v \in \Lambda^2$ and $m \colon \R_+ \to \cA$ is continuous and adapted, then $\Delta_{u,v}p(m) \in L_{\loc}^1(\R_+;\cA)$ is adapted as well.
The last fact we shall need about $\Delta_{u,v}p(m)$ to prove the functional free It\^{o} formula for polynomials is the following product rule.
(However, please see Remark \ref{rem.tech} for additional comments about $\Delta_{u,v}p(m)$.)

\begin{lem}[Product Rule for $\Delta_{u,v}p(m)$]\label{lem.Deltaprodrule}
If $p,q \in \C[\lambda]$, then
\[
\Delta_{u,v}(pq)(m) = \Delta_{u,v}p(m)\, q(m) + p(m)\,\Delta_{u,v}q(m) + Q_{\tau}(\partial p(m) \, u, \partial q(m) \, v) +  Q_{\tau}(\partial p(m) \, v,\partial q(m) \, u),
\]
for all $m \in \cA$ and $u,v \in \cA \wotimes \cA^{\op}$.
\end{lem}
\begin{proof}
By Proposition \ref{prop.divdiff}\ref{item.divdiffprod} and the definition of $\partial^2$, if $A$ is a unital $\C$-algebra and $p,q \in \C[\lambda]$, then
\begin{align*}
    \partial^2(pq)(a_1,a_2,a_3) & = \partial^2p(a_1,a_2,a_3) (1 \otimes 1 \otimes q(a_3)) \\
    & \hspace{10mm} + (p(a_1) \otimes 1 \otimes 1) \partial^2q(a_1,a_2,a_3) \\
    & \hspace{10mm} + 2(\partial p(a_1,a_2) \otimes 1)(1 \otimes \partial q(a_2,a_3)),
\end{align*}
for all $a_1,a_2,a_3 \in A$.
Applying this to the algebra $A = \cA \wotimes \cA^{\op}$ and writing $\boldsymbol{1} = 1 \otimes 1$ for the identity in $\cA \wotimes \cA^{\op}$ to avoid confusion, we have
\begin{align*}
    \partial^2(pq)(m \otimes 1, 1 \otimes m, m \otimes 1) & = \partial^2p(m \otimes 1, 1 \otimes m, m \otimes 1) (\boldsymbol{1} \otimes \boldsymbol{1} \otimes q(m \otimes 1)) \\
    & \hspace{10mm} + (p(m \otimes 1) \otimes \boldsymbol{1} \otimes \boldsymbol{1}) \partial^2q(m \otimes 1, 1 \otimes m, m \otimes 1) \\
    & \hspace{10mm} + 2(\partial p(m \otimes 1 ,1\otimes m) \otimes \boldsymbol{1})(\boldsymbol{1} \otimes \partial q(1 \otimes m,m \otimes 1)),
\end{align*}
for all $m \in \cA$.
Now, notice that if $u_1,u_2 \in \cA \wotimes \cA^{\op}$ and $A \in (\cA \wotimes \cA^{\op})^{\otimes 3}$, then
\[
((u_1 \otimes \boldsymbol{1} \otimes \boldsymbol{1}) A (\boldsymbol{1} \otimes \boldsymbol{1} \otimes u_2))\sh_2^{\mathsmaller{\otimes}}[c,d] = u_1 (A \sh_2^{\mathsmaller{\otimes}}[c,d])u_2.
\]
Since $p(m \otimes 1) = p(m) \otimes 1$ and $q(m \otimes 1) = q(m) \otimes 1$, it follows from the above that if $a \in \cA$, then
\begin{align*}
    \tau\big(a \,\Delta_{u,v}(pq)(m)\big) & = \frac{1}{2}(\tau \wotimes \tau^{\op})\big((a \otimes 1) \big(\partial^2p(m \otimes 1, 1 \otimes m, m \otimes 1) \sh_2^{\mathsmaller{\otimes}}[uv^{\flip}+vu^{\flip},\boldsymbol{1}]\big)(q(m) \otimes 1)\big) \\
    & \hspace{4.5mm} +\frac{1}{2}(\tau \wotimes \tau^{\op})\big((a \otimes 1)(p(m) \otimes 1) \partial^2q(m \otimes 1, 1 \otimes m, m \otimes 1) \sh_2^{\mathsmaller{\otimes}}[uv^{\flip}+vu^{\flip},\boldsymbol{1}]\big)  \\
    & \hspace{4.5mm} + (\tau \wotimes \tau^{\op})\big((a\otimes 1)((\partial p(m \otimes 1 ,1\otimes m) \otimes \boldsymbol{1})(\boldsymbol{1} \otimes \partial q(1 \otimes m,m \otimes 1)))\sh_2^{\mathsmaller{\otimes}}[uv^{\flip}+vu^{\flip},\boldsymbol{1}]\big) \\
    & = \frac{1}{2}(\tau \wotimes \tau^{\op})\big(((q(m)\,a) \otimes 1) \partial^2p(m \otimes 1, 1 \otimes m, m \otimes 1) \sh_2^{\mathsmaller{\otimes}}[uv^{\flip}+vu^{\flip},\boldsymbol{1}]\big) \\
    & \hspace{4.5mm} +\frac{1}{2}(\tau \wotimes \tau^{\op})\big(((a\,p(m)) \otimes 1)\partial^2q(m \otimes 1, 1 \otimes m, m \otimes 1) \sh_2^{\mathsmaller{\otimes}}[uv^{\flip}+vu^{\flip},\boldsymbol{1}]\big) + R_a \\
    & = \underbrace{\tau\big(q(m)\,a\,\Delta_{u,v}p(m)\big)}_{\tau(a\,\Delta_{u,v}p(m)\,q(m))} + \tau\big(a \,p(m)\,\Delta_{u,v}q(m)\big) + R_a,
\end{align*}
where
\[
R_a = (\tau \wotimes \tau^{\op})\big((a\otimes 1)((\partial p(m \otimes 1 ,1\otimes m) \otimes \boldsymbol{1})(\boldsymbol{1} \otimes \partial q(1 \otimes m,m \otimes 1)))\sh_2^{\mathsmaller{\otimes}}[uv^{\flip}+vu^{\flip},\boldsymbol{1}]\big).\pagebreak
\]
But now, note that if $P_1(\lambda_1,\lambda_2) = \lambda_1^{\gamma_1}\lambda_2^{\gamma_2}$, $P_2(\lambda_1,\lambda_2) = \lambda_1^{\delta_1}\lambda_2^{\delta_2}$, $u_1 = m \otimes 1$, and $u_2 = 1 \otimes m$, then
\begin{align*}
    & (\tau \wotimes \tau^{\op})\big((a \otimes 1)((P_1(u_1 \otimes \boldsymbol{1} , \boldsymbol{1} \otimes u_2) \otimes \boldsymbol{1})(\boldsymbol{1} \otimes P_2(u_2 \otimes \boldsymbol{1},\boldsymbol{1} \otimes u_1)))\sh_2^{\mathsmaller{\otimes}}[uv^{\flip},\boldsymbol{1}]\big)\\
    & \hspace{10mm} = (\tau \wotimes \tau^{\op})\big((a \otimes 1)u_1^{\gamma_1}uv^{\flip}u_2^{\gamma_2}u_2^{\delta_1}u_1^{\delta_2}\big) = (\tau \wotimes \tau^{\op})\big((a \otimes 1)u_1^{\gamma_1}u_2^{\gamma_2}u(u_1^{\delta_1}u_2^{\delta_2}v)^{\flip}\big) \\
    & \hspace{10mm} = (\tau \wotimes \tau^{\op})\big((a \otimes 1)\,P_1(m \otimes 1,1 \otimes m)u(P_2(m \otimes 1,1 \otimes m)v)^{\flip}\big)
\end{align*}
by traciality of $\tau \wotimes \tau^{\op}$, the fact that $u_2 = 1 \otimes m$ commutes with both $a \otimes 1$ and $u_1 = m \otimes 1$, and the identity $u_1^{\flip} = u_2$.
By linearity, the above formula holds for all $P_1,P_2 \in \C[\lambda_1,\lambda_2]$.
Applying the formula to $P_1 = p^{[1]}$ and $P_2 = q^{[1]}$ gives
\begin{align*}
    R_a & = (\tau \wotimes \tau^{\op})\big((a \otimes 1) \partial p(m)u(\partial q(m)v)^{\flip}\big) + (\tau \wotimes \tau^{\op})\big((a \otimes 1) \partial p(m)v(\partial q(m)u)^{\flip}\big) \\
    & = \tau(a \,Q_{\tau}(\partial p(m)\,u,\partial q(m)\,v)) + \tau(a \,Q_{\tau}(\partial p(m)\,v,\partial q(m)\,u)).
\end{align*}
This completes the proof.
\end{proof}

This, together with the free It\^{o} product rule, gives the functional free It\^{o} formula for polynomials.

\begin{thm}[Functional Free It\^{o} Formula for Polynomials]\label{thm.FFIFpoly}
Fix a polynomial $p \in \C[\lambda]$.
\begin{enumerate}[label=(\roman*),font=\normalfont,leftmargin=2\parindent]
    \item Suppose that $(x_1,\ldots,x_n) \colon \R_+ \to \cA_{\sa}^n$ is an $n$-dimensional semicircular Brownian motion.
    If $m$ is a free It\^{o} process satisfying Equation \eqref{eq.frItoprocess}, then $d\,p(m(t)) = \partial p(m(t))\sh dm(t) + \frac{1}{2}\sum_{i=1}^n\Delta_{u_i(t)}p(m(t))\,dt$.\label{item.FFIFpolyx}
    \item Suppose that $(z_1,\ldots,z_n) \colon \R_+ \to \cA^n$ is an $n$-dimensional circular Brownian motion.
    If $m$ is a free It\^{o} process satisfying Equation \eqref{eq.zfrItoprocess}, then $d\,p(m(t)) = \partial p(m(t))\sh dm(t) + \sum_{i=1}^n\Delta_{u_i(t),v_i(t)}p(m(t))\,dt$.\label{item.FFIFpolyz}
\end{enumerate}
\end{thm}
\begin{rem}
In either case, the map $\R_+ \ni t \mapsto \partial p(m(t)) \in \cA \potimes \cA^{\op}$ is clearly continuous and adapted.
In particular, if $\ell \in L_{\loc}^1(\R_+;\cA)$ and $u \in \Lambda^2$, then $\partial p(m) \sh \ell \in L_{\loc}^1(\R_+;\cA)$ and, by Corollary \ref{cor.RCLL}, $\partial p(m) \, u \in \Lambda^2$.
Thus all of the integrals in the statement of Theorem \ref{thm.FFIFpoly} make sense.
\end{rem}
\begin{proof}
Using the comments after Definition \ref{def.freeItoprocess}, it is easy to see that item \ref{item.FFIFpolyz} follows from item \ref{item.FFIFpolyx} with twice the dimension.
It therefore suffices to prove item \ref{item.FFIFpolyx}.
To this end, let $p,q \in \C[\lambda]$ be polynomials, and suppose that the formula in item \ref{item.FFIFpolyx} holds for both $p$ and $q$.
Then the free It\^{o} product rule (Theorem \ref{thm.FIPR}), Proposition \ref{prop.divdiff}\ref{item.divdiffprod}, the definition of $\partial$, and Lemma \ref{lem.Deltaprodrule} give
\begin{align*}
    d\,(pq)(m(t)) & = d\,p(m(t)) \; q(m(t)) + p(m(t)) \; d\,q(m(t)) + d\,p(m(t))\;d\,q(m(t)) \\
    & = \big((1 \otimes q(m(t))) \partial p(m(t)) + (p(m(t)) \otimes 1)\partial q(m(t))\big) \sh dm(t) \\
    & \hspace{10mm} +\sum_{i=1}^n\Big(\frac{1}{2} \big(\Delta_{u_i(t)}p(m(t))\,q(m(t)) + p(m(t))\Delta_{u_i(t)}q(m(t))\big) \\
    & \hspace{30mm} + Q_{\tau}\big(\partial p(m(t)) \, u_i(t), \partial q(m(t)) \, u_i(t)\big)\Big)\,dt \\
    & = \partial (pq)(m(t))\sh dm(t) + \frac{1}{2}\sum_{i=1}^n\Delta_{u_i(t)}(pq)(m(t))\,dt.
\end{align*}
Thus the formula of interest holds for the polynomial $pq$ as well.

Next, note that the formula holds trivially for $p(\lambda) = p_0(\lambda) \equiv 1$ and $p(\lambda) = p_1(\lambda) = \lambda$.
Now, let $n \geq 1$, and assume the formula holds for $p(\lambda) = p_n(\lambda) = \lambda^n$.
By what we just proved, this implies the formula holds for $p(\lambda) = p_n(\lambda)p_1(\lambda) = \lambda^{n+1} = p_{n+1}(\lambda)$.
By induction, the formula holds for $p = p_n$, for all $n \in \N_0$.
Since $\{p_n : n \in \N_0\}$ is a basis for $\C[\lambda]$, we are done.
\end{proof}

\subsection{The Traced Formula}\label{sec.trFFIF}

From Theorem \ref{thm.FFIFpoly} and a symmetrization argument, we obtain a highly useful ``traced" formula.
Before stating it, giving examples, and proving it, we present a rigorous proof of a ``folklore" characterization of when a free It\^{o} process is self-adjoint.

\begin{prop}\label{prop.frItoprocess0}
Suppose that $(x_1,\ldots,x_n) \colon \R_+ \to \cA_{\sa}^n$ is an $n$-dimensional semicircular Brownian motion.
For each $\ell \in \{1,2\}$, let $m_{\ell}$ be a free It\^{o} process satisfying $dm_{\ell}(t) = \sum_{i=1}^n u_{\ell i}(t) \sh dx_i(t) + k_{\ell}(t)\,dt$.
Then $m_1 = m_2$ if and only if $m_1(0)=m_2(0)$, $k_1 = k_2$ a.e., and $u_{1i} = u_{2i}$ a.e. for all $i$.
\end{prop}
\pagebreak
\begin{proof}
Let $m$ be a free It\^{o} process satisfying Equation \eqref{eq.frItoprocess}.
It suffices to show that $m \equiv 0$ if and only if $m(0) = 0$, $k \equiv 0$ almost everywhere, and $u_1 \equiv \cdots \equiv u_n \equiv 0$ almost everywhere.
The ``if" direction is obvious.
For the converse, notice that if $m \equiv 0$, then 
\[
0=dm^*(t) = \sum_{i=1}^nu_i^{\mathsmaller{\bigstar}}(t)\sh dx_i(t) + k^*(t)\,dt,
\]
so that
\[
0= d(mm^*)(t) = dm(t)\,m^*(t) + m(t)\,dm^*(t) + \sum_{i=1}^nQ_{\tau}(u_i(t),u_i^{\mathsmaller{\bigstar}}(t)) \,dt = \sum_{i=1}^nQ_{\tau}(u_i(t),u_i^{\mathsmaller{\bigstar}}(t)) \,dt
\]
by the free It\^{o} product rule.
In other words, $\int_0^t \sum_{i=1}^nQ_{\tau}(u_i(s),u_i^{\mathsmaller{\bigstar}}(s)) \,ds = 0$, for all $t \geq 0$.
By, for instance, the (vector-valued) Lebesgue Differentiation Theorem, this implies $\sum_{i=1}^nQ_{\tau}(u_i(t),u_i^{\mathsmaller{\bigstar}}(t))= 0$ for almost every $t \geq 0$.
We claim that this implies $u_1 \equiv \cdots \equiv u_n \equiv 0$ almost everywhere.
Indeed, if $u \in \cA \wotimes \cA^{\op}$ is arbitrary, then, by definition of $Q_{\tau}$,
\[
\tau(Q_{\tau}(u,u^{\mathsmaller{\bigstar}})) = (\tau \wotimes \tau^{\op})(u(u^{\mathsmaller{\bigstar}})^{\flip}) = (\tau \wotimes \tau^{\op})(uu^*) = (\tau \wotimes \tau^{\op})(u^*u) = \|u\|_{L^2(\tau \wotimes \tau^{\op})}^2.
\]
Our claim is then proven by an appeal to the faithfulness of $\tau \wotimes \tau^{\op}$.
We are left with $\int_0^t k(s)\,ds = 0$ for all $t \geq 0$.
Again from the Lebesgue Differentiation Theorem, we conclude that $k \equiv 0$ almost everywhere.
\end{proof}

\begin{cor}\label{cor.safrItoprocess}
A free It\^{o} process $m$ as in Equation \eqref{eq.frItoprocess} satisfies $m^* = m$ if and only if $m(0)^*=m(0)$, $k^*=k$ a.e., and $u_i^{\mathsmaller{\bigstar}} = u_i$ a.e. for all $i$.
Also, a free It\^{o} process $m$ as in Equation \eqref{eq.zfrItoprocess} satisfies $m^* = m$ if and only if $m(0)^*=m(0)$, $k^*=k$ a.e., and $u_i^{\mathsmaller{\bigstar}} = v_i$ a.e. for all $i$.
\end{cor}

We now state the traced formula.

\begin{thm}[Traced Functional Free It\^{o} Formula]\label{thm.trFFIF}
The following formulas hold.
\begin{enumerate}[label=(\roman*),font=\normalfont,leftmargin=2\parindent]
    \item Suppose that $(x_1,\ldots,x_n) \colon \R_+ \to \cA_{\sa}^n$ is an $n$-dimensional semicircular Brownian motion.
    If $m$ is a free It\^{o} process satisfying Equation \eqref{eq.frItoprocess} and $f \in \C[\lambda]$, then
    \[
    \tau(f(m)) = \tau(f(m(0))) + \int_0^{\boldsymbol{\cdot}}\Big(\tau\big(f'(m(t))\,k(t)\big) + \frac{1}{2}\sum_{i=1}^n(\tau \wotimes \tau^{\op})\big(u_i^{\flip}(t)\,\partial f'(m(t))\,u_i(t)\big)\Big) \, dt. \numberthis\label{eq.trFFIFx}
    \]
    If $m^*=m$ (i.e., $m(0)^*=m(0)$, $k^*=k$ a.e., and $u_i^{\mathsmaller{\bigstar}} = u_i$ a.e. for all $i$), then Equation \eqref{eq.trFFIFx} holds for any $f \colon \R \to \C$ that is $C^2$ in a neighborhood of the closure of $\bigcup_{t \geq 0}\sigma(m(t))$.\label{item.trFFIFx}
    \item Suppose that $(z_1,\ldots,z_n) \colon \R_+ \to \cA^n$ is an $n$-dimensional circular Brownian motion.
    If $m$ is a free It\^{o} process satisfying Equation \eqref{eq.zfrItoprocess} and $f \in \C[\lambda]$, then
    \[
    \tau(f(m)) = \tau(f(m(0))) + \int_0^{\boldsymbol{\cdot}}\Big(\tau\big(f'(m(t))\,k(t)\big) + \sum_{i=1}^n(\tau \wotimes \tau^{\op})\big(v_i^{\flip}(t)\,\partial f'(m(t))\,u_i(t)\big)\Big) \, dt. \numberthis\label{eq.trFFIFz}
    \]
    If $m^*=m$ (i.e., $m(0)^*=m(0)$, $k^*=k$ a.e., and $u_i^{\mathsmaller{\bigstar}} = v_i$ a.e. for all $i$), then Equation \eqref{eq.trFFIFz} holds for any $f \colon \R \to \C$ that is $C^2$ in a neighborhood of the closure of $\bigcup_{t \geq 0}\sigma(m(t))$.\label{item.trFFIFz}
\end{enumerate}
\end{thm}
\begin{rem}\label{rem.law}
Let $m$ be as in Equation \eqref{eq.frItoprocess}.
Note that when $m^*=m$ and $f \colon \R \to \C$ is $C^2$ on a neighborhood of the closure of $\bigcup_{t \geq 0} \sigma(m(t))$, we have $(\tau \wotimes \tau^{\op})(u_i^{\flip}\,\partial f'(m)\,u_i) = \la \partial f'(m)\,u_i,u_i \ra_{L^2(\tau \wotimes \tau^{\op})}$ because $u_i^{\flip} = u_i^*$.
By the functional-calculus-based definition of $\partial f'(m)$, we may therefore read Equation \eqref{eq.trFFIFx} (almost everywhere) more pleasantly as
\[
\frac{d}{dt} \tau(f(m(t))) = \tau\big(f'(m(t))\,k(t)\big) + \frac{1}{2}\sum_{i=1}^n\int_{\R^2}\frac{f'(\lambda)-f'(\mu)}{\lambda-\mu}\,\rho_{m(t),u_i(t)}(d\lambda,d\mu),
\]
where $\rho_{m,u_i}(d\lambda,d\mu) \coloneqq \la P^{m \otimes 1, 1 \otimes m}(d\lambda,d\mu)\,u_i,u_i\ra_{L^2(\tau \wotimes \tau^{\op})}$.
Here, $P^{m \otimes 1, 1 \otimes m}$ is the projection-valued joint spectral measure of the pair $(m \otimes 1, 1 \otimes m)$.
Similar comments apply to Equation \eqref{eq.trFFIFz}.
\end{rem}

Before proving this theorem, we demonstrate its utility.
\pagebreak

\begin{ex}\label{ex.trlog}
Fix a circular Brownian motion $z \colon \R_+ \to \cA$.
Also, let $a_1,b_1\ldots,a_n,b_n,k \colon \R_+ \to \cA$ be continuous adapted processes and $m \colon \R_+ \to \cA$ be a free It\^{o} process satisfying
\[
dm(t) = \sum_{i=1}^n (a_i(t)\,dz(t)\,b_i(t)+c_i(t)\,dz^*(t)\,d_i(t)) + k(t)\,dt.
\]
Now, suppose in addition that $m \geq 0$ (i.e., $m^*=m$ and $\sigma(m(t)) \subseteq \R_+$ whenever $t \geq 0$).
For example, if $\widetilde{m}$ is as in Equation \eqref{eq.exfrItoprocess} and $m \coloneqq |\widetilde{m}|^2 = \widetilde{m}^*\widetilde{m}$, then, as is shown in Example \ref{ex.FIprod}, $m$ is a free It\^{o} process of the form we have just described.

Now, let $\e > 0$, and define $f_{\e}(\lambda) \coloneqq \log(\lambda +\e)$ for $\lambda > -\e$ and $f_{\e} \equiv 0$ on $(-\infty,-\e]$.
Then $f_{\e} \in C^{\infty}((-\e,\infty))$ and $\bigcup_{t \geq 0}\sigma(m(t))\subseteq \R_+ \subseteq (-\e,\infty)$.
Also, if $\lambda,\mu > -\e$, then
\[
f_{\e}'(\lambda) = \frac{1}{\lambda+\e} \; \text{ and } \; (f_{\e}')^{[1]}(\lambda,\mu) = \frac{(\lambda+\e)^{-1}-(\mu+\e)^{-1}}{\lambda-\mu} = - \frac{1}{(\lambda+\e)(\mu+\e)}.
\]
Thus
\[
f_{\e}'(m) = (m+\e)^{-1} \; \text{ and } \; \partial f_{\e}'(m) = (f_{\e}')^{[1]}(m \otimes 1,1 \otimes m) = -(m+\e)^{-1}\otimes (m+\e)^{-1}.
\]
In particular, if $u = \sum_{i=1}^n a_i \otimes b_i$ and $v = \sum_{i=1}^n c_i \otimes d_i$, then
\[
v^{\flip}\partial f_{\e}'(m)\,u = -\sum_{i,j=1}^n\underbrace{(d_j \otimes c_j)((m+\e)^{-1}\otimes (m + \e)^{-1})(a_i \otimes b_i).}_{(d_j(m+\e)^{-1}a_i) \otimes (b_i(m+\e)^{-1}c_j)}
\]
It follows from Theorem \ref{thm.trFFIF} and the Fundamental Theorem of Calculus that
\begin{align*}
    \frac{d}{dt}\tau(f_{\e}(m(t))) & = \tau(f_{\e}'(m(t))\,k(t)) + (\tau \wotimes \tau^{\op})(v^{\flip}(t)\,\partial f_{\e}'(m(t))\,u(t)) \\
    & = \tau((m(t)+\e)^{-1}k(t)) - \sum_{i,j=1}^n\tau\big(d_j(t)(m(t)+\e)^{-1}a_i(t)\big)\,\tau\big(b_i(t)(m(t)+\e)^{-1}c_j(t)\big), \numberthis\label{eq.log}
\end{align*}
for all $t > 0$. Special cases of Equation \eqref{eq.log} have shown up in the calculation of Brown measures of solutions to various free SDEs. Please see \cite{driverhallkemp,hozhong,demnihamdi,hallho}. Thus far, such equations are proven in the literature using power series arguments. Theorem \ref{thm.trFFIF} provides a more intuitive/natural way to do such calculations.

For concreteness, we demonstrate how Equation \eqref{eq.log} can be used to re-prove a key identity (Lemma 5.2 in \cite{driverhallkemp}) that B. K. Driver, B. Hall, and T. Kemp use in the process of computing the Brown measure of free multiplicative Brownian motion.
Similar calculations can be used to re-prove formulas in \cite{hozhong,demnihamdi,hallho}.

We return to the setup of the end of Example \ref{ex.FIprod}, i.e., $dg(t)  = g(t)\,dz(t)$ and $g(0) = h \in \cA_0$.
We then take $g_{\lambda} \coloneqq g-\lambda$ for $\lambda \in \C$ and $m \coloneqq |g_{\lambda}|^2$.
As we showed in Example \ref{ex.FIprod},
\[
dm(t) = g_{\lambda}^*(t)\,g(t) \,dz(t) + dz^*(t)\,g^*(t) \,g_{\lambda}(t) + \tau(g^*(t)\,g(t))\, dt. 
\]
By Equation \eqref{eq.log},
\[
\frac{d}{dt}\tau(\log(m(t)+\e)) = \tau((m(t)+\e)^{-1})\tau(g^*(t)\,g(t))-\tau (g^*(t)\,g_{\lambda}(t)(m(t)+\e)^{-1}g_{\lambda}^*(t)\,g(t))\tau((m(t)+\e)^{-1}).  \numberthis\label{eq.DHKlog}
\]
But now, $\tau (g^*g_{\lambda}(m+\e)^{-1}g_{\lambda}^*g) = \tau ((m+\e)^{-1}g_{\lambda}^*gg^*g_{\lambda})$,
\[
\tau(g^*g) = \tau((m+\e)^{-1}(m+\e)g^*g),
\]
and
\[
(m+\e)g^*g - g_{\lambda}^*gg^*g_{\lambda} = \e \, g^*g + g_{\lambda}^*g_{\lambda}g^*g - g_{\lambda}^*gg^*g_{\lambda} = \e \, g^*g
\]
because $g_{\lambda} = g-\lambda$, and $\lambda = \lambda 1$ commutes with all elements.
From Equation \eqref{eq.DHKlog}, we then get
\[
\frac{d}{dt}\tau\big(\log(|g(t)-\lambda|^2+\e)\big) = \e\,\tau\big((|g(t)-\lambda|^2+\e)^{-1}|g(t)|^2\big)\tau\big((|g(t)-\lambda|^2+\e)^{-1}\big)
\]
for all $t > 0$.
This is equivalent to (a generalization to arbitrary starting point of) Lemma 5.2 in \cite{driverhallkemp}.
\end{ex}

We now begin the proof of Theorem \ref{thm.trFFIF}, the keys to which are the following identities.
\pagebreak

\begin{lem}\label{lem.trid}
If $p \in \C[\lambda]$, $m,k \in \cA$, and $u,v \in \cA \wotimes \cA^{\op}$, then
\[
\tau( \partial p(m) \sh k) = \tau(p'(m)\,k) \; \text{ and } \; \tau \big(\Delta_{u,v}p(m)\big) = (\tau \wotimes \tau^{\op})\big(v^{\flip}\partial p'(m)\,u\big).
\]
\end{lem}
\begin{proof}
Fix $n \in \N_0$, and define $p_n(\lambda) \coloneqq \lambda^n$.
For the first identity, note that
\[
\tau(\partial p_n(m) \sh k) = \sum_{\delta_1+\delta_2=n-1}\tau( m^{\delta_1}k\,m^{\delta_2}) = \sum_{\delta_1+\delta_2=n-1}\tau(m^{\delta_2}m^{\delta_1}k) = \tau(n \,m^{n-1}k) = \tau(p_n'(m)\,k).
\]
By linearity, the first desired identity holds for all $p \in \C[\lambda]$.
Proving the second identity is slightly more involved.
We begin by making two key observations.
First, fix a polynomial $P \in \C[\lambda_1,\lambda_2,\lambda_3]$ and two elements $u_1,u_2 \in \cA \wotimes \cA^{\op}$ that commute.
If we define $q(\lambda_1,\lambda_2) \coloneqq P(\lambda_1,\lambda_2,\lambda_1)$ and $\boldsymbol{1} \coloneqq 1 \otimes 1$, then 
\[
(\tau \wotimes \tau^{\op})\big(P(u_1 \otimes \boldsymbol{1} \otimes \boldsymbol{1}, \boldsymbol{1} \otimes u_2 \otimes \boldsymbol{1},\boldsymbol{1} \otimes \boldsymbol{1} \otimes u_1)\sh_2^{\mathsmaller{\otimes}}[u,\boldsymbol{1}]\big) = (\tau \wotimes \tau^{\op})(q(u_1,u_2)\,u),
\]
as the reader may easily verify.
(The computation is similar to that of $R_a$ in the proof of Lemma \ref{lem.Deltaprodrule}.)
Second, $(\tau \wotimes \tau^{\op})(u^{\flip}) = (\tau \wotimes \tau^{\op})(u)$.
This is because $(\tau \wotimes \tau^{\op})(a \otimes b) = \tau(a)\,\tau(b) = \tau(b)\,\tau(a) = (\tau \wotimes \tau^{\op})(b \otimes a)$ whenever $a,b \in \cA$, and $\cA \otimes \cA^{\op}$ is $\sigma$-weakly dense in $\cA \wotimes \cA^{\op}$.
Now, note that
\[
(q(u_1,u_2)\,u)^{\flip} = u^{\flip} q(u_1,u_2)^{\flip} = u^{\flip} q(u_1^{\flip},u_2^{\flip}). \numberthis\label{eq.flip}
\]
Combining these observations and appealing again to traciality of $\tau \wotimes \tau^{\op}$, we get that if in addition $u_1^{\flip} = u_2$, and if $w \in \cA \wotimes \cA^{\op}$ satisfies $w^{\flip} = w$, then
\[
(\tau \wotimes \tau^{\op})\big(P(u_1 \otimes \boldsymbol{1} \otimes \boldsymbol{1}, \boldsymbol{1} \otimes u_2 \otimes \boldsymbol{1},\boldsymbol{1} \otimes \boldsymbol{1} \otimes u_1)\sh_2^{\mathsmaller{\otimes}}[w,\boldsymbol{1}]\big) = (\tau \wotimes \tau^{\op})(r(u_1,u_2)\,w), \numberthis\label{eq.symm}
\]
where
\[
r(\lambda_1,\lambda_2) = \frac{q(\lambda_1,\lambda_2)+q(\lambda_2,\lambda_1)}{2} = \frac{P(\lambda_1,\lambda_2,\lambda_1)+P(\lambda_2,\lambda_1,\lambda_2)}{2}.
\]
Now, if $P = 2\,p^{[2]}$, then
\[
r(\lambda_1,\lambda_2) = \frac{P(\lambda_1,\lambda_2,\lambda_1)+P(\lambda_2,\lambda_1,\lambda_2)}{2} = p^{[2]}(\lambda_1,\lambda_2,\lambda_1)+p^{[2]}(\lambda_2,\lambda_1,\lambda_2) = (p')^{[1]}(\lambda_1,\lambda_2),
\]
as can be seen by taking $\lambda_3 \to \lambda_1$ in the definition of $p^{[2]}(\lambda_1,\lambda_2,\lambda_3)$ and using the symmetry of $p^{[1]}$.
Therefore, if we apply Equation \eqref{eq.symm} with $P = 2 \, p^{[2]}$, $u_1 = m \otimes 1$, $u_2 = 1 \otimes m$, and $w = \frac{1}{2}(uv^{\flip}+vu^{\flip})$, then we obtain
\[
\tau\big(\Delta_{u,v}p(m)\big) = \frac{1}{2}(\tau \wotimes \tau^{\op})\big(\partial p'(m)\,(uv^{\flip}+vu^{\flip})\big) \numberthis\label{eq.almost}
\]
by definition of $\Delta_{u,v}p(m)$ and noncommutative derivatives.
To complete the proof, notice that if $q \in \C[\lambda_1,\lambda_2]$ is symmetric, $u_1$ and $u_2$ satisfy $u_1^{\flip} = u_2$, and $w \in \cA \wotimes \cA^{\op}$ is arbitrary, then, by Equation \eqref{eq.flip},
\begin{align*}
    (\tau \wotimes \tau^{\op})(q(u_1,u_2)\,w) & = (\tau \wotimes \tau^{\op})\big((q(u_1,u_2)\,w)^{\flip}\big) = (\tau \wotimes \tau^{\op})\big(w^{\flip}q(u_2,u_1)\big) \\
    & = (\tau \wotimes \tau^{\op})\big(w^{\flip}q(u_1,u_2)\big) = (\tau \wotimes \tau^{\op})\big(q(u_1,u_2)\,w^{\flip}\big).
\end{align*}
Therefore, Equation \eqref{eq.almost} reduces to
\[
\tau(\Delta_{u,v}p(m)) = (\tau \wotimes \tau^{\op})(\partial p'(m)\,uv^{\flip}) = (\tau \wotimes \tau^{\op})(v^{\flip}\partial p'(m)\,u),
\]
as desired.
\end{proof}

\begin{proof}[Proof of Theorem \ref{thm.trFFIF}]
We prove Theorem \ref{thm.trFFIF}\ref{item.trFFIFx} using Theorem \ref{thm.FFIFpoly}\ref{item.FFIFpolyx}.
Theorem \ref{thm.trFFIF}\ref{item.trFFIFz} follows in the exact same way from Theorem \ref{thm.FFIFpoly}\ref{item.FFIFpolyz}.
Fix an $n$-dimensional semicircular Brownian motion $(x_1,\ldots,x_n) \colon \R_+ \to \cA_{\sa}^n$, and suppose that $m$ is a free It\^{o} process satisfying Equation \eqref{eq.frItoprocess}.

Since free stochastic integrals against $x_i$ are noncommutative martingales that start at zero, they have trace zero.
Thus, applying $\tau$ to the result of Theorem \ref{thm.FFIFpoly}\ref{item.FFIFpolyx}, bringing $\tau$ (which is bounded-linear) into the Bochner integrals, and appealing to Lemma \ref{lem.trid}, we have
\begin{align*}
    \tau(p(m)) & = \tau(p(m(0))) + \int_0^{\boldsymbol{\cdot}} \Big(\tau(\partial p(m(t)) \sh k(t)) + \frac{1}{2}\sum_{i=1}^n\tau\big(\Delta_{u_i(t)}p(m(t))\big)\Big) \, dt \\
    & = \tau(p(m(0))) + \int_0^{\boldsymbol{\cdot}}\Big(\tau\big(p'(m(t))\,k(t)\big) + \frac{1}{2}\sum_{i=1}^n(\tau \wotimes \tau^{\op})\big(u_i^{\flip}(t)\,\partial p'(m(t))\,u_i(t)\big)\Big) \, dt,
\end{align*}
for all $p \in \C[\lambda]$.
\pagebreak

Suppose now that $m^*=m$ and $U \subseteq \R$ is an open set containing $\overline{\bigcup_{t \geq 0} \sigma(m(t))}$ such that $f \in C^2(U)$.
Since $m$ is continuous in the operator norm, $m$ is locally bounded in the operator norm.
In particular, $K_t \coloneqq \overline{\bigcup_{0 \leq s \leq t} \sigma(m(s))} \subseteq U$ is compact.
Next, fix $t \geq 0$, and let $V_t \subseteq \R$ and $g_t \in C^2(\R)$ be such that $V_t$ is open, $K_t \subseteq V_t \subseteq U$, and $g_t=f$ on $V_t$.
By the Weierstrass Approximation Theorem, there is a sequence $(q_N)_{N \in \N}$ of polynomials such that, for all $j \in \{0,1,2\}$, $q_N^{(j)} \to g_t^{(j)}$ uniformly on compact subsets of $\R$ as $N \to \infty$.
In particular, $(q_N')^{[1]} \to (g_t')^{[1]}$ uniformly on compact subsets of $\R^2$ as $N \to \infty$.
But now,
\[
\tau(q_N(m(t))) = \tau(q_N(m(0))) + \int_0^t \Big(\tau\big(q_N'(m(s))\,k(s)\big) + \frac{1}{2}\sum_{i=1}^n(\tau \wotimes \tau^{\op})\big(u_i^{\flip}(s)\,\partial q_N'(m(s))\,u_i(s)\big)\Big) \, ds,
\]
for all $N \in \N$, by the previous paragraph.
By basic operator norm estimates on functional calculus and the Dominated Convergence Theorem, we can take $N \to \infty$ in this identity to conclude
\[
\tau(g_t(m(t))) = \tau(g_t(m(0))) + \int_0^t \Big(\tau\big(g_t'(m(s))\,k(s)\big) + \frac{1}{2}\sum_{i=1}^n(\tau \wotimes \tau^{\op})\big(u_i^{\flip}(s)\,\partial g_t'(m(s))\,u_i(s)\big)\Big) \, ds.
\]
But $g_t=f$ on $V_t \supseteq K_t$ and thus $(g_t')^{[1]} = (f')^{[1]}$ on $K_t \times K_t$.
We therefore have that $g_t(m(s)) = f(m(s))$ and $\partial g'(m(s)) = \partial f'(m(s))$, for all $s \in [0,t]$.
Since $t \geq 0$ was arbitrary, this completes the proof.
\end{proof}

\section{Free Stochastic Calculus II: Noncommutative \texorpdfstring{$C^2$}{} Functions}\label{sec.fstochcalc2}

\subsection{Noncommutative \texorpdfstring{$C^k$}{} Functions}\label{sec.NCk}

We now briefly discuss the class $NC^k(\R)$ of noncommutative $C^k$ functions introduced in \cite{nikitopoulosNCk}, which contains proofs of all the statements in this section. 

\begin{nota}
If $\Om$ is a set and $\varphi \colon \Om \to \C$ is a function, then we write $\|\varphi\|_{\ell^{\infty}(\Om)} \coloneqq \sup_{\om \in \Om}|\varphi(\om)|$.
If $\sF$ is a $\sigma$-algebra of subsets of $\Om$, then we write $\ell^{\infty}(\Om,\sF) \coloneqq \{ \text{bounded }(\sF,\cB_{\C})\text{-measurable functions }\Om \to \C\}$.
For the duration of this section, fix $k \in \N$ and, for each $j \in \{1,\ldots,k+1\}$, a Polish space (i.e., a complete separable metric space) $\Om_j$ with Borel $\sigma$-algebra $\cB_{\Om_j}$.
Also, write $\Om \coloneqq \Om_1 \times \cdots \times \Om_{k+1}$.
\end{nota}

We now define the \textit{integral projective tensor product} $\ell^{\infty}(\Om_1,\cB_{\Om_1}) \iotimes \cdots \iotimes \ell^{\infty}(\Om_{k+1},\cB_{\Om_{k+1}})$, the idea for which comes from the work of V. V. Peller \cite{peller1}.
Formally, we shall replace the $\sum_{n=1}^{\infty}$ in the decomposition \eqref{eq.projdecomp} of elements of the classical projective tensor product (of $\ell^{\infty}$-spaces) with an integral $\int_{\Sigma} \cdot \, d\rho$ over a $\sigma$-finite measure space.
Here is a precise definition.

\begin{defi}[\textbf{IPTPs}]\label{def.babyIPTP}
A \textbf{$\boldsymbol{\ell^{\infty}}$-integral projective decomposition} (IPD) of a function $\varphi \colon \Om \to \C$ is a choice $(\Sigma,\rho,\varphi_1,\ldots,\varphi_{k+1})$ of a $\sigma$-finite measure space $(\Sigma,\sH,\rho)$ and, for each $j \in \{1,\ldots,k+1\}$, a product measurable function $\varphi_j \colon \Om_j \times \Sigma \to \C$ such that $\varphi_j(\cdot,\sigma) \in \ell^{\infty}(\Om_j,\cB_{\Om_j})$ whenever $\sigma\in \Sigma$,
\begin{align*}
    & \int_{\Sigma} \|\varphi_1(\cdot,\sigma)\|_{\ell^{\infty}(\Om_1)}\cdots\|\varphi_{k+1}(\cdot,\sigma)\|_{\ell^{\infty}(\Om_{k+1})} \, \rho(d\sigma) < \infty, \; \text{ and} \numberthis\label{eq.intfincond} \\
    & \varphi(\boldsymbol{\om}) = \int_{\Sigma} \varphi_1(\om_1,\sigma) \cdots \varphi_{k+1}(\om_{k+1},\sigma) \, \rho(d\sigma), \text{ for all } \boldsymbol{\om} \in \Om,
\end{align*}
where we write $\boldsymbol{\om} = (\om_1,\ldots,\om_{k+1})$.
Also, for any function $\varphi \colon \Om \to \C$, define
\begin{align*}
    \|\varphi\|_{\ell^{\infty}(\Om_1,\cB_{\Om_1}\hspace{-0.1mm}) \iotimes \cdots \iotimes \ell^{\infty}(\Om_{k+1},\cB_{\Om_{k+1}}\hspace{-0.1mm})} & \coloneqq \inf\int_{\Sigma} \prod_{j=1}^{k+1}\|\varphi_j(\cdot,\sigma)\|_{\ell^{\infty}(\Om_j)}\,\rho(d\sigma),
\end{align*}
where $(\Sigma,\rho,\varphi_1,\ldots,\varphi_{k+1})$ is a $\ell^{\infty}$-IPD of $\varphi$ and $\inf \emptyset \coloneqq \infty$.
Finally, we define
\[
\ell^{\infty}(\Om_1,\cB_{\Om_1}) \iotimes \cdots \iotimes \ell^{\infty}(\Om_{k+1},\cB_{\Om_{k+1}}) \coloneqq \big\{\varphi \in \ell^{\infty}(\Om,\cB_{\Om}) : \|\varphi\|_{\ell^{\infty}(\Om_1,\cB_{\Om_1}\hspace{-0.1mm}) \iotimes \cdots \iotimes \ell^{\infty}(\Om_{k+1},\cB_{\Om_{k+1}}\hspace{-0.1mm})} < \infty\big\}
\]
to be the \textbf{integral projective tensor product} (IPTP) of the spaces $\ell^{\infty}(\Om_1,\cB_{\Om_1}),\ldots,\ell^{\infty}(\Om_{k+1},\cB_{\Om_{k+1}})$. 
\end{defi}

It is not obvious that the integral in Equation \eqref{eq.intfincond} makes sense.
In fact, the function being integrated is not necessarily measurable, but it \textit{is} ``almost measurable," i.e., measurable with respect to the $\rho$-completion of $\sH$.
For a proof, please see Lemma 2.2.1 in \cite{nikitopoulosNCk}.
\pagebreak

It is easy to see that if $\varphi \colon \Om \to \C$ is a function, then
\[
\|\varphi\|_{\ell^{\infty}(\Om)} \leq  \|\varphi\|_{\ell^{\infty}(\Om_1,\cB_{\Om_1}\hspace{-0.1mm}) \iotimes \cdots \iotimes \ell^{\infty}(\Om_{k+1},\cB_{\Om_{k+1}}\hspace{-0.1mm})}.
\]
It is also the case that $\ell^{\infty}(\Om_1,\cB_{\Om_1}) \iotimes \cdots \iotimes \ell^{\infty}(\Om_{k+1},\cB_{\Om_{k+1}}) \subseteq \ell^{\infty}(\Om,\cB_{\Om})$ is a $\ast$-subalgebra with respect to pointwise operations and that
\[
\big(\ell^{\infty}(\Om_1,\cB_{\Om_1}) \iotimes \cdots \iotimes \ell^{\infty}(\Om_{k+1},\cB_{\Om_{k+1}}),\|\cdot\|_{\ell^{\infty}(\Om_1,\cB_{\Om_1}\hspace{-0.1mm}) \iotimes \cdots \iotimes \ell^{\infty}(\Om_{k+1},\cB_{\Om_{k+1}}\hspace{-0.1mm})}\big)
\]
is a Banach $\ast$-algebra.
For proofs, please see Proposition 2.2.3 in \cite{nikitopoulosNCk}.

\begin{nota}[The Space $\cC^{[k]}(\R)$]\label{nota.superCk}
For $\varphi \colon \R^{k+1} \to \C$ and $r > 0$, define
\[
\|\varphi\|_{r,k+1} \coloneqq \big\|\varphi|_{[-r,r]^{k+1}}\big\|_{\ell^{\infty}([-r,r],\cB_{[-r,r]}\hspace{-0.1mm})^{\iotimes(k+1)}} \in [0,\infty].
\]
Now, if $f \in C^k(\R)$, then we define\vspace{-0.25mm}
\begin{align*}
    \|f\|_{\cC^{[k]},r} & \coloneqq \sum_{j=0}^k\big\|f^{[j]}\big\|_{r,j+1} \in [0,\infty] \; \text{ and} \\
    \cC^{[k]}(\R) & \coloneqq \big\{ g \in C^k(\R) : \|g\|_{\cC^{[k]},r}  < \infty, \text{ for all } r > 0\big\},
\end{align*}
where $\|\cdot\|_{r,1} \coloneqq \|\cdot\|_{\ell^{\infty}([-r,r])}$.
\end{nota}

Note that $\cC^{[k]}(\R) \subseteq C^k(\R)$ is a linear subspace, and $\{\|\cdot\|_{\cC^{[k]},r} : r > 0\}$ is a collection of seminorms on $\cC^{[k]}(\R)$.
This collection of seminorms makes $\cC^{[k]}(\R)$ into a Fr\'{e}chet space --- actually, a Fr\'{e}chet $\ast$-algebra.
This is proven as Proposition 3.1.3(iv) in \cite{nikitopoulosNCk}.

\begin{ex}[Polynomials]\label{ex.polysuperCk}
Fix $n \in \N$.
For each $j \in \{1,\ldots,k+1\}$ and $\ell \in \{1,\ldots,n\}$, fix a bounded Borel measurable function $\psi_{j,\ell} \colon \Om_j \to \C$.
If $\psi(\boldsymbol{\om}) \coloneqq \sum_{\ell=1}^n\psi_{1,\ell}(\om_1)\cdots\psi_{k+1,\ell}(\om_{k+1})$, for all $\boldsymbol{\om} \in \Om$, then it is easy to see that $\psi \in \ell^{\infty}(\Om_1,\cB_{\Om_1}) \iotimes \cdots \iotimes \ell^{\infty}(\Om_{k+1},\cB_{\Om_{k+1}})$ with\vspace{-0.25mm}
\[
\|\psi\|_{\ell^{\infty}(\Om_1,\cB_{\Om_1}\hspace{-0.1mm}) \iotimes \cdots \iotimes \ell^{\infty}(\Om_{k+1},\cB_{\Om_{k+1}}\hspace{-0.1mm})} \leq \sum_{\ell=1}^n\|\psi_{1,\ell}\|_{\ell^{\infty}(\Om_1)}\cdots\|\psi_{k+1,\ell}\|_{\ell^{\infty}(\Om_{k+1})}.\vspace{-0.25mm}
\]
In particular, if $P \in \C[\lambda_1,\ldots,\lambda_{k+1}]$, then\vspace{-0.25mm}
\[
P|_{[-r,r]^{k+1}} \in \ell^{\infty}\big([-r,r],\cB_{[-r,r]}\big)^{\iotimes(k+1)},\vspace{-0.25mm}
\]
for all $r > 0$.
Therefore, by Example \ref{ex.divdiffpoly}, $\C[\lambda] \subseteq \cC^{[k]}(\R)$.
\end{ex}

This brings us to the definition of $NC^k(\R)$.

\begin{defi}[Noncommutative $C^k$ Functions]\label{def.NCk}
We define
\[
NC^k(\R) \coloneqq \overline{\C[\lambda]} \subseteq \cC^{[k]}(\R)
\]
to be the space of \textbf{noncommutative $\boldsymbol{C^k}$ functions}.
To be clear, the closure above takes place in the Fr\'{e}chet space $\cC^{[k]}(\R)$.
\end{defi}
\begin{rem}
The idea for the name of $NC^k(\R)$ comes from parallel work by D. A. Jekel, who, in Section 18 of \cite{jekel}, defined an abstract analog of $NC^k(\R)$ via completion and using classical projective tensor powers of $C([-r,r])$ in place of the integral projective tensor powers of $\ell^{\infty}([-r,r],\cB_{[-r,r]})$.
Jekel notates his space of noncommutative $C^k$ functions as $C_{\text{nc}}^k(\R)$.
By definition of $C_{\mathrm{nc}}^k(\R)$, the inclusion $\C[\lambda] \hookrightarrow NC^k(\R)$ extends uniquely to a continuous map linear map $C_{\text{nc}}^k(\R) \to NC^k(\R)$.
\end{rem}

Since $\C[\lambda] \subseteq \cC^{[k]}(\R)$ is a $\ast$-subalgebra, $NC^k(\R)$ is a Fr\'{e}chet $\ast$-algebra in its own right.
We record many examples of noncommutative $C^k$ functions in Theorem \ref{thm.NCk} below.

\begin{defi}[Wiener Space]\label{def.Wk}
Write $M(\R,\cB_{\R})$ for the space of Borel complex measures on $\R$.
For $\mu \in M(\R,\cB_{\R})$, write $|\mu|$ for the total variation measure of $\mu$, $\mu_{(0)} \coloneqq |\mu|(\R)$ for the total variation norm of $\mu$, and $\mu_{(k)} \coloneqq \int_{\R} |\xi|^k\,|\mu|(d\xi) \in [0,\infty]$ for the ``$k^{\text{th}}$ moment" of $|\mu|$.
The \textbf{$\boldsymbol{k^{\text{th}}}$ Wiener space} $W_k(\R)$ is the set of functions $f \colon \R \to \C$ such that there exists (necessarily unique) $\mu \in M(\R,\cB_{\R})$ with $\mu_{(k)} < \infty$ and $f(\lambda) = \int_{\R} e^{i\lambda \xi} \, \mu(d\xi)$, for all $\lambda \in \R$.
Finally, we define $W_k(\R)_{\loc}$ to be set of functions $f \colon \R \to \C$ such that for all $r > 0$, there exists $g \in W_k(\R)$ such that $f|_{[-r,r]} = g|_{[-r,r]}$.
\end{defi}
\pagebreak

\begin{thm}[Nikitopoulos \cite{nikitopoulosNCk}]\label{thm.NCk}
Write $\dot{B}_1^{k,\infty}(\R)$ for the homogeneous $(k,\infty,1)$-Besov space (Definition 3.3.1 in \cite{nikitopoulosNCk}) and $C_{\loc}^{k,\e}(\R)$ for the space of $C^k$ functions whose $k^{\text{th}}$ derivatives are locally $\e$-H\"{o}lder continuous (Definition 3.3.8 in \cite{nikitopoulosNCk}).
\begin{enumerate}[label=(\roman*),font=\normalfont,leftmargin=2\parindent]
    \item $C^{k+1}(\R) \subseteq W_k(\R)_{\loc} \subseteq NC^k(\R)$, and $W_k(\R)$ is dense in $NC^k(\R)$.\label{item.WkinNCk}
    \item $\dot{B}_1^{k,\infty}(\R) \subseteq NC^k(\R)$ and $C_{\loc}^{k,\e}(\R) \subseteq NC^k(\R)$, for all $\e > 0$.\label{item.BesovHolderinNCk}
    \item $W_k(\R)_{\loc} \subsetneq NC^k(\R)$. Specifically, $C^{k,\frac{1}{4}}(\R) \setminus W_k(\R)_{\loc} \neq \emptyset$.\label{item.Wklocstrict}
\end{enumerate}
\end{thm}
\begin{proof}[Sketch of proof]
We outline the highlights of the proof and refer the reader to \cite{nikitopoulosNCk} for details.
First, we record a property of $NC^k(\R)$ that we shall use.
For $\mathcal{S} \subseteq C^k(\R)$, write $\mathcal{S}_{\loc}$ for the set of $f \in C^k(\R)$ such that for all $r > 0$, there exists $g \in \mathcal{S}$ such that $f|_{[-r,r]} = g|_{[-r,r]}$.
Proposition 3.1.3(ii) in \cite{nikitopoulosNCk} says that if $\mathcal{S} \subseteq \cC^{[k]}(\R)$, then $\mathcal{S}_{\loc} \subseteq \overline{\mathcal{S}} \subseteq \cC^{[k]}(\R)$, where the closure takes place in the Fr\'{e}chet space $\cC^{[k]}(\R)$.
Since $NC^k(\R) \subseteq \cC^{[k]}(\R)$ is closed, we have that if $\mathcal{S} \subseteq NC^k(\R)$, then $\mathcal{S}_{\loc} \subseteq \overline{\mathcal{S}} \subseteq NC^k(\R)$.

\ref{item.WkinNCk} The containment $C^{k+1}(\R) \subseteq W_k(\R)_{\loc}$ is proven using elementary Fourier analysis as Lemma 3.2.3(iii) in \cite{nikitopoulosNCk}.
More specifically, one proves that if $f \in C^{k+1}(\R)$, $r > 0$, and $\eta_r \in C_c^{\infty}(\R)$ satisfies $\eta_r \equiv 1$ on $[-r,r]$, then $\eta_rf = \frac{1}{2\pi}\int_{\R} e^{i\boldsymbol{\cdot}\xi}\cF(\eta_rf)(\xi)\,d\xi \in W_k(\R)$, where $\cF(g)(\xi) = \int_{\R} e^{-i\xi \lambda}g(\lambda)\,d\lambda$ is the Fourier transform of $g$. 

Now, one may use Equation \eqref{eq.divdiffWk} to show that $W_k(\R) \subseteq \cC^{[k]}(\R)$ (Lemma 3.2.4 in \cite{nikitopoulosNCk}).
It follows that $\C[\lambda] \subseteq C^{k+1}(\R) \subseteq W_k(\R)_{\loc} \subseteq \overline{W_k(\R)} \subseteq \cC^{[k]}(\R)$, and thus $NC^k(\R) \subseteq \overline{W_k(\R)}$.
To complete the proof of this item, we must show that $W_k(\R) \subseteq NC^k(\R)$. To this end, let $f = \int_{\R} e^{i \boldsymbol{\cdot} \xi}\,\mu(d\xi) \in W_k(\R)$.
For $n \in \N$, define $\mu_n(d\xi) \coloneqq 1_{[-n,n]}(\xi)\,\mu(d\xi)$ and $f_n \coloneqq \int_{\R} e^{i\boldsymbol{\cdot}\xi}\,\mu_n(d\xi) \in W_k(\R)$.
Then one has $f_n \to f$ in $\cC^{[k]}(\R)$ as $n \to \infty$, so it suffices to assume that $|\mu|$ has compact support.
If $|\mu|$ has compact support, then we may define $q_n(\lambda) \coloneqq \int_{\R} \sum_{j=0}^n \frac{(i\lambda\xi)^j}{j!}\,\mu(d\xi) \in \C[\lambda]$, for all $n \in \N$.
One then has that $q_n \to f$ in $\cC^{[k]}(\R)$ as $n \to \infty$.
For details, please see Theorem 3.2.6 in \cite{nikitopoulosNCk}.

\ref{item.BesovHolderinNCk} Using a result of Peller (Theorem 5.5 in \cite{peller1}), Littlewood--Paley decompositions, and the previous item, one can prove $\dot{B}_1^{k,\infty}(\R) \subseteq NC^k(\R)$.
Thus $\dot{B}_1^{k,\infty}(\R)_{\loc} \subseteq NC^k(\R)$.
But if $\e > 0$, then one also has $C_{\loc}^{k,\e}(\R) \subseteq \dot{B}_1^{k,\infty}(\R)_{\loc}$, so $C_{\loc}^{k,\e}(\R) \subseteq NC^k(\R)$.
For details, please see Section 3.3 of \cite{nikitopoulosNCk}.

\ref{item.Wklocstrict} Lemma 3.4.3 in \cite{nikitopoulosNCk} says that if $g \in C(\R)$ has compact support and $h \in C^k(\R)$ satisfies $h^{(k)} = g$, then $h \in W_k(\R)_{\loc}$ if and only if $\wh{g} = \cF(g)$ is integrable. Lemma 3.4.4 in \cite{nikitopoulosNCk} gives an example of a compactly supported function $\kappa \in C^{0,\frac{1}{4}}(\R)$ such that $\wh{\kappa}$ is not integrable. Therefore, if $f \in C^k(\R)$ and $f^{(k)} = \kappa$, then $f \in C^{k,\frac{1}{4}}(\R) \setminus W_k(\R)_{\loc} \subseteq NC^k(\R) \setminus W_k(\R)_{\loc}$.
\end{proof}

This collection of results paints the picture that a function only has to be ``a tiny bit better than $C^k$" to be noncommutative $C^k$.
However, a function does not have to belong locally to $W_k(\R)$ to belong to $NC^k(\R)$.
Though we shall not need it, it is also the case that $NC^k(\R) \subsetneq C^k(\R)$ (Theorem 4.4.1 in \cite{nikitopoulosNCk}).

\subsection{Multiple Operator Integrals (MOIs)}\label{sec.MOI}

In the next section, we shall define and/or interpret the quantities $\partial f(m) \sh k$ and $\Delta_{u,v}f(m)$ in terms of multiple operator integrals (MOIs).
In this section, we state the definitions and facts from the vast and rich theory of MOIs that we need for the present application.
Please see A. Skripka and A. Tomskova's book \cite{skripka} for a thorough and well-organized survey of the MOI literature and its applications.

For the duration of this section, fix a complex Hilbert space $H$, a von Neumann algebra $\cM \subseteq B(H)$, $k \in \N$, and $\boldsymbol{a} = (a_1,\ldots,a_{k+1}) \in \cM_{\sa}^{k+1}$.
For the present application, we shall use the ``separation of variables" approach, developed in \cite{peller1,azamovetal,peller2,nikitopoulosMOI}, to defining the MOI
\[
\big(I^{\boldsymbol{a}}\varphi\big)[b] = \int_{\sigma(a_{k+1})}\cdots\int_{\sigma(a_1)}\varphi(\blambda) \, P^{a_1}(d\lambda_1) \, b_1 \cdots P^{a_k}(d\lambda_k) \, b_k \, P^{a_{k+1}}(d\lambda_{k+1})
\]
for $b \in \cM^k$ and $\varphi \in \ell^{\infty}(\sigma(a_1),\cB_{\sigma(a_1)}) \iotimes \cdots \iotimes \ell^{\infty}(\sigma(a_{k+1}),\cB_{\sigma(a_{k+1})})$.
To motivate the definition, let $(\Sigma,\rho,\varphi_1,\ldots,\varphi_{k+1})$ be a $\ell^{\infty}$-IPD of $\varphi$.
We would like
\begin{align*}
    \big(I^{\boldsymbol{a}}\varphi\big)[b] & = \int_{\sigma(a_{k+1})}\cdots\int_{\sigma(a_1)}\int_{\Sigma} \varphi_1(\lambda_1,\sigma) \cdots \varphi_{k+1}(\lambda_{k+1},\sigma) \, \rho(d\sigma)\, P^{a_1}(d\lambda_1) \, b_1 \cdots P^{a_k}(d\lambda_k) \, b_k \, P^{a_{k+1}}(d\lambda_{k+1}) \\\displaybreak
    & = \int_{\Sigma}\int_{\sigma(a_{k+1})}\cdots\int_{\sigma(a_1)} \varphi_1(\lambda_1,\sigma) \cdots \varphi_{k+1}(\lambda_{k+1},\sigma) \, P^{a_1}(d\lambda_1) \, b_1 \cdots P^{a_k}(d\lambda_k) \, b_k \, P^{a_{k+1}}(d\lambda_{k+1})\, \rho(d\sigma) \\
    & = \int_{\Sigma}\Bigg(\int_{\sigma(a_1)}\varphi_1(\cdot,\sigma) \,dP^{a_1}\, b_1 \cdots \int_{\sigma(a_k)}\varphi_k(\cdot,\sigma) \,dP^{a_k}\,b_k \int_{\sigma(a_{k+1})}\varphi_{k+1}(\cdot,\sigma) \,dP^{a_{k+1}}\Bigg)\,\rho(d\sigma) \\
    & = \int_{\Sigma} \varphi_1(a_1,\sigma)\,b_1\cdots \varphi_k(a_k,\sigma)\,b_k\,\varphi_{k+1}(a_{k+1},\sigma)\,\rho(d\sigma). \numberthis\label{eq.MOIputdef}
\end{align*}
The integral in Equation \eqref{eq.MOIputdef} will become the definition of $(I^{\boldsymbol{a}}\varphi)[b]$.
First, we must explain what kind of integral this is.
If $(\Sigma,\sH,\rho)$ is a measure space and $F \colon \Sigma \to B(H)$ is a map, then we say that $F$ is \textbf{pointwise Pettis integrable} if, for every $h_1,h_2 \in H$, $\la F(\cdot)h_1,h_2 \ra \colon \Sigma \to \C$ is $(\sH,\cB_{\C})$-measurable and $\int_{\Sigma} |\la F(\sigma)h_1,h_2 \ra|\,\rho(d\sigma) < \infty$.
In this case, by Lemma 4.2.1 in \cite{nikitopoulosNCk}, there exists unique $T \in B(H)$ such that $\la Th_1,h_2 \ra = \int_{\Sigma} \la F(\sigma)h_1,h_2 \ra\,\rho(d\sigma)$, for all $h_1,h_2 \in H$;
moreover, $T \in W^*(F(\sigma) : \sigma \in \Sigma)$.
We shall write $\int_{\Sigma} F \, d\rho = \int_{\Sigma}F(\sigma)\,\rho(d\sigma) \coloneqq T$ for this operator.
Note that if $F \colon \Sigma \to B(H)$ is Bochner integrable, then it is also pointwise Pettis integrable, and the operator $T$ coincides with the Bochner $\rho$-integral of $F$.

\begin{prop}\label{prop.integ}
If $\varphi \in \ell^{\infty}(\sigma(a_1),\cB_{\sigma(a_1)}) \iotimes \cdots \iotimes \ell^{\infty}(\sigma(a_{k+1}),\cB_{\sigma(a_{k+1})})$, $(\Sigma,\rho,\varphi_1,\ldots,\varphi_{k+1})$ is a $\ell^{\infty}$-IPD of $\varphi$, and $b_1,\ldots,b_k \in \cM$, then the map
\[
\Sigma \ni \sigma \mapsto \varphi_1(a_1,\sigma)\,b_1\cdots \varphi_k(a_k,\sigma)\,b_k\,\varphi_{k+1}(a_{k+1},\sigma) \in \cM
\]
is pointwise Pettis integrable.
\end{prop}
\begin{proof}
Please see Corollary 4.2.4 in \cite{nikitopoulosMOI} or Theorem 4.2.4(i) in \cite{nikitopoulosNCk}.
\end{proof}

Crucially, the pointwise Pettis integral in Equation \eqref{eq.MOIputdef} is also independent of the chosen $\ell^{\infty}$-IPD.

\begin{thm}[Well-Definition of MOIs]\label{thm.basicMOI}
Fix
\[
\varphi \in \ell^{\infty}\big(\sigma(a_1),\cB_{\sigma(a_1)}\big) \iotimes \cdots \iotimes \ell^{\infty}\big(\sigma(a_{k+1}),\cB_{\sigma(a_{k+1})}\big),
\]
$b = (b_1,\ldots,b_k) \in \cM^k$, and a $\ell^{\infty}$-IPD $(\Sigma,\rho,\varphi_1,\ldots,\varphi_{k+1})$ of $\varphi$.
\begin{enumerate}[label=(\roman*),font=\normalfont,leftmargin=2\parindent]
    \item The pointwise Pettis integral
    \[
    \big(I^{\boldsymbol{a}}\varphi\big)[b] \coloneqq \int_{\Sigma} \varphi_1(a_1,\sigma)\,b_1\cdots \varphi_k(a_k,\sigma)\,b_k\,\varphi_{k+1}(a_{k+1},\sigma)\,\rho(d\sigma) \in \cM
    \]
    is independent of the chosen $\ell^{\infty}$-IPD.\label{item.welldef}
    \item The assignment $\varphi \mapsto (I^{\boldsymbol{a}}\varphi)[b]$ is linear.\label{item.lin}
    \item $\|(I^{\boldsymbol{a}}\varphi)[b]\| \leq \|\varphi\|_{\ell^{\infty}(\sigma(a_1),\cB_{\sigma(a_1)}\hspace{-0.1mm}) \iotimes \cdots \iotimes \ell^{\infty}(\sigma(a_{k+1}),\cB_{\sigma(a_{k+1})}\hspace{-0.1mm})}\|b_1\|\cdots\|b_k\|$.\label{item.bd}
\end{enumerate}
\end{thm}
\begin{proof}
We prove or give references for each item in turn.

\ref{item.welldef} For a proof in the case where $H$ is separable, please see Theorem 2.1.1 in \cite{peller2} (or Lemma 4.3 in \cite{azamovetal} for a more restrictive class of $\ell^{\infty}$-IPDs).
For $H$ not necessarily separable (in a very general setting), please see Theorem 4.2.12 in \cite{nikitopoulosMOI}.
For a proof that the separable case implies the general case in the present setting, please see the sketch of Theorem 4.2.4(ii) in \cite{nikitopoulosNCk}.

\ref{item.lin} Please see Proposition 4.3.1(i) in \cite{nikitopoulosMOI} or Theorem 4.2.4(iii) in \cite{nikitopoulosNCk}.

\ref{item.bd} For $\sigma \in \Sigma$, write $F(\sigma) \coloneqq \varphi_1(a_1,\sigma)\,b_1\cdots \varphi_k(a_k,\sigma)\,b_k\,\varphi_{k+1}(a_{k+1},\sigma)$.
Then
\begin{align*}
    \|F(\sigma)\| & \leq \|\varphi_1(a_1,\sigma)\|\,\|b_1\|\cdots\|\varphi_k(a_k,\sigma)\|\,\|b_k\|\,\|\varphi_{k+1}(a_{k+1},\sigma)\| \\
    & \leq \|\varphi_1(\cdot,\sigma)\|_{\ell^{\infty}(\sigma(a_1))}\cdots\|\varphi_{k+1}(\cdot,\sigma)\|_{\ell^{\infty}(\sigma(a_{k+1}))}\|b_1\|\cdots\|b_k\|.
\end{align*}
Therefore,
\begin{align*}
    \big\|\big(I^{\boldsymbol{a}}\varphi\big)[b]\big\| & = \Bigg\|\int_{\Sigma} F \,d\rho\Bigg\| = \sup\Bigg\{\Bigg|\Bigg\la \Bigg(\int_{\Sigma} F\,d\rho\Bigg)h_1,h_2 \Bigg\ra\Bigg| : h_1,h_2 \in H, \, \|h_1\|,\|h_2\| \leq 1\Bigg\} \\
    & = \sup\Bigg\{\Bigg|\int_{\Sigma} \la F(\sigma)h_1,h_2 \ra \, \rho(d\sigma)\Bigg| : h_1,h_2 \in H, \, \|h_1\|,\|h_2\| \leq 1 \Bigg\} \\
    & \leq \|b_1\|\cdots\|b_k\|\int_{\Sigma} \prod_{j=1}^{k+1}\|\varphi_j(\cdot,\sigma)\|_{\ell^{\infty}(\sigma(a_j))}\,\rho(d\sigma).
\end{align*}
Taking the infimum over $\ell^{\infty}$-IPDs of $\varphi$ gives the desired bound.
\end{proof}

This development allows us to make the definition we wanted.

\begin{defi}[Multiple Operator Integral]\label{def.MOI}
If $\varphi \in \ell^{\infty}(\sigma(a_1),\cB_{\sigma(a_1)}) \iotimes \cdots \iotimes \ell^{\infty}(\sigma(a_{k+1}),\cB_{\sigma(a_{k+1})})$ and $b = (b_1,\ldots,b_k) \in \cM^k$, then we define
\[
\int_{\sigma(a_{k+1})}\cdots\int_{\sigma(a_1)}\varphi(\blambda) \, P^{a_1}(d\lambda_1) \, b_1 \cdots P^{a_k}(d\lambda_k) \, b_k \, P^{a_{k+1}}(d\lambda_{k+1}) \numberthis\label{eq.MOI}
\]
to be the element $(I^{\boldsymbol{a}}\varphi)[b] \in \cM$ from Theorem \ref{thm.basicMOI}\ref{item.welldef}.
We call the $k$-linear map $I^{\boldsymbol{a}}\varphi \colon \cM^k \to \cM$ the \textbf{multiple operator integral} (MOI) of $\varphi$ with respect to $(P^{a_1},\ldots,P^{a_{k+1}})$.
\end{defi}

\begin{ex}[Polynomials]\label{ex.polyMOI}
Fix $P(\blambda) = \sum_{|\delta| \leq d} c_{\delta}\,\blambda^{\delta} \in \C[\lambda_1,\ldots,\lambda_{k+1}]$.
By Example \ref{ex.polysuperCk} and the definition of MOIs, if $b = (b_1,\ldots,b_k) \in \cM^k$, then
\[
\big(I^{\boldsymbol{a}}P\big)[b] = \sum_{|\delta| \leq d} c_{\delta}\,a_1^{\delta_1}b_1\cdots a_k^{\delta_k}b_k\,a_{k+1}^{\delta_{k+1}}.
\]
In particular, by Example \ref{ex.divdiffpoly} and Equation \eqref{eq.polyNCder}, if $p \in \C[\lambda]$, then
\begin{align}
    \big(I^{a_1,a_2}p^{[1]}\big)[b] & = \partial p(a_1,a_2) \sh b  \; \text{ and} \label{eq.polyMOI1}\\
    \big(I^{u_1,u_2,u_3}p^{[2]}\big)[v_1,v_2] & = \frac{1}{2}\partial^2 p(u_1,u_2,u_3)\sh_2^{\mathsmaller{\otimes}}[v_1,v_2], \label{eq.polyMOI2}
\end{align}
for all $a_1,a_2 \in \cA_{\sa}$, $b \in \cA$, $u_1,u_2,u_3 \in (\cA \wotimes \cA^{\op})_{\sa}$, and $v_1,v_2 \in \cA \wotimes \cA^{\op}$.
Recall that $(\cA,(\cA_t)_{t \geq 0},\tau)$ is our fixed filtered $W^*$-probability space, and the operations $\#$ and $\#_2^{\mathsmaller{\otimes}}$ are defined in Notation \ref{nota.alg}.
\end{ex}

Because of the correction term $\frac{1}{2}\Delta_{u,v}f(m)$ in the functional free It\^{o} formula(s) to come, we shall also need to understand MOIs of the form $\int_{\sigma(a_2)} \int_{\Lambda} \int_{\sigma(a_1)} \varphi(\lambda_1,\lambda_2,\lambda_3)\,P^{a_1}(d\lambda_1)\,b_1\,\mu(d\lambda_2)\,b_2 \, P^{a_2}(d\lambda_3)$, where $\Lambda$ is a Polish space and $\mu$ is a Borel complex measure on $\Lambda$.

\begin{lem}\label{lem.onemeas}
Let $\Lambda$ be a Polish space and $\mu$ be a Borel complex measure on $\Lambda$.
If
\[
\varphi \in \ell^{\infty}(\sigma(a_1),\cB_{\sigma(a_1)}) \iotimes \ell^{\infty}(\Lambda,\cB_{\Lambda}) \iotimes \ell^{\infty}(\sigma(a_2),\cB_{\sigma(a_2)}) \; \text{ and } \; \varphi^{\mu}(\lambda_1,\lambda_3) \coloneqq \int_{\Lambda} \varphi(\lambda_1,\lambda_2,\lambda_3)\,\mu(d\lambda_2)
\]
for $(\lambda_1,\lambda_3) \in \sigma(a_1) \times \sigma(a_2)$, then $\varphi^{\mu} \in \ell^{\infty}(\sigma(a_1),\cB_{\sigma(a_1)}) \iotimes \ell^{\infty}(\sigma(a_2),\cB_{\sigma(a_2)})$.
Henceforth, if $b_1,b_2 \in \cM$, then we shall write
\[
\int_{\sigma(a_2)}\int_{\Lambda} \int_{\sigma(a_1)} \varphi(\lambda_1,\lambda_2,\lambda_3) \, P^{a_1}(d\lambda_1) \,b_1 \,\mu(d\lambda_2)\,b_2\,P^{a_2}(d\lambda_3)
\]
for the element $(I^{a_1,a_2}\varphi^{\mu})[b_1b_2] \in \cM$.
\end{lem}
\begin{proof}
Let $(\Sigma,\rho,\varphi_1,\varphi_2,\varphi_3)$ be a  $\ell^{\infty}$-integral projective decomposition of $\varphi$.
If we define
\[
\varphi_1^{\mu}(\lambda_1,\sigma) \coloneqq \varphi_1(\lambda_1,\sigma) \int_{\Lambda} \varphi_2(\lambda_2,\sigma)\,\mu(d\lambda_2) \; \text{ and } \; \varphi_3^{\mu}(\lambda_3,\sigma) \coloneqq \varphi_3(\lambda_3,\sigma),
\]
for all $\lambda_1 \in \sigma(a_1),\lambda_3 \in \sigma(a_2), \sigma \in \Sigma$, then it is easy to see that $(\Sigma,\rho,\varphi_1^{\mu},\varphi_3^{\mu})$ is a $\ell^{\infty}$-IPD of $\varphi^{\mu}$.
\end{proof}

It follows from the proof above and Definition \ref{def.MOI} that
\begin{align*}
    & \int_{\sigma(a_2)}\int_{\Lambda} \int_{\sigma(a_1)} \varphi(\lambda_1,\lambda_2,\lambda_3) \, P^{a_1}(d\lambda_1) \,b_1 \,\mu(d\lambda_2)\,b_2\,P^{a_2}(d\lambda_3) \\
    & \hspace{35mm} = \int_{\Sigma} \mu(\varphi_2(\cdot,\sigma))\,\varphi_1(a_1,\sigma) \, b_1b_2\,\varphi_3(a_2,\sigma)\,\rho(d\sigma) \numberthis\label{eq.MOIwithcompmeas}
\end{align*}
whenever\hspace{-0.25mm} $(\Sigma,\hspace{-0.15mm}\rho,\hspace{-0.15mm}\varphi_1,\hspace{-0.15mm}\varphi_2,\hspace{-0.15mm}\varphi_3)$\hspace{-0.25mm} is \hspace{-0.25mm}a\hspace{-0.1mm} $\ell^{\infty}$-integral \hspace{-0.25mm}projective\hspace{-0.25mm} decomposition \hspace{-0.25mm}of\hspace{-0.25mm} $\varphi$, \hspace{-0.25mm}where\hspace{-0.25mm} $\mu(\varphi_2(\cdot,\sigma)) \hspace{-0.25mm}\coloneqq\hspace{-0.25mm} \int_{\Lambda} \hspace{-0.25mm}\varphi_2(\lambda,\sigma)\mu(d\lambda)$.

\subsection{Functional Free It\^{o} Formula for \texorpdfstring{$NC^2(\R)$}{}}\label{sec.FFIF}

In this section, we finally prove our It\^{o} formula for noncommutative $C^2$ functions of self-adjoint free It\^{o} processes (Theorem \ref{thm.FFIF}).
Recall that $(\cA,(\cA_t)_{t \geq 0},\tau)$ is a fixed filtered $W^*$-probability space.

We begin by identifying $\partial f(m) \sh k$ as a MOI.
\pagebreak

\begin{lem}\label{lem.partialfMOI}
If $f \in W_1(\R)_{\loc}$ and $a_1,a_2 \in \cA_{\sa}$, then
\[
\partial f(a_1,a_2) \in \cA \potimes \cA^{\op} \; \text{ and } \; \partial f(a_1,a_2) \sh b = \big(I^{a_1,a_2}f^{[1]}\big)[b],
\]
for all $b \in \cA$.
Moreover, the map
\[
\cA_{\sa}^2 \ni (a_1,a_2) \mapsto \partial f(a_1,a_2) \in \cA \potimes \cA^{\op}
\]
is continuous.
\end{lem}
\begin{proof}
Fix $a_1,a_2 \in \cA_{\sa}$, and let $r > 0$ be such that $\sigma(a_1) \cup \sigma(a_2) \subseteq [-r,r]$.
By definition of $W_1(\R)_{\loc}$, there is some $g = \int_{\R} e^{i \boldsymbol{\cdot} \xi}\,\mu(d\xi) \in W_1(\R)$ such that $g|_{[-r,r]} = f|_{[-r,r]}$.
In particular, $g^{[1]}|_{[-r,r]^2} = f^{[1]}|_{[-r,r]^2}$. Thus
\[
\partial f(a_1,a_2) = \partial g(a_1,a_2) \in \cA \potimes \cA^{\op} \numberthis\label{eq.partialagree}
\]
by Example \ref{ex.NCderWk}.
Moreover, since $\cA \potimes \cA^{\op} \ni u \mapsto u \sh b \in \cA$ is bounded-linear, the same example gives
\begin{align*}
    \partial g(a_1,a_2) \sh b & = \int_0^1 \int_{\R}(i\xi) \,e^{ita_1} b\, e^{i(1-t)a_2} \, \mu(d\xi)\,dt \\
    & = \int_{\R \times [0,1]}(i\xi)\,e^{ita_1} b\, e^{i(1-t)a_2}\frac{d\mu}{d|\mu|}(\xi)  \, |\mu|(d\xi)\,dt \\
    & = \big(I^{a_1,a_2}g^{[1]}\big)[b] = \big(I^{a_1,a_2}f^{[1]}\big)[b],
\end{align*}
for all $b \in \cA$, where the third identity holds by Equation \eqref{eq.divdiffWk} and the definition of MOIs.

For the continuity claim, note that the map
\[
\cA_{\sa}^2 \ni (a_1,a_2) \mapsto \int_0^1 \int_{\R}(i\xi) \,e^{ita_1} \otimes e^{i(1-t)a_2} \, \mu(d\xi)\,dt \in \cA \potimes \cA^{\op}
\]
is continuous by the vector-valued Dominated Convergence Theorem.
In particular, $(a_1,a_2) \mapsto \partial g(a_1,a_2)$ is continuous.
Since Equation \eqref{eq.partialagree} holds whenever $\sigma(a_1) \cup \sigma(a_2) \subseteq [-r,r]$ (i.e., whenever $\|a_1\| \leq r$ and $\|a_2\| \leq r$), we conclude that $(a_1,a_2) \mapsto \partial f(a_1,a_2)$ is continuous on $\{(a_1,a_2) \in \cA_{\sa}^2: \|a_1\| \leq r, \, \|a_2\| \leq r\}$.
Since $r > 0$ was arbitrary, we are done.
\end{proof}

Since $C^2(\R) \subseteq W_1(\R)_{\loc}$ (Theorem \ref{thm.NCk}\ref{item.WkinNCk}), the conclusion of Lemma \ref{lem.partialfMOI} holds for all $f \in C^2(\R)$.
Thus Equation \eqref{eq.polyMOI1} is a special case of Lemma \ref{lem.partialfMOI}.

Next, we make sense of $\Delta_{u,v}f(m)$ in terms of MOIs.
If $f \in \cC^{[2]}(\R)$, $m \in \cA_{\sa}$, and $u,v \in \cA \wotimes \cA^{\op}$, then we define
\[
\ell_{f,u,v}(a) \coloneqq (\tau \wotimes \tau^{\op})\big((a \otimes 1)\, \big(I^{m \otimes 1, 1 \otimes m, m \otimes 1}f^{[2]}\big)[uv^{\flip}+vu^{\flip},1 \otimes 1]\big),
\]
for all $a \in L^1(\cA,\tau)$.
By Theorem \ref{thm.basicMOI}\ref{item.bd},
\begin{align*}
    \|\ell_{f,u,v}\|_{L^1(\cA,\tau)^*} & \leq \big\|\big(I^{m \otimes 1, 1 \otimes m, m \otimes 1}f^{[2]}\big)[uv^{\flip}+vu^{\flip},1 \otimes 1]\big\|_{L^{\infty}(\tau \wotimes \tau^{\op})} \\
    & \leq \big\|f^{[2]}\big\|_{\ell^{\infty}(\sigma(m),\cB_{\sigma(m)})^{\iotimes 3}}\|uv^{\flip}+vu^{\flip}\|_{L^{\infty}(\tau \wotimes \tau^{\op})}\|1 \otimes 1\|_{L^{\infty}(\tau \wotimes \tau^{\op})} \\
    & \leq 2\big\|f^{[2]}\big\|_{\ell^{\infty}(\sigma(m),\cB_{\sigma(m)})^{\iotimes 3}}\|u\|_{L^{\infty}(\tau \wotimes \tau^{\op})}\|v\|_{L^{\infty}(\tau \wotimes \tau^{\op})} < \infty.
\end{align*}
In particular, the following definition makes sense.

\begin{defi}\label{def.Deltaf}
If $f \in \cC^{[2]}(\R)$, $m \in \cA_{\sa}$, and $u,v \in \cA \wotimes \cA^{\op}$, then we define $\Delta_{u,v}f(m)$ to be the unique element of $\cA$ such that
\[
\tau\big(a \, \Delta_{u,v}f(m)\big) = (\tau \wotimes \tau^{\op})\big((a \otimes 1)\, \big(I^{m \otimes 1, 1 \otimes m, m \otimes 1}f^{[2]}\big)[uv^{\flip}+vu^{\flip},1 \otimes 1]\big),
\]
for all $a \in \cA$ (and thus $a \in L^1(\cA,\tau)$).
Also, write $\Delta_u f(m) \coloneqq \Delta_{u,u}f(m)$.
\end{defi}

By Equation \eqref{eq.polyMOI2}, Definition \ref{def.Deltaf} agrees with Definition \ref{def.Deltap} when both definitions apply (i.e., when $f \in \C[\lambda]$ and $m \in \cA_{\sa}$).
Also, if $f \in \cC^{[2]}(\R)$, $m \in \cA_{\sa}$, and $u,v \in \cA \wotimes \cA^{\op}$, then
\[
\|\Delta_{u,v}f(m)\| = \|\ell_{f,u,v}\|_{L^1(\cA,\tau)^*} \leq 2\big\|f^{[2]}\big\|_{\ell^{\infty}(\sigma(m),\cB_{\sigma(m)})^{\iotimes 3}}\|u\|_{L^{\infty}(\tau \wotimes \tau^{\op})}\|v\|_{L^{\infty}(\tau \wotimes \tau^{\op})} \numberthis\label{eq.Deltafbd}
\]
by the paragraph before Definition \ref{def.Deltaf}.
\pagebreak

\begin{lem}\label{lem.Deltaf}
If $f \in NC^2(\R)$, $m \in C(\R_+;\cA_{\sa})$, and $u,v \in L_{\loc}^2(\R_+;\cA \wotimes \cA^{\op})$, then
\[
\Delta_{u,v}f(m) \in L_{\loc}^1(\R_+;\cA) \; \text{ and } \; \|\Delta_{u,v}f(m)\|_{L_t^1L^{\infty}(\tau)} \leq 2\big\|f^{[2]}\big\|_{r_t,3}\|u\|_{L_t^2L^{\infty}(\tau \wotimes \tau^{\op})}\|v\|_{L_t^2L^{\infty}(\tau \wotimes \tau^{\op})},
\]
for all $t \geq 0$, where $r_t \coloneqq \|m\|_{L_t^{\infty}L^{\infty}(\tau)} = \sup_{0 \leq s \leq t}\|m(s)\|$.
\end{lem}
\begin{proof}
When $f \in \C[\lambda]$, we know from Section \ref{sec.FFIFpoly} that $\Delta_{u,v}f(m) \in L_{\loc}^1(\R;\cA)$.
The claimed bound follows from applying Equation \eqref{eq.Deltafbd} pointwise and then using the Cauchy--Schwarz Inequality.
If $f \in NC^2(\R)$ is arbitrary, then there is a sequence $(q_N)_{N \in \N}$ of polynomials converging in $NC^2(\R)$ (i.e., in $\cC^{[2]}(\R)$) to $f$.
What we just proved implies that the sequence $(\Delta_{u,v}q_N(m))_{N \in \N}$ is Cauchy in $L_{\loc}^1(\R_+;\cA)$, and Equation \eqref{eq.Deltafbd} implies that $\Delta_{u,v}q_N(m) \to \Delta_{u,v}f(m)$ almost everywhere as $N \to \infty$.
It follows that $\Delta_{u,v}f(m) \in L_{\loc}^1(\R_+;\cA)$, and that the claimed bound holds for $\Delta_{u,v}f(m)$ as well.
\end{proof}

We are finally ready for the functional free It\^{o} formula for $NC^2$ functions.

\begin{thm}[Functional Free It\^{o} Formula]\label{thm.FFIF}
Fix $f \in NC^2(\R)$.
\begin{enumerate}[label=(\roman*),font=\normalfont,leftmargin=2\parindent]
   \item Suppose that $(x_1,\ldots,x_n) \colon \R_+ \to \cA_{\sa}^n$ is an $n$-dimensional semicircular Brownian motion, and that $m$ is a free It\^{o} process satisfying Equation \eqref{eq.frItoprocess}.
   If $m^*=m$, then\label{item.FFIFx}
    \[
    d\,f(m(t)) = \partial f(m(t))\sh dm(t) + \frac{1}{2}\sum_{i=1}^n\Delta_{u_i(t)}f(m(t))\,dt.\numberthis\label{eq.FFIFx}
    \]
    Recall from Corollary \ref{cor.safrItoprocess} that $m^*=m$ is equivalent to $m(0)^*=m(0)$, $k^*=k$ a.e., and $u_i^{\mathsmaller{\bigstar}} = u_i$ a.e. for all $i$.
    \item Suppose that $(z_1,\ldots,z_n) \colon \R_+ \to \cA^n$ is an $n$-dimensional circular Brownian motion, and that $m$ is a free It\^{o} process satisfying Equation \eqref{eq.zfrItoprocess}.
    If $m^*=m$, then\label{item.FFIFz}
    \[
    d\,f(m(t)) = \partial f(m(t))\sh dm(t) + \sum_{i=1}^n\Delta_{u_i(t),u_i^{\text{\scalebox{0.65}{$\bigstar$}}}(t)}f(m(t))\,dt.\numberthis\label{eq.FFIFz}
    \]
    Recall from Corollary \ref{cor.safrItoprocess} that $m^*=m$ is equivalent to $m(0)^*=m(0)$, $k^*=k$ a.e., and $u_i^{\mathsmaller{\bigstar}} = v_i$ a.e. for all $i$.
\end{enumerate}
\end{thm}
\begin{rem}
Note that in either case, $\R_+ \ni t \mapsto \partial f(m(t)) \in \cA \potimes \cA^{\op}$ is continuous (Lemma \ref{lem.partialfMOI}) and adapted.
In particular, if $\ell \in L_{\loc}^1(\R_+;\cA)$ and $u \in \Lambda^2$, then $\partial f(m) \sh \ell \in L_{\loc}^1(\R_+;\cA)$ and, by Corollary \ref{cor.RCLL}, $\partial f(m) \, u \in \Lambda^2$.
Thus all of the integrals in the statement of Theorem \ref{thm.FFIF} make sense.
\end{rem}
\begin{proof}
As usual, item \ref{item.FFIFz} follows from item \ref{item.FFIFx} with twice the dimension, so it suffices to prove item \ref{item.FFIFx}.
To this end, let $m=m^*$ be a free It\^{o} process satisfying Equation \eqref{eq.frItoprocess}.
By Theorem \ref{thm.FFIFpoly}\ref{item.FFIFpolyx}, Equation \eqref{eq.FFIFx} holds when $f \in \C[\lambda]$.
For general $f \in NC^2(\R)$, let $(q_N)_{N \in \N}$ be a sequence of polynomials converging in $NC^2(\R)$ to $f$, and fix $t \geq 0$.
Since $q_N \to f$ uniformly on compact sets, we have that $q_N(m(t)) \to f(m(t))$ in $\cA$ as $N \to \infty$.
Next, fix $i \in \{1,\ldots,n\}$.
By Lemma \ref{lem.Deltaf}, $\Delta_{u_i}q_N(m) \to \Delta_{u_i}f(m)$ in $L_{\loc}^1(\R_+;\cA)$ as $N \to \infty$.
In particular, $\int_0^t\Delta_{u_i(s)}q_N(m(s))\,ds \to \int_0^t\Delta_{u_i(s)}f(m(s))\,ds$ in $\cA$ as $N \to \infty$.
Now, write $r_t \coloneqq \sup_{0 \leq s \leq t}\|m(s)\| < \infty$.
Then
\begin{align*}
    \|\partial q_N(m) \,u_i - \partial f(m) \,u_i\|_{L_t^2L^{\infty}(\tau \wotimes \tau^{\op})} & = \big\|(q_N-f)^{[1]}(m \otimes 1, 1 \otimes m) \, u_i\big\|_{L_t^2L^{\infty}(\tau \wotimes \tau^{\op})} \\
    & \leq \big\|(q_N-f)^{[1]}\big\|_{\ell^{\infty}([-r_t,r_r]^2)}\|u_i\|_{L_t^2L^{\infty}(\tau \wotimes \tau^{\op})} \to 0
\end{align*}
as $N \to \infty$ by basic properties of functional calculus and the fact that $\|\cdot\|_{\ell^{\infty}([-r_t,r_t]^2)} \leq \|\cdot\|_{r_t,2}$.
Therefore, by the $L^{\infty}$-BDG Inequality, $\int_0^t (\partial q_N(m(s))\,u_i(s))\sh dx_i(s) \to \int_0^t (\partial f(m(s))\,u_i(s))\sh dx_i(s)$ in $\cA$ as $N \to \infty$.
Finally, by Lemma \ref{lem.partialfMOI} and Theorem \ref{thm.basicMOI}\ref{item.lin}--\ref{item.bd}, we have
\[
\|\partial q_N(m) \sh k - \partial f(m) \sh k\|_{L_t^1L^{\infty}(\tau)} = \big\|\big(I^{m,m}(q_N-f)^{[1]}\big)[k]\big\|_{L_t^1L^{\infty}(\tau)} \leq \big\|(q_N-f)^{[1]}\big\|_{r_t,2}\|k\|_{L_t^1L^{\infty}(\tau)} \to 0
\]
as $N \to \infty$.
In particular, $\int_0^t\partial q_N(m(s)) \sh k(s)\,ds \to \int_0^t \partial f(m(s)) \sh k(s)\,ds$ in $\cA$ as $N \to \infty$.
Therefore, we may deduce Equation \eqref{eq.FFIFx} by taking $N \to \infty$ in the corresponding identity for $q_N$.
\end{proof}

We end this section by deriving an explicit formula for $\Delta_{u,v}f(m)$ (with $u,v \in \cA \otimes \cA^{\op}$) in terms of MOIs.
Using this formula, we shall see directly that Theorem \ref{thm.FFIF}\ref{item.FFIFx} generalizes Proposition 4.3.4 in \cite{bianespeicher}.
For this development, we shall view $\cA$ as a $W^*$-subalgebra of $B(L^2(\cA,\tau))$ via the standard representation, i.e., as acting on $L^2(\cA,\tau)$ by left multiplication.

\begin{prop}[Computing $\Delta_{u,v}f(m)$]\label{prop.expressDeltaf}
Fix $f \in \cC^{[2]}(\R)$ and $m \in \cA_{\sa}$, and let $(\Sigma,\rho,\varphi_1,\varphi_2,\varphi_3)$ be a $\ell^{\infty}$-IPD of $f^{[2]}$ on $\sigma(m)^3$.
If $u,v \in \cA \otimes \cA^{\op}$, then
\begin{align*}
    \Delta_{u,v}f(m) & = \int_{\Sigma}\mathcal{M}_{\tau}\big((1 \otimes v) \boldsymbol{\cdot} (\varphi_1(m,\sigma) \otimes \varphi_2(m,\sigma) \otimes \varphi_3(m,\sigma)) \boldsymbol{\cdot} (u \otimes 1) \\
    & \hspace{22.5mm}  + (1 \otimes u) \boldsymbol{\cdot} (\varphi_1(m,\sigma) \otimes \varphi_2(m,\sigma) \otimes \varphi_3(m,\sigma)) \boldsymbol{\cdot} (v \otimes 1)\big) \,\rho(d\sigma),
\end{align*}
where the right hand side is a pointwise Pettis integral in $\cA \subseteq B(L^2(\cA,\tau))$.
\end{prop}
\begin{proof}
As usual, we write $\boldsymbol{1} \coloneqq 1 \otimes 1$.
If $a,b \in L^2(\cA,\tau)$ (so that $ab^* \in L^1(\cA,\tau)$), then we have
\begin{align*}
    \big\la \big(\Delta_{u,v}f(m)& \big)a,b \big\ra_{L^2(\tau)} = \tau\big(b^* \, \Delta_{u,v}f(m)a\big) = \tau\big(ab^* \, \Delta_{u,v}f(m)\big) \\
    & = (\tau \wotimes \tau^{\op})\big((ab^* \otimes 1) \big(I^{m \otimes 1, 1 \otimes m, m \otimes 1}f^{[2]}\big)[uv^{\flip}+vu^{\flip},\boldsymbol{1}]\big) \numberthis\label{eq.Deltacalc0} \\
    & = (\tau \wotimes \tau^{\op})\big((b \otimes 1)^* \big(I^{m \otimes 1, 1 \otimes m, m \otimes 1}f^{[2]}\big)[uv^{\flip}+vu^{\flip},\boldsymbol{1}](a \otimes 1)\big) \\
    & = \big\la \big(I^{m \otimes 1, 1 \otimes m, m \otimes 1}f^{[2]}\big)[uv^{\flip}+vu^{\flip},\boldsymbol{1}](a \otimes 1),b \otimes 1\big\ra_{L^2(\tau \wotimes \tau^{\op})} \\
    & = \int_{\Sigma} \la \varphi_1(m \otimes 1,\sigma)(uv^{\flip}+vu^{\flip})\varphi_2(1 \otimes m,\sigma)\,\varphi_3(m \otimes 1,\sigma)(a \otimes 1),b \otimes 1\ra_{L^2(\tau \wotimes \tau^{\op})}\,\rho(d\sigma) \numberthis\label{eq.Deltacalc1}\\
    & = \int_{\Sigma} \la (\varphi_1(m,\sigma) \otimes 1)(uv^{\flip}+vu^{\flip})(1 \otimes \varphi_2(m,\sigma))(\varphi_3(m ,\sigma)\otimes 1)(a \otimes 1),b \otimes 1\ra_{L^2(\tau \wotimes \tau^{\op})}\,\rho(d\sigma) \\
    & = \int_{\Sigma}(\tau \wotimes \tau^{\op})\big((ab^* \otimes 1)(\varphi_1(m,\sigma) \otimes 1)(uv^{\flip}+vu^{\flip})(1 \otimes \varphi_2(m,\sigma))(\varphi_3(m ,\sigma)\otimes 1)\big)\,\rho(d\sigma) \\
    & = \int_{\Sigma}\tau\big(ab^*\,\mathcal{M}_{\tau}\big((1 \otimes v) \boldsymbol{\cdot} (\varphi_1(m,\sigma) \otimes \varphi_2(m,\sigma) \otimes \varphi_3(m,\sigma)) \boldsymbol{\cdot} (u \otimes 1) \\
    & \hspace{27.5mm}  + (1 \otimes u) \boldsymbol{\cdot} (\varphi_1(m,\sigma) \otimes \varphi_2(m,\sigma) \otimes \varphi_3(m,\sigma)) \boldsymbol{\cdot} (v \otimes 1)\big)\big) \,\rho(d\sigma) \numberthis\label{eq.Deltacalc2}\\
    & = \int_{\Sigma}\big\la \mathcal{M}_{\tau}\big((1 \otimes v) \boldsymbol{\cdot} (\varphi_1(m,\sigma) \otimes \varphi_2(m,\sigma) \otimes \varphi_3(m,\sigma)) \boldsymbol{\cdot} (u \otimes 1) \\
    & \hspace{27.5mm}  + (1 \otimes u) \boldsymbol{\cdot} (\varphi_1(m,\sigma) \otimes \varphi_2(m,\sigma) \otimes \varphi_3(m,\sigma)) \boldsymbol{\cdot} (v \otimes 1)\big)a,b\big\ra_{L^2(\tau)} \,\rho(d\sigma),
\end{align*}
where Equation \eqref{eq.Deltacalc0} holds by definition of $\Delta_{u,v}f(m)$, Equation \eqref{eq.Deltacalc1} holds by definition of MOIs and pointwise Pettis integrals, and Equation \eqref{eq.Deltacalc2} holds by Lemma \ref{lem.keyalg} (and an elementary limiting argument).
By definition of pointwise Pettis integrals, this completes the proof.
\end{proof}

\begin{cor}\label{cor.expressDeltaf}
Retain the setting of Proposition \ref{prop.expressDeltaf}.
If $a,b,c,d \in \cA$, then
\begin{align*}
    \Delta_{a \otimes b, c \otimes d}f(m) & =  \int_{\sigma(m)}\int_{\sigma(m)}\int_{\sigma(m)}f^{[2]}(\lambda_1,\lambda_2,\lambda_3) \, P^m(d\lambda_1) \,a\,\tau(b \,P^m(d\lambda_2)\,c)\,d \,P^m(d\lambda_3) \\
    & \hspace{12.5mm} + \int_{\sigma(m)}\int_{\sigma(m)}\int_{\sigma(m)} f^{[2]}(\lambda_1,\lambda_2,\lambda_3) \, P^m(d\lambda_1) \,c\,\tau(d\, P^m(d\lambda_2)\,a)\,b \,P^m(d\lambda_3).\numberthis\label{eq.Deltafdet}
\end{align*}
Note that $\mu(d\lambda) = \tau(b\,P^m(d\lambda)\,c)$ and $\nu(d\lambda) = \tau(d \, P^m(d\lambda)\,a)$ are Borel complex measures on $\sigma(m)$.
\end{cor}
\begin{proof}
This follows immediately from Proposition \ref{prop.expressDeltaf}, the definition of $\mathcal{M}_{\tau}$, and Equation \eqref{eq.MOIwithcompmeas}.
\end{proof}

\begin{ex}[Connection to Biane--Speicher Formula]\label{ex.BS}
Retain the setting of Proposition \ref{prop.expressDeltaf}, but suppose further that $f \in W_2(\R)_{\loc} \subseteq NC^2(\R)$.
Let $g = \int_{\R} e^{i\boldsymbol{\cdot}\xi}\,\mu(d\xi) \in W_2(\R)$ be such that $g|_{[-r,r]} = f|_{[-r,r]}$, where $r = \|m\|$.
Since $f^{[2]}|_{[-r,r]^3} = g^{[2]}|_{[-r,r]^3}$, Equation \eqref{eq.divdiffWk} gives
\begin{align*}
    f^{[2]}(\lambda_1,\lambda_2,\lambda_3) & = \int_0^1\int_0^{1-t} \int_{\R}(i\xi)^2e^{is \lambda_1\xi}e^{it\lambda_2 \xi}e^{i(1-s-t)\lambda_3\xi}\,\mu(d\xi)\,ds\,dt \\
    & = \int_{\R \times \Sigma_2}(i\xi)^2e^{is \lambda_1\xi}e^{it\lambda_2 \xi}e^{i(1-s-t)\lambda_3\xi}\frac{d\mu}{d|\mu|}(\xi)\,|\mu|(d\xi)\,ds\,dt,
\end{align*}
for all $(\lambda_1,\lambda_2,\lambda_3) \in [-r,r]^3$.
Therefore, by Proposition \ref{prop.expressDeltaf}, if $u,v \in \cA \otimes \cA^{\op}$, then
\begin{align*}
    \Delta_{u,v}f(m) & = \int_{\R \times \Sigma_2} \mathcal{M}_{\tau}\Big( (1 \otimes v) \boldsymbol{\cdot} \Big((i\xi)^2e^{is\xi m} \otimes e^{it \xi m } \otimes \Big(e^{i(1-s-t)\xi m}\frac{d\mu}{d|\mu|}(\xi)\Big)\Big) \boldsymbol{\cdot} (u \otimes 1) \\
    & \hspace{27.5mm} + (1 \otimes u) \boldsymbol{\cdot} \Big((i\xi)^2e^{is\xi m} \otimes e^{it \xi m } \otimes \Big(e^{i(1-s-t)\xi m}\frac{d\mu}{d|\mu|}(\xi)\Big)\Big) \boldsymbol{\cdot} (v \otimes 1) \Big) \,|\mu|(d\xi)\,ds\,dt \\
    & = -\int_0^1\int_0^{1-t}\int_{\R} \xi^2 \mathcal{M}_{\tau}\big( (1 \otimes v) \boldsymbol{\cdot} (e^{is\xi m} \otimes e^{it \xi m } \otimes e^{i(1-s-t)\xi m}) \boldsymbol{\cdot} (u \otimes 1) \\
    & \hspace{42.5mm} + (1 \otimes u) \boldsymbol{\cdot} (e^{is\xi m} \otimes e^{it \xi m } \otimes e^{i(1-s-t)\xi m}) \boldsymbol{\cdot} (v \otimes 1) \big) \,\mu(d\xi)\,ds\,dt.
\end{align*}
When $u=v$, this is exactly Biane and Speicher's definition of $\Delta_uf(m)$ from \cite{bianespeicher}.\footnote{Beware: As is noted in \cite{bianespeicherF}, the definition of $\Delta_uf(m)$ actually written in \cite{bianespeicher} is missing a factor of $2$.}
Moreover, since we saw in the proof of Lemma \ref{lem.partialfMOI} that $\partial f(m) = i\int_0^1\int_{\R}\xi\,e^{itm} \otimes e^{i(1-t)m}\,\mu(d\xi)\,dt$, this demonstrates directly that Theorem \ref{thm.FFIF}\ref{item.FFIFx} does, in fact, generalize Proposition 4.3.4 in \cite{bianespeicher}.
\end{ex}

\begin{rem}\label{rem.tech}
If $X,Y,Z$ are topological spaces and $F \colon X \times Y \to Z$ is a function, then we call $F$ \textbf{argumentwise continuous} if, for every $x \in X$ and $y \in Y$, the maps $F(x,\cdot) \colon Y \to Z$ and $F(\cdot,y) \colon X \to Z$ are continuous.
Now, fix $m \in \cA$, $p \in \C[\lambda]$, and $f \in \cC^{[2]}(\R)$.
Write $B \colon (\cA \wotimes \cA^{\op})^2 \to \cA$ for any one of the bilinear maps $Q_{\tau}$, $\Delta_{\cdot,\cdot}p(m)$, or $\Delta_{\cdot,\cdot}f(m)$.
Of course, when $B = \Delta_{\cdot,\cdot}f(m)$, we implicitly assume $m \in \cA_{\sa}$.
When $B \in \{Q_{\tau},\Delta_{\cdot,\cdot}p(m)\}$, it is easy to see from the definition that $B$ is argumentwise continuous with respect to the weak$^*$ topologies (i.e., $\sigma$-WOTs)  on $\cA \wotimes \cA^{\op}$ and $\cA$. This is also true when $B = \Delta_{\cdot,\cdot}f(m)$, but it is substantially harder to prove.
The key is that the MOI in Equation \eqref{eq.MOI} is argumentwise $\sigma$-weakly continuous in $b$;
this is a special case of Corollary 4.2.11 in \cite{nikitopoulosMOI}.
In any case, no matter the choice of $B$, $B$ is argumentwise $\sigma$-weakly continuous.
Since $\cA \otimes \cA^{\op}$ is $\sigma$-weakly dense in $\cA \wotimes \cA^{\op}$, $B|_{(\cA \otimes \cA^{\op})^2}$ extends uniquely to \textit{an argumentwise $\sigma$-weakly continuous bilinear map} $(\cA \wotimes \cA^{\op})^2 \to \cA$.
To this extent, $B$ is determined by its respective algebraic formula (Equations \eqref{eq.Qtau}, \eqref{eq.Deltap}, or \eqref{eq.Deltafdet}).
However, $\cA \otimes \cA^{\op}$ is \textit{not} necessarily (operator) norm dense in $\cA \wotimes \cA^{\op}$.
For example, if $\cA = L^{\infty}([0,1])$, then $\cA \wotimes \cA^{\op} = L^{\infty}([0,1]) \wotimes L^{\infty}([0,1]) = L^{\infty}([0,1]^2)$, and it is a standard exercise to show that if $\Delta_+ = \{(x,y) : 0 \leq x \leq y \leq 1\}$, then $1_{\Delta_+} \in L^{\infty}([0,1]^2) \setminus L^{\infty}([0,1]) \otimes_{\min} L^{\infty}([0,1])$.
In particular, boundedness of $B|_{(\cA \otimes \cA^{\op})^2}$ as a bilinear map does \textit{not} necessarily imply that there exists a unique bounded bilinear extension of $B|_{(\cA \otimes \cA^{\op})^2}$ to $(\cA \wotimes \cA^{\op})^2$.
Biane and Speicher implicitly claim uniqueness of such an extension in the paragraphs after Definition 4.3.1 and Lemma 4.3.3 in \cite{bianespeicher}.
However, this luckily does not harm their development because we \textit{can} guarantee a unique bounded bilinear extension to $(\cA \otimes_{\min} \cA^{\op})^2$, and, as we noted in Remark \ref{rem.min}, $\Lambda^2 \subseteq L_{\loc}^2(\R_+;\cA \otimes_{\min} \cA^{\op})$.
\end{rem}

\appendix

\section{Matrix Stochastic Calculus Formulas}\label{sec.fdmotiv}

The main purpose of this appendix is to motivate our main results (Theorems \ref{thm.FFIF} and \ref{thm.trFFIF}) by studying an It\^{o} formula for $C^2$ scalar functions of Hermitian matrix-valued It\^{o} processes (Theorem \ref{thm.FIF}).
To the author's knowledge, this formula is not written elsewhere in the literature, though its existence is mentioned --- at least for polynomials --- in \cite{anshelevich}.
For the duration of the appendix, fix a filtered probability space $(\Om,\sF,(\sF_t)_{t \geq 0},P)$, with filtration satisfying the usual conditions, to which all processes to come are adapted.
Also, we shall adhere to Notation \ref{nota.mat} and, for a function $f \colon \R \to \C$, write $f_{\MnC} \colon \MnC_{\sa} \to \MnC$ for the associated \textbf{matrix function} $\MnC_{\sa} \ni M \mapsto f(M) \in \MnC$ defined via functional calculus.

Fix $n,N \in \N$, and --- as in the introduction --- let $\big(X_1^{(N)},\ldots,X_n^{(N)}\big)  = (X_1,\ldots,X_n)$ be an $n$-tuple of independent standard $(\MnC_{\sa},\la \cdot,\cdot \ra_N)$-valued Brownian motions.
Concretely, if $\mathcal{E} \subseteq \MnC_{\sa}$ is any orthonormal basis (ONB) for the real inner product space $(\MnC_{\sa},\la\cdot,\cdot\ra_N)$, then
\[
X_i = \sum_{E \in \mathcal{E}} b_{i,E}\,E,\numberthis\label{eq.repofXj}
\]
where $\{b_{j,E} = (b_{j,E}(t))_{t \geq 0} : 1 \leq j \leq n, \,E \in \mathcal{E}\}$ is a collection of $nN^2$ independent standard real Brownian motions.
This representation of $X_i$ will allow us to use the following ``Magic Formula" to identify various ``trace terms" in our stochastic calculus formulas.
Please see Section 3.1 of \cite{driverhallkempSB}, the paper from which the name ``Magic Formula" originates, for a proof.
\pagebreak

\begin{lema}[Magic Formula]\label{lem.magic}
If $\mathcal{E} \subseteq \MnC_{\sa}$ is a $\la \cdot, \cdot \ra_N$-ONB for $\MnC_{\sa}$, then
\[
\sum_{E \in \mathcal{E}}EBE = \tr(B) \, I_N,
\]
for all $B \in \MnC$, where $I_N$ is the $N \times N$ identity matrix.
\end{lema}

Next, we set some algebraic notation.

\begin{notaa}
For $k \in \N$, write $L_k(\MnC)$ for the space of $k$-linear maps $\MnC^k \to \MnC$, and let $\#_k \colon \MnC^{\otimes (k+1)} \to L_k(\MnC)$ be the linear map determined by
\[
\#_k(A_1 \otimes \cdots \otimes A_{k+1})[B_1,\ldots,B_k] = A_1B_1\cdots A_kB_kA_{k+1},
\]
for all $A_1,B_1,\ldots,A_k,B_k,A_{k+1} \in \MnC$.
Whenever $U \in \MnC^{\otimes(k+1)}$ and $B = (B_1,\ldots,B_k) \in \MnC^k$, we shall write
\[
U \sh_k B \coloneqq \#_k(U)[B].
\]
Also, when $k=1$, we shall view the domain of $\#_1$ as $\MnC \otimes \MnC^{\op}$ and write simply $\#_1 = \#$.
\end{notaa}

Using basic linear algebra, one can show that if $k \in \N$, then $\#_k \colon \MnC^{\otimes(k+1)} \to L_k(\MnC)$ is a linear isomorphism.
Also, $\# \colon \MnC \otimes \MnC^{\op} \to L_1(\MnC) = \End(\MnC)$ is an algebra homomorphism.
In particular, we may identify $\End(\MnC)$-valued processes $U = (U(t))_{t \geq 0}$ with $\MnC \otimes \MnC^{\op}$-valued processes and write, for instance, $\int_0^t U(s)\sh dY(s) = \int_0^t U(s)[dY(s)]$ for the stochastic integral of $U$ against the $\MnC$-valued semimartingale $Y$ (when this makes sense).
In view of this identification and notation, we introduce $N \times N$ \textit{matrix It\^{o} processes}.

\begin{defia}[Matrix It\^{o} Process]\label{def.matItoprocess}
A $\boldsymbol{N \times N}$ \textbf{matrix It\^{o} process} is an adapted process $M$ taking values in $\MnC$ that satisfies
\[
dM(t) = \sum_{i=1}^nU_i(t) \sh dX_i(t) + K(t) \,dt \numberthis\label{eq.matItoprocess} 
\]
for some predictable $\MnC \otimes \MnC^{\op}$-valued processes $U_1,\ldots,U_n$ and some progressively measurable $\MnC$-valued process $K$ satisfying
\[
\sum_{i=1}^n\int_0^t\|U_i(s)\|_{\otimes_N}^2\,ds + \int_0^t\|K(s)\|_N \, ds < \infty, \, \text{ for all } t \geq 0, \numberthis\label{eq.integUjK}
\]
almost surely, where $\|\cdot\|_{\otimes_N}$ is the norm associated to the tensor inner product $\la \cdot, \cdot \ra_{\otimes_N}$ on $\MnC \otimes \MnC^{\op}$ induced by the usual Hilbert--Schmidt (Frobenius) inner product on $\MnC$ (and $\MnC^{\op}$).
\end{defia}
\begin{rema}
The conditions in and preceding Equation \eqref{eq.integUjK} guarantee that all the integrals in Equation \eqref{eq.matItoprocess} make sense and that $M$ is a continuous $\MnC$-valued semimartingale.
\end{rema}

Now, we compute the quadratic covariation of two matrix It\^{o} processes.

\begin{defia}[Magic Operator]
Write $\mathcal{M}_{\tr} \colon \MnC^{\otimes 3} \to \MnC$ for the linear map determined by
\[
\mathcal{M}_{\tr}(A \otimes B \otimes C) = A \,\tr(B) \, C = \tr(B)\,AC,
\]
for all $A,B,C \in \MnC$. We call $\mathcal{M}_{\tr}$ the \textbf{magic operator}.
Another way to write it is
\[
\mathcal{M}_{\tr} = \mathfrak{m}_{\MnC} \circ (\id_{\MnC} \otimes \tr \otimes \id_{\MnC}),
\]
where $\mathfrak{m}_{\MnC} \colon \MnC \otimes \MnC \to \MnC$ is the linear map induced by multiplication in $\MnC$.
\end{defia}

\begin{lema}\label{lem.keyFIFcalc}
Suppose that $\mathcal{E} \subseteq \MnC_{\sa}$ is a $\la \cdot, \cdot \ra_N$-orthonormal basis.
If $W \in \MnC^{\otimes 3}$ and $U,V \in \MnC \otimes \MnC^{\op}$, then
\[
\sum_{E \in \mathcal{E}} W\sh_2 [U\sh E, V \sh E] = \mathcal{M}_{\tr}((I_N \otimes V)\boldsymbol{\cdot} W \boldsymbol{\cdot} (U \otimes I_N)),
\]
where $\boldsymbol{\cdot}$ is multiplication in $\MnC \otimes \MnC^{\op} \otimes \MnC$ (for example, one has $(A \otimes B \otimes C)\boldsymbol{\cdot} (D \otimes E \otimes F) = (AD) \otimes (EB) \otimes (CF)$).
\end{lema}
\begin{proof}
It suffices to prove the formula when $U = A \otimes B$, $V = C \otimes D$, and $W = A_1 \otimes A_2 \otimes A_3$ are pure tensors.
In this case, we have
\begin{align*}
    \sum_{E \in \mathcal{E}} W\sh_2 [U\sh E, V \sh E] & = \sum_{E \in \mathcal{E}} W\sh_2 [AEB, CED] = \sum_{E \in \mathcal{E}} A_1AEBA_2CEDA_3 \\
    & = A_1A\,\tr(BA_2C)\,DA_3 = \mathcal{M}_{\tr}(A_1A \otimes BA_2C \otimes DA_3) \\
    & = \mathcal{M}_{\tr}((I_N \otimes C \otimes D)\boldsymbol{\cdot}(A_1 \otimes A_2 \otimes A_3) \boldsymbol{\cdot} (A \otimes B \otimes I_N)) \\
    & = \mathcal{M}_{\tr}((I_N \otimes V)\boldsymbol{\cdot} W \boldsymbol{\cdot} (U \otimes I_N))
\end{align*}
by Lemma \ref{lem.magic} and the definitions of $\mathcal{M}_{\tr}$ and the $\boldsymbol{\cdot}$ operation.
\end{proof}

\begin{thma}[Quadratic Covariation of Matrix It\^{o} Processes]\label{thm.IMPR}
If, for each $\ell \in \{1,2\}$, $M_{\ell}$ is a $N \times N$ matrix It\^{o} process satisfying $dM_{\ell}(t) = \sum_{i=1}^n U_{\ell i}(t) \sh dX_i(t) + K_{\ell}(t)\,dt$ and $W = (W(t))_{t \geq 0}$ is a continuous $\MnC^{\otimes 3}$-valued process, then
\[
\int_0^t W(s)\sh_2[dM_1(s),dM_2(s)] = \sum_{i=1}^n\int_0^t  \mathcal{M}_{\tr}((I_N \otimes U_{2i}(s))\boldsymbol{\cdot} W(s) \boldsymbol{\cdot} (U_{1i}(s) \otimes I_N))\,ds
\]
for all $t \geq 0$, almost surely.
\end{thma}
\begin{proof}
Recall that bounded variation terms do not contribute to quadratic covariation, so we may assume $K_1 \equiv K_2 \equiv 0$.
Now, using the expression \eqref{eq.repofXj} for $X_i$ and the fact that $db_{i,E}(t)\,db_{j,F}(t) = \delta_{ij}\delta_{E,F}\,dt$, we get
\begin{align*}
    \int_0^t W(s)\sh_2[dM_1(s),dM_2(s)] & = \sum_{i,j=1}^n\sum_{E,F\in \mathcal{E}} \int_0^t W(s)\sh_2[U_{1i}(s)\sh E, U_{2j}(s)\sh F]\,db_{i,E}(s)\,db_{j,F}(s) \\
    & = \sum_{i = 1}^n\sum_{E \in \mathcal{E}} \int_0^t W(s)\sh_2[U_{1i}(s)\sh E, U_{2i}(s)\sh E]\,ds \\
    & = \sum_{i = 1}^n\int_0^t\Bigg(\sum_{E\in \mathcal{E}}W(s)\sh_2[U_{1i}(s)\sh E,U_{2i}(s)\sh E]\Bigg)\,ds \\
    & = \sum_{i=1}^n\int_0^t  \mathcal{M}_{\tr}((I_N \otimes U_{2i}(s))\boldsymbol{\cdot} W(s) \boldsymbol{\cdot} (U_{1i}(s) \otimes I_N))\,ds
\end{align*}
by Lemma \ref{lem.keyFIFcalc}.
\end{proof}

From the cases $W = A \otimes B \otimes C$, $M_1 \in \{X_i,I_N\}$, and $M_2 \in \{X_j,I_N\}$, we get Equations \eqref{eq.dXidXj1}--\eqref{eq.dXidXj2}.
Let us now see how Theorem \ref{thm.IMPR} gives rise to a ``functional" It\^{o} formula for $C^2$ scalar functions of Hermitian matrix It\^{o} processes.

\begin{notaa}[Noncommutative Derivatives]
For $f \in C^k(\R)$, write $f^{[k]} \in C(\R^{k+1})$ for the $k^{\text{th}}$ divided difference of $f$.
(Please see Definition \ref{def.divdiff} and Proposition \ref{prop.divdiff}.)
If $M \in \MnC_{\sa}$, then
\[
\partial^kf(M) \coloneqq k!\sum_{\blambda \in \sigma(M)^{k+1}} f^{[k]}(\blambda)\,P_{\lambda_1}^M \otimes \cdots \otimes P_{\lambda_{k+1}}^M \in \MnC^{\otimes(k+1)},
\]
where $\blambda = (\lambda_1,\ldots,\lambda_{k+1})$ above. We shall view $\partial f(M) \coloneqq \partial^1f(M)$ as an element of $\MnC \otimes \MnC^{\op}$.
\end{notaa}

Here is the key fact.
Please see Appendix A of \cite{nikitopoulosNCk} for a self-contained proof.

\begin{thma}[Dalteskii--Krein \cite{daletskiikrein}, Hiai \cite{hiai}]\label{thm.matfunccalcder}
If $f \in C^k(\R)$, then $f_{\MnC} \in C^k(\MnC_{\sa};\MnC)$ and
\begin{align*}
    D^kf_{\MnC}(A)[B_1,\ldots,B_k] & = \frac{1}{k!}\sum_{\pi \in S_k}\partial^kf(A)\sh_k [B_{\pi(1)},\ldots, B_{\pi(k)}] \\
    & = \sum_{\pi \in S_k}\sum_{\blambda \in \sigma(A)^{k+1}}f^{[k]}(\blambda) \, P_{\lambda_1}^AB_{\pi(1)}\cdots P_{\lambda_k}^AB_{\pi(k)}P_{\lambda_{k+1}}^A, \numberthis\label{eq.matfunccalcderncder}
\end{align*}
for all $A,B_1,\ldots,B_k \in \MnC_{\sa}$, where $D^k$ is the $k^{\text{th}}$ Fr\'{e}chet derivative and $S_k$ is the symmetric group on $k$ letters.
\end{thma}

We are now ready to state and prove the (matrix) functional It\^{o} formula that motivates our functional free It\^{o} formula (Theorem \ref{thm.FFIF}).

\begin{notaa}
If $f \in C^2(\R)$ and $U \in \MnC \otimes \MnC^{\op}$, then we define
\[
\Delta_Uf(M) \coloneqq \mathcal{M}_{\tr}((I_N \otimes U)\boldsymbol{\cdot}\partial^2f(M)\boldsymbol{\cdot}(U \otimes I_N)) \in \MnC,
\]
where $\boldsymbol{\cdot}$ is multiplication in $\MnC \otimes \MnC^{\op} \otimes \MnC$ as usual.
\end{notaa}

\begin{thma}[Functional It\^{o} Formula]\label{thm.FIF}
Let $M$ be a $N \times N$ matrix It\^{o} process satisfying Equation \eqref{eq.matItoprocess}, and suppose $M^*=M$.
If $f \in C^2(\R)$, then
\[
d\,f(M(t)) = \partial f(M(t))\sh dM(t) + \frac{1}{2}\sum_{i=1}^n\Delta_{U_i(t)}f(M(t))\,dt. \numberthis\label{eq.FIF}
\]
\end{thma}
\begin{proof}
If $f \in C^2(\R)$, then $f_{\MnC} \in C^2(\MnC_{\sa};\MnC)$, so we may apply It\^{o}'s formula (Equation \eqref{eq.Itoform}) with $F = f_{\MnC}$.
Doing so gives
\begin{align*}
    d\,f(M(t)) & = d\,f_{\MnC}(M(t)) \\
    & = Df_{\MnC}(M(t))[dM(t)] + \frac{1}{2}D^2f_{\MnC}(M(t))[dM(t),dM(t)] \\
    & = \partial f(M(t)) \sh dM(t) + \frac{1}{2}\partial^2f(M(t)) \sh_2[dM(t),dM(t)] \\
    & = \partial f(M(t)) \sh dM(t) + \frac{1}{2}\sum_{i=1}^n\mathcal{M}_{\tr}((I_N \otimes U_i(t))\boldsymbol{\cdot}\partial^2f(M(t))\boldsymbol{\cdot}(U_i(t) \otimes I_N))\,dt \\
    & = \partial f(M(t)) \sh dM(t) + \frac{1}{2}\sum_{i=1}^n \Delta_{U_i(t)}f(M(t))\,dt
\end{align*}
by Theorem \ref{thm.matfunccalcder}, Theorem \ref{thm.IMPR}, and the definition of $\Delta_Uf(M)$.
\end{proof}

Applying $\tr = \frac{1}{N}\Tr$ to Equation \eqref{eq.FIF} and using symmetrization arguments similar to those from the proof of Lemma \ref{lem.trid} yields the following ``traced" formula that motivates Theorem \ref{thm.trFFIF}.
We leave the details to the interested reader.
In the statement below, if $U = \sum_{i=1}^k A_i \otimes B_i \in \MnC \otimes \MnC^{\op}$, then $U^{\flip} \coloneqq \sum_{i=1}^k B_i \otimes A_i \in \MnC \otimes \MnC^{\op}$.
Also, we write $\tr^{\op}$ for $\tr$ considered as a function $\MnC^{\op} \to \C$.

\begin{cora}[Traced Functional It\^{o} Formula]\label{cor.trFIF}
Let $M$ be a $N \times N$ matrix It\^{o} process satisfying Equation \eqref{eq.matItoprocess}, and suppose that $M^*=M$.
If $f \in C^2(\R)$, then
\[
d\,\tr(f(M(t))) = \tr\big(f'(M(t))\,dM(t)\big) + \frac{1}{2}\sum_{i=1}^n(\tr \otimes \tr^{\op})\big(U_i^{\flip}(t)\,\partial f'(M(t))\,U_i(t)\big)\,dt,
\]
where $U_i^{\flip}\partial f'(M)\,U_i$ is a product in the algebra $\MnC \otimes \MnC^{\op}$.
Under sufficient additional boundedness/integrability conditions (e.g., $U_i$, $K$, and $M$ are all uniformly bounded), we also have
\[
d\,\tau_N(f(M(t))) = \Big(\tau_N\big(f'(M(t))\,K(t)\big) + \frac{1}{2}\sum_{i=1}^n(\tau_N \otimes \tau_N^{\op})\big(U_i^{\flip}(t)\,\partial f'(M(t))\,U_i(t)\big)\Big)\,dt
\]
where $\tau_N = \mathbb{E} \circ \tr$ and $\tau_N^{\op} = \mathbb{E} \circ \tr^{\op}$.
\end{cora}

\begin{ack}
\phantomsection
\addcontentsline{toc}{section}{Acknowledgements}
I acknowledge support from NSF grant DGE 2038238 and partial support from NSF grants DMS 1253402 and DMS 1800733.
I am grateful to Bruce Driver, Adrian Ioana, and Todd Kemp for many helpful conversations.
Bruce Driver's extensive feedback on earlier versions of the paper was particularly helpful.
I would also like to thank David Jekel for bringing to my attention his parallel work on $C_{\mathrm{nc}}^k(\R)$ in Section 18 of his dissertation, \cite{jekel}.
Finally, I thank the anonymous referee, whose careful review led to improvements of both the present paper and my paper on noncommutative $C^k$ functions \cite{nikitopoulosNCk}.
\end{ack}

\end{document}